\documentclass{amsart}

\usepackage{latexsym}
\usepackage{amsmath,amsfonts,amssymb,mathrsfs}
\usepackage{amscd,multicol}
\usepackage[a4paper=true,pagebackref=true]{hyperref}
\renewcommand*{\backref}[1]{}
\renewcommand*{\backrefalt}[4]{{\tiny[%
  \ifcase #1 Not cited.\relax\or Page~#2.%
  \else Pages #2.\fi]}}
\hypersetup{pdftitle={Decomposition numbers for Hecke algebras of type $G(r,p,n)$},
  pdfauthor={Jun Hu and Andrew Mathas},
  colorlinks = true,
  linkcolor =blue,
  anchorcolor = red,
  citecolor = blue,
  pagecolor = red,
  urlcolor = blue
}

\usepackage{cite}

\def\simb{\sim_\b}
\def\ssim{\sim_\sigma}
\newcommand\Comp[1][p,n]{\mathscr{C}_{#1}}
\newcommand\Compp[1][p,n]{\mathscr{C}^\sigma_{#1}}
\newcommand{\Klesh}[1][r,n]{\mathscr{K}_{#1}}
\newcommand{\Kleshh}[1][r,n]{\mathscr{K}^\sigma_{#1}}
\newcommand{\Kleshb}[1][r,n]{\mathscr{K}^{\b}_{#1}}
\newcommand\Part[1][r,n]{\mathscr{P}_{#1}}
\newcommand\Partt[1][r,n]{\mathscr{P}^\sigma_{#1}}
\newcommand\Partb[1][r,n]{\mathscr{P}^{\b}_{#1}}

\def\Schsigma{\hat\sigma}

\def\bark{{\overline{K}}}

\newcommand\sscal[1][\relax]{\dot{\mathfrak{s}}_\blam^{#1}}
\newcommand\sscalt[1][\blam^{[t]}]{\dot{\mathfrak{s}}_{#1}}
\newcommand\dfscal[1][\relax]{
\ifx#1\relax\dot{\mathfrak{f}}_\blam\else\dot{\mathfrak{f}}^{(#1)}_\blam\fi
}
\newcommand\fscal[1][\relax]{
   \ifx#1\relax\mathfrak{f}_\blam\else\mathfrak{f}^{(#1)}_\blam\fi
}

\def\DoteqQ{\dot\eps,\dot q,\dot\bQ}
\newcommand{\gscal}[1][\blam]{\mathfrak{g}_{#1}}
\newcommand{\dgscal}[1][\blam]{\dot{\mathfrak{g}}_{#1}}

\let\<=\langle
\let\>=\rangle
\def\o{\mathsf{o}}
\newcommand\proD[2][2]{\displaystyle\kern-#1mm\prod_{#2}\kern-#1mm}
\def\Sum{\displaystyle\sum}

\def\th{^{\text{th}}}
\let\phi=\varphi

\def\bijection{\overset{\simeq}{\longrightarrow}}

\let\Sect=\S
\def\AKBasis{\mathfrak{B}}
\newcommand{\BS}{\mathfrak S}
\let\bar=\overline
\let\Sym=\BS

\def\Stab{\mathsf{S}}
\def\Ttab{\mathsf{T}}
\def\S{\mathbb S}
\newcommand{\N}{\mathbb N}
\newcommand{\Z}{\mathbb Z}
\newcommand\LH{\mathcal{H}^L}
\def\E{\mathscr{E}}

\newcommand{\Sch}[1][r,n]{{\mathscr{S}_{#1}}}
\newcommand{\Schrpn}[1][p]{\Sch[r,#1,n]}
\renewcommand{\H}{{\mathscr{H}}}
\newcommand{\Hrn}[1][n]{\H_{r,#1}}
\newcommand{\HrnK}[1][n]{{\H^\bark_{r,#1}}}
\newcommand{\Hrpn}[1][p]{{\H_{r,#1,n}}}
\newcommand{\HrpnK}[1][p]{{\H^\bark_{r,#1,n}}}
\newcommand{\HF}[1][r,n]{\H^\mF_{#1}}
\let\H=\H
\def\hatHdb{\widehat{\H}_{d,\b}}
\newcommand{\C}{\mathbb C}
\newcommand{\LL}[1][\relax]%
  {\ifx#1\relax\mathcal L\else\mathcal L^{(#1)}\fi}

\newcommand{\mA}{{\mathcal A}}
\newcommand{\mF}{{\mathcal F}}
\renewcommand\O{{\mathcal{O}}}
\newcommand{\Q}{\mathbb Q}
\newcommand{\LTheta}{\widehat\Theta}
\newcommand{\RTheta}{\Theta}
\newcommand{\lam}{\lambda}
\newcommand{\ft}{\mathfrak t}
\def\philam{{\phi_{\Ttab^\blam}}}
\def\tlam{ {\ft^\blam} }

\newcommand{\eps}{\varepsilon}

\DeclareMathOperator\Char{char}
\DeclareMathOperator\Tr{Tr}
\DeclareMathOperator\tr{tr}

\newcommand\bbeta{{\boldsymbol\beta}}
\newcommand\blam{{\boldsymbol\lambda}}
\newcommand\bgam{{\boldsymbol\gamma}}
\newcommand\bmu{{\boldsymbol\mu}}
\newcommand\bnu{{\boldsymbol\nu}}
\newcommand\btau{{\boldsymbol\tau}}
\newcommand\bomega[1][\b]{{\boldsymbol\omega_{#1}}}
\def\phib{\phi_{\bomega}}
 
\DeclareMathOperator\Std{Std}
\def\SStd(#1,#2){\mathcal{T}_0(#1,#2)}
\newcommand\s{\mathfrak{s}}
\renewcommand\t{\mathfrak{t}}

\renewcommand\v{\mathfrak{v}}
\def\HFun{\mathtt{H}}
\def\SFunc{\mathtt{F}_{\bomega}}
\def\Func{\mathtt{F}}
\def\EFunc{\Func_\E}

\def\shift#1{{\langle}#1{\rangle}}
\def\Morita{\xrightarrow[\text{\rm\,  Morita }]{\simeq}}
\def\Groth{\mathcal{R}}

\let\gedom=\trianglerighteq
\let\gdom=\vartriangleright

\def\ind_#1^#2(#3){{#3}\Ind_{#1}^{#2}}
\DeclareMathOperator{\Ind}{{\uparrow}}
\DeclareMathOperator{\Res}{{\downarrow}}

\DeclareMathOperator{\head}{Head}
\def\cont_#1(#2){\mathrm{cont}_{#1}(#2)}

\DeclareMathOperator{\rad}{rad}
\DeclareMathOperator{\soc}{Soc}

\DeclareMathOperator{\id}{id}
\DeclareMathOperator{\Hom}{Hom}
\DeclareMathOperator{\Irr}{Irr}
\DeclareMathOperator{\End}{End}
\DeclareMathOperator{\cha}{ch}
\DeclareMathOperator{\chat}{\text{ch}_t^l}
\DeclareMathOperator{\Mod}{Mod-\!}

\newcommand{\bbQ}{\mathbf{Q}^{\boldsymbol{\vee\eps}}}
\newcommand{\ibbQ}{\check{\mathbf{Q}}^{\boldsymbol{\vee\eps^m}}}
\newcommand{\bQ}{\mathbf{Q}}
\renewcommand\b{\mathbf{b}}
\newcommand\bp{\mathbf{b'}}
\newcommand{\ba}{\mathbf{a}}
\renewcommand{\c}{\mathbf{c}}

\newcounter{main}
\newtheorem{THEOREM}[main]{Theorem}
\newtheorem*{COR}{Corollary}
\numberwithin{equation}{section}
\newtheorem{prop}[equation]{Proposition}
\newtheorem{Theorem}[equation]{Theorem}
\newtheorem{cor}[equation]{Corollary}
\newtheorem{Lemma}[equation]{Lemma}
\newtheorem{Defn}[equation]{Definition}
\theoremstyle{remark}
\newtheorem{Remark}[equation]{Remark}

\theoremstyle{definition}

\def\map#1#2{\,{:}\,#1\!\longrightarrow\!#2}

{\catcode`\|=\active
  \gdef\set#1{\mathinner{\lbrace\,{\mathcode`\|"8000%
                                   \let|\midvert #1}\,\rbrace}}
}
\def\midvert{\egroup\mid\bgroup}

\title[Decomposition numbers for Hecke algebras of type $G(r,p,n)$]%
   {Decomposition numbers for Hecke algebras of type $G(r,p,n)$:
   the $(\eps,q)$-separated case}
\author{Jun Hu}
\thanks{Both authors were supported, in part, by the University of Sydney and the
Australian Research Council. The first author was also supported by the National
Natural Science Foundation of China.
}
\address{
Department of Mathematics, Beijing Institute of Technology,
Beijing, 100081, China}
\email{junhu303@yahoo.com.cn}
\author{Andrew Mathas}
\address{School of Mathematics and Statistics F07,
University of Sydney, NSW 2006, Australia}
\email{a.mathas@usyd.edu.au}

\begin{document}
\maketitle

\begin{abstract}
  The paper studies the modular representation theory of the cyclotomic Hecke
  algebras of type $G(r,p,n)$ with $(\eps,q)$-separated parameters. We show that
  the decomposition numbers of these algebras are completely determined by the
  decomposition matrices of related cyclotomic Hecke algebras of type
  $G(s,1,m)$, where $1\le s\le r$ and $1\le m\le n$. Furthermore, the proof
  gives an explicit algorithm for computing these decomposition numbers.
  Consequently, in principle, the decomposition matrices of these algebras are
  now known in characteristic zero.

  In proving these results, we develop a Specht module theory for
  these algebras, explicitly construct their simple modules and
  introduce and study analogues of the cyclotomic Schur algebras of
  type $G(r,p,n)$ when the parameters are $(\eps,q)$-separated.

  The main results of the paper rest upon two Morita equivalences: the
  first reduces the calculation of all decomposition numbers to the
  case of the \textit{$l$-splittable decomposition numbers} and the
  second Morita equivalence allows us to compute these decomposition
  numbers using an analogue of the cyclotomic Schur algebras for the
  Hecke algebras of type~$G(r,p,n)$.
\end{abstract}

\maketitle

\setcounter{tocdepth}{1}\tableofcontents

\section{Introduction}

The cyclotomic Hecke algebras~\cite{BM:cyc} are an important class of
algebras which arise in the representation theory of  finite reductive
groups. These algebras can be defined using generators and relations and
they are deformations of the group algebras of the complex reflection
groups. The cyclotomic Hecke algebras can also be constructed using the
monodromy representation of the associated braid groups~\cite{BMR} and, in
characteristic zero, they are closely connected with category~$\mathcal O$
for the rational Cherednik algebras by the Knizhnik-Zamolodchikov
functor~\cite{GGOR}.

This paper is concerned with the representation theory of the cyclotomic
Hecke algebras $\Hrpn$ of type $G(r,p,n)$, where $r=pd$, $p>1$ and $n\geq
3$. Throughout we work over a field~$K$ which contains a primitive $p$th
root of unity~$\eps$. The algebra $\Hrpn$ depends upon the
parameters $q\in K$ and $\bQ=(Q_1,\dots,Q_d)\in K^d$ (see
Definition~\ref{D:lSplittable}).  The $d$-tuple of parameters $\bQ$ is
\textbf{$(\eps,q)$-separated} over~$K$ if
\begin{equation}\label{E:separated}
\prod_{1\le i,j\le d}\prod_{-n<k<n}\prod_{1\leq t<p}
       \big(Q_i-\eps^{t}q^{k}Q_j\big)\ne0.
\end{equation}
As we explain in Lemma~\ref{L:separated}, $(\eps,q)$-separation is almost the
same  as assuming that $\<\eps\>\cap\<q\>=1$. This can be viewed as the quantum
analogue of the common assumption in Clifford theory that the characteristic of
the field should not divide the index of the normal subgroup inside the parent
group. In general, the algebra~$\Hrpn$ is not semisimple when~$\bQ$ is
$(\eps,q$)-separated.

The main result of this paper is the following.

\begin{THEOREM}\label{main}
  Suppose that $K$ is a field of characteristic zero and that $\bQ$ is
  $(\eps,q)$-separated over~$K$. Then the decomposition matrix of $\Hrpn$
  is determined by the decomposition matrices of the cyclotomic Hecke
  algebras of type $G(s,1,m)$, where $1\le s\le r$ and $1\le m\le n$.
\end{THEOREM}

In proving this result we also obtain an analogous but slightly weaker result for
the decomposition numbers of~$\Hrpn$ in positive characteristic. Moreover, when
combined with the results of \cite{HuMathas:Morita},
Theorem~\ref{main} gives an explicit algorithm for computing the decomposition
numbers of $\Hrpn$ in terms of the decomposition matrices of related Hecke
algebras of type $G(s,1,m)$. Ariki~\cite{Ariki:can} has
determined the decomposition numbers of the Hecke algebras $\Hrn=\H_{r,1,n}$ of
type $G(r,1,n)$ when he, famously, proved and generalised the LLT conjecture.
Hence, combining~\cite{Ariki:can} and Theorem~\ref{main} implies the following.

\begin{COR}
  Suppose that $K$ is a field of characteristic zero and that $\bQ$ is
  $(\eps,q)$-separated over~$K$. Then the decomposition matrix of $\Hrpn$
  is, in principle, known.
\end{COR}

We note that Theorem~\ref{main} and its corollary have been obtained by the
first author in the special case of the Hecke algebras of type $D$, when
$r=p=2$~\cite{Hu:DecompDEven}. This paper is a (non-trivial) generalization of
the results in~\cite{Hu:DecompDEven} to the algebras~$\Hrpn$.

To prove Theorem~\ref{main} it is enough by~\cite[Theorem~B]{HuMathas:Morita}
(see Theorem~\ref{T:SplittableRedction}), to compute the
\textbf{$l$--splittable} decomposition numbers of the Hecke algebras of type
$G(r,p,n)$. As is usual in Clifford theory, a decomposition number $[S:D]$ for
$\Hrpn$ is \textbf{$p$-splittable} if $S$ and $D$ both have trivial
\textit{inertia groups}; see Definition~\ref{lsplittable} for a purely
combinatorial definition.

In Theorem~\ref{splittable} below we give a closed formula for
all of the $l$-splittable decomposition numbers of~$\Hrpn$. This formula depends
on the decomposition numbers of certain Hecke algebras~$\H_{s,m}=\H_{s,1,m}$,
where $s\le r$ and $m\le n$, and some scalars $\gscal\in K$ which come from the
semisimple representation theory of $\H_{s,m}$. More precisely, $\gscal$ is an
$l$th root of a quotient of two~\textbf{Schur elements}. The scalars~$\gscal$
enter the picture because they can be used to decompose the Specht modules of
$\H_{r,n}$ into a direct sum of~$\Hrpn$-modules.

All of the results in this paper are geared towards computing the $l$-splittable
decomposition numbers of $\Hrpn$. This requires a considerable amount of
preliminary work, much of which takes place inside the algebra~$\Hrn$. This
story begins with the Morita equivalence theorem of Dipper and the second
author~\cite{DM:Morita} which shows, that modulo some technical
assumptions on~$\bQ$, that there is a Morita equivalence
\begin{equation}\label{E:Morita}
  \Mod\Hrn\Morita\bigoplus_{\b\in\Comp}\Mod\H_{d,\b},
\end{equation}
where $\Comp$ is the set of compositions of $n$ into $p$ parts and if
$\b=(b_1,\dots,b_p)\in\Comp$ then
$\H_{d,\b}=\H_{d,b_1}\otimes\dots\otimes\H_{d,b_p}$. This result is proved
by constructing an explicit $(\H_{d,\b},\Hrn)$-bimodule $V_\b=v_\b\Hrn$
(Definition~\ref{vb defn}), and showing that $V_\b$ is projective as an
$\Hrn$-module and that $\H_{d,\b}\cong\End_{\Hrn}(V_\b)$.

In this paper we use the Morita equivalence~\eqref{E:Morita} to understand how
the Specht modules of~$\Hrn$ behave under restriction to~$\Hrpn$. One of the key
results is Theorem~\ref{subalgebra} which shows that there is an invertible
central element $z_\b$ in $\H_{d,\b}$ such that $e_\b=z_\b^{-1}\cdot v_\b T_\b$
is the idempotent in $\H_{r,n}$ which generates~$V_\b$, where $T_\b=T_{w_\b}$
for a certain permutation $w_\b\in\Sym_n$. As a byproduct we construct a
\textit{parabolic subalgebra} of $\Hrn$ which is isomorphic to $\H_{d,\b}$ and
we show that the Morita equivalence \eqref{E:Morita} corresponds to induction
from these subalgebras.

The first aim of this paper is to show that $z_\b$ acts as multiplication by an
invertible scalar $\fscal$ on certain Specht modules of~$\Hrn$. In order to describe
these results, and how they help prove Theorem~\ref{main}, we need some more
notation. Recall from \cite{DJM:cyc} that~$\Hrn$ is a cellular algebra with cell
modules, the \textbf{Specht modules} $S(\blam)$, indexed by the
$r$-multipartitions $\blam=(\lambda^{(1)},\lambda^{(2)},\dots,\lambda^{(r)})$ of~$n$.
If $\Hrn$ is semisimple then the Specht modules are a complete set of pairwise
non-isomorphic irreducible $\Hrn$-modules. More generally, define
$D(\blam)=S(\blam)/\rad S(\blam)$, where $\rad S(\blam)$ is the radical of the
bilinear form on $S(\blam)$. Then the non-zero $D(\blam)$ are a complete set of
pairwise non-isomorphic $\Hrn$-modules.

For each $\blam\in\Part$, we write
$\blam=(\blam^{[1]},\cdots,\blam^{[p]})$, where
\begin{equation}\label{E:blamt}
  \blam^{[t]}=(\blam^{(dt-d+1)},\blam^{(dt-d+2)},\dots,\blam^{(dt)}),\quad
  \text{for $1\le t\le p$}.
\end{equation}
For convenience, set $\blam^{[t+kp]}=\blam^{[t]}$, for all $k\in\Z$.
Let $\b=(b_1,\dots,b_p)\in\Comp$ and set
$\Part[d,\b]=\set{\blam\in\Part||\blam^{[t]}|=b_t\text{ for }1\le t\le p}$.
Then, by \cite{DJM:cyc}, the algebra $\H_{d,\b}$ is a
cellular algebra with cell modules
$S_\b(\blam)\cong S(\blam^{[1]})\otimes\dots\otimes S(\blam^{[p]}),$
for $\blam\in\Part[d,\b]$. Again, the modules
$D_\b(\blam)=S_\b(\blam)/\rad S_\b(\blam)$, for $\blam\in\Part[d,\b]$, are
either absolutely irreducible or zero.

Let $\mF=\Q(\dot\eps,\dot q,\dot\bQ)$, where $\dot\eps\in\mathbb{C}$
is a fixed primitive $p$th root of unity in ~$\mathbb C$ and~$\dot q$
and~$\dot\bQ$ are indeterminates. The cyclotomic Hecke algebras
$\HF$ and $\H^\mF_{d,\b}$ over~$\mF$
are semisimple and they 
come equipped with non-degenerate trace forms~$\Tr$ and~$\Tr_\b$,
respectively. Define the \textbf{Schur elements} $\sscal$ and
$\sscal[\b]$ of $\HF$ and $\H^\mF_{d,\b}$, respectively, are the scalars in~$\mF$ 
determined by 
\begin{equation}\label{E:SchurElements}
\Tr=\sum_{\blam\in\Part}\frac1{\sscal}\chi^\blam
\quad\text{and}\quad
  \Tr_\b=\sum_{\blam\in\Part[d,\b]}\frac1{\sscal[\b]}\chi_\b^\blam,
\end{equation}
where $\chi^\blam$ and $\chi^\blam_\b$ are the characters of the irreducible
Specht modules $S(\blam)$ and~$S_\b(\blam)$, respectively.

The Schur elements $\sscal$ and $\sscal[\b]$ are explicitly
known~\cite{M:gendeg} and, as we now explain, they are closely related to the scalars
$\fscal$ which give the action of $z_\b$ on the Specht modules of $\Hrn$. To
state this result define
$\o_\blam=\min\set{k\ge1|\blam^{[k+t]}=\blam^{[t]}, \text{ for all }t\in\Z}$,
and set $p_\blam=p/\o_\blam$. Note that~$\o_\blam$ divides $p$ so that $p_\blam$
is an integer, for all $\blam\in\Part[d,\b]$.


\begin{THEOREM}\label{scalars}
  Suppose that $\bQ$ is $(\eps,q)$-separated over $K$ and that
  $\blam\in\Part[d,\b]$. Then  there exists a non-zero scalar
  $\fscal\in K$ such that $z_\b \cdot v = \fscal v$, for all $v\in S(\blam)$.
  Moreover,
  \[\fscal=(\mathfrak{s}_{\blam}/\mathfrak{s}_{\blam}^{\b})\Tr(v_\b T_\b)
     =\eps^{\frac12d\o_\blam n(1-p_\blam)}\gscal^{p_\blam},\]
  where $\gscal\in K$ and where
  $(\mathfrak{s}_{\blam}/\mathfrak{s}_{\blam}^{\b})(\eps,q,\bQ)
             =({\sscal}/{\sscal[\b]})(\eps,q,\bQ)$ is the
  specialization of the rational function ${\sscal}/{\sscal[\b]}$ at
  $(\DoteqQ)=(\eps,q,\bQ)$ $($which is well-defined and non-zero$)$.
\end{THEOREM}

Roughly half of this paper is devoted to proving Theorem~\ref{scalars}, but the
payoff is considerable as the scalars $\fscal$ and $\gscal$ play a role in
everything that follows. The three main steps in its proof are
Theorem~\ref{subalgebra}, which explicitly relates the primitive idempotents
in~$\HF$ and $\H^\mF_{d,\b}$ under~\eqref{E:Morita}; Theorem~\ref{comparison},
which is a comparison theorem relating the trace functions $\Tr$ and $\Tr_\b$;
and, Theorem~\ref{root}, which uses \textit{shifting homomorphisms}, some
Clifford theory and seminormal forms to show that~$\fscal$ has a $p_\blam$th
root.

The reason why Theorem~\ref{scalars} is important is that multiplication
by~$z_\b$ induces an $\Hrn$-module endomorphism of~$S(\blam)$. We show that the
factorisation of~$\fscal$ given in Theorem~\ref{scalars} corresponds to a
factorisation of this endomorphism and hence that there exists a
$\Hrpn$-module endomorphism~$\theta_\blam$ of~$S(\blam)$ such that
$\theta_\blam^{p_\blam}$ is $\gscal^{p_\blam}$ times the identity map
on~$S(\blam)$ (Corollary~\ref{kernel}). This allows us to decompose the Specht module as
$S(\blam)=S^\blam_1\oplus\dots\oplus S^\blam_{p_\blam}$, where
\[S^\blam_t=\set{x\in S(\blam)|\theta_\blam(x)=\eps^{t\o_\blam}\gscal x}\]
is one of the $\theta_\blam$-eigenspace of $S(\blam)$, for $1\le t\le p_\blam$.
The module $S^\blam_t$ is an $\Hrpn$-module which is an
analogue of a Specht module for $\Hrpn$.

Next we want to construct the irreducible $\Hrpn$-modules. Let $D^\blam_t$ be
the head of $S^\blam_t$. We will show that $D^\blam_t$ is either irreducible or
zero. To describe the complete set of irreducible $\Hrpn$-modules let
$\Klesh=\{\blam\in\Part|D(\blam)\neq 0\}$ be the set of \textbf{Kleshchev
multipartitions} for~$\bbQ$. Then $\set{D(\blam)|D(\blam)\ne0}$ is a complete
set of pairwise non-isomorphic irreducible $\Hrn$-modules by~\cite{DJM:cyc}.
Define an equivalence relation $\ssim$ on $\Part$ by $\blam\ssim\bmu$ if there
exists a $k\in\Z$ such that $\blam^{[t]}=\bmu^{[t+k]}$, for $1\le t\le p$ and
$\blam,\bmu\in\Part$. If $\bQ$ is $(\eps,q)$-separated over~$K$, then~$\ssim$
induces an equivalence relation on $\Klesh$ (cf. Lemma \ref{Morita Specht}). Let
$\Partt$ and $\Kleshh$ be the sets of $\ssim$-equivalence classes in $\Part$ and
$\Klesh$, respectively.

\begin{THEOREM}\label{main simple}
  Suppose that $\bQ$ is $(\eps,q)$-separated over the field~$K$. Then:
  \begin{enumerate}
    \item $\set{D_t^{\bmu}|\bmu\in\Kleshh\text{ and } 1\leq t\leq p_\bmu}$
      is a complete set of pairwise non-isomorphic absolutely irreducible
      $\Hrpn$-modules. Hence, $K$ is a splitting field for~$\Hrpn$.
    \item The decomposition matrix of $\Hrpn$ is unitriangular.
  \end{enumerate}
\end{THEOREM}

The structure of the Specht modules~$S^\blam_t$ and the simple modules
$D^\blam_t$ is described in more detail in Theorem~\ref{thm54} and
Theorem~\ref{Hrpn simples}.

We now have the notation to define the most important $l$-splittable
decomposition numbers of~$\Hrpn$.

\begin{Defn} \label{lsplittable}
  Suppose that $l$ divides~$p$, $\blam,\bmu\in\Part[d,\b]$
  and that $1\le i\le p_\blam$ and $1\le j\le p_\bmu$. The
  decomposition number $[S^\blam_i:D^\bmu_j]$ is \textbf{$l$-splittable} if
  $p_\blam=l=p_\bmu$.
\end{Defn}

By the results in section~4, and the general theory developed in
\cite{HuMathas:Morita}, the decomposition number $[S^\blam_i:D^\bmu_j]$ is
$p$-splittable if and only if $S^\blam_j$ and $D^\bmu_j$ both have trivial
\textit{inertia groups} in the usual sense of Clifford theory.

Now suppose that $l$ divides $p$ and let $m=p/l$. To give an explicit formula
for the $l$-splittable decomposition numbers of $\Hrpn$ let $V(l)$ be the
$l\times l$ Vandermonde matrix
\[
V(l)=\begin{pmatrix} 1& 1& \dots & 1\\
         \eps^m  & \eps^{2m} & \dots & \eps^{lm}\\
         \vdots &\vdots &\vdots & \vdots\\
         \eps^{(l-1)m} & \eps^{2(l-1)m} & \dots & \eps^{l(l-1)m}
\end{pmatrix}.
\]
For $1\le i\le l$ define $V_i(l)$ to be the matrix obtained from $V(l)$
by replacing its $i$th column with the column vector
\[
\begin{pmatrix}
  d_{\blam_m\bmu_m}^{\,l}\\
  \big(\frac{\gscal}{\gscal[\bmu]}\big)^{1}d_{\blam_m\bmu_m}^{\,l_1}\\
  \vdots\\
  \big(\frac{\gscal}{\gscal[\bmu]}\big)^{l-1}d_{\blam_m\bmu_m}^{\,l_{l-1}}\\
\end{pmatrix}
\]
where $d_{\blam_m\bmu_m}=[S(\blam^{[1]},\dots,\blam^{[m]}):D(\bmu^{[1]},\dots,\bmu^{[m]})]$
and $l_t=\gcd(l,t)$.

\begin{THEOREM}\label{splittable}
  Suppose that $K$ is a field, that $\bQ$ is $(\eps,q)$-separated over $K$
  and that the decomposition number $[S^\blam_i:D^\bmu_j]$ is
  $l$-splittable, for some $l$ dividing~$p$. Then
  \[[S^\blam_i:D^\bmu_j]
        \equiv \frac{\det V_{j-i}(l)}{\det V(l)}\pmod{\Char K},\]
  for $1\le i,j\le l=p_\blam=p_\bmu$. In particular, the $l$-splittable
  decomposition numbers of~$\Hrpn$ are known when $K$ is a field of
  characteristic zero.
\end{THEOREM}

A closed formula for $\gscal$ is given in Proposition~\ref{P:gscal} and
Remark~\ref{R:gscal}, so Theorem~\ref{splittable} completely determines the
splittable decomposition numbers of~$\Hrpn$. 

The main idea underpinning Theorem~\ref{splittable} is the introduction of a new
algebra $\Schrpn$, which is an analogue of the cyclotomic Schur
algebra~\cite{DJM:cyc} for $\Hrpn$. We construct Weyl modules and simple modules
for $\Schrpn$ and then compute the $l$-splittable decomposition numbers
of~$\Schrpn$ using the \textbf{twining characters} of $\Schrpn$. The twining
characters, which generalize of formal characters, compute the trace of certain
elements $\vartheta_\blam\in\Schrpn$ on the weight spaces of some
$\Schrpn$-modules. The map $\vartheta_\blam$ is constructed from the
endomorphism $\theta_\blam$ mention earlier so, once again, $\vartheta_\blam$
comes from action of $z_\b$ upon certain $\Schrpn$-modules.  Finally,
Theorem~\ref{splittable} is proved using a considerable amount of Clifford
theory and some natural functors
\[\bigoplus_{\b\in\Compp}\Mod\Schrpn(\b)
     \xrightarrow{\oplus\SFunc^{(p)}}\bigoplus_{\b\in\Compp}\Mod\E_{d,\b}
     \Morita\Mod\Hrpn.\]
where the first functor is an analogue of the Schur functor and the second
functor is the restriction of the Morita equivalence of (\ref{E:Morita})
to~$\Hrpn$.

Very briefly, the outline of this paper is as follows. Chapter~2 studies
the right ideals $V_\b=v_\b\Hrn$. The main
results are Lemma~\ref{vb shift} which shows the existence of the central
element~$z_\b$, Theorem~\ref{subalgebra} which produces a subalgebra of
$\Hrn$ isomorphic to~$\H_{d,\b}$, and
Theorem~\ref{comparison} which is a comparison theorem for the natural
trace forms on $\H_{d,\b}$ and $\Hrn$. In Chapter~3 these results are used
to compute the scalars~$\fscal$, for $\blam\in\Part$, which describe the
action of $z_\b$ on the Specht modules $S(\blam)$ of $\Hrn$. This proves the
first half of Theorem~\ref{scalars}. Section~\ref{S:Twist} marks the first
direct appearance of the algebras $\Hrpn$. Using seminormal forms we factorize
the scalars~$\fscal$ in Theorem~\ref{root}, completing the proof of
Theorem~\ref{scalars}. We then use the roots of the scalars~$\gscal$ to
decompose the Specht modules as $\Hrpn$-modules, culminating in Theorem~\ref{thm54} and
Theorem~\ref{Hrpn simples} which describe the Specht modules and simple modules
of~$\Hrpn$, respectively. This completes the proof of Theorem~C. Chapter~4 begins
by lifting the Morita equivalence~(\ref{E:Morita}) to a new Morita equivalence
between~$\Hrpn$ and a new algebra $\E_d$ in Corollary~\ref{cor64}.
Section~\ref{S:Schur} introduces and studies the algebras $\Schrpn$, which are
analogues of the cyclotomic Schur algebras for~$\Hrpn$. Theorem~\ref{8mainthm3}
computes the $l$-splittable decomposition numbers of $\Schrpn$ using its twining
characters. Applying the functors mentioned in the last paragraph we then prove
Theorem~\ref{splittable} and hence complete the proof of Theorem~\ref{main}.
Finally, in the appendix we prove some technical results whose proofs were
deferred from Chapter~2.

\section*{Index of notation}
\def\notation#1#2{\rlap{#1}\hspace*{14mm} \hbox to 47mm{#2\hfill}}
{\small
\multicolsep=1mm
\columnsep=4mm
\begin{center}\parindent=0pt%
\begin{multicols}{2}\noindent%
  \notation{$\ssim$}{Equivalence relation $\b\sim\b\shift {k}$}
   \notation{$\sim_\b$}{Equivalence relation $\blam\sim\blam\shift{k\o_\b}$}
    \notation{$\Ind_B^A, \Res_B^A$}{Induction \& restriction functors}
    \notation{$A(\varepsilon,q,\bQ)$}%
              {$\prod_{i,j,|k|<n,1\leq t<p}(Q_i-\varepsilon^tq^kQ_j)$}
    \notation{$\mA$}{$\Z[\dot\eps,\dot q^{\pm1},\dot Q_1^{\pm 1},\dots,\dot
                      Q_d^{\pm1}, \frac1{A(\DoteqQ)}]$}
\notation{$\b_i^j$}{$\delta_{i\le j}(b_i+\dots+b_j)$}
    \notation{$\ba\vee\b$}{The concatenation of $\ba$ and $\b$}
    \notation{$\b\shift z$}{The shift of the sequence $\b$ by $z$}
    \notation{$\Comp$}{Compositions of $n$ of length $p$}
    \notation{$\cha M$}{$\sum_\bmu(\dim M_\bmu)e^\bmu$}
    \notation{$\chat M$}{$\sum_\gamma\Tr(\vartheta_\blam^t,M_{\gamma^{l_t}})e^\gamma$}
    \notation{$d$}{$r/p$}
    \notation{$D(\blam)$}{Simple module for $\Hrn$}
    \notation{$D_\b(\blam)$}{Simple module for $\H_{d,\b}$}
    \notation{$D^\blam_{i,p}$}{Simple module for $\E_{d,\b}$}
    \notation{$D^\blam_i$}{Simple module for $\Hrpn$}
    \notation{$\Delta(\blam)$}{Weyl module for $\Sch$}
    \notation{$\Delta_\b(\blam)$}{Weyl module for $\Sch[d,\b]$}
    \notation{$\Delta^\blam_{i,p}$}{Weyl module for $\Schrpn$}
    \notation{$\eps$}{Primitive $p$th root of unity in $K$}
    \notation{$\dot\eps$}{Primitive $p$th root of unity in $\C$}
    \notation{$e_\b$}{The idempotent $z_\b^{-1}\cdot v_\b T_{\b}$}
    \notation{$\E_{d,\b}$}{$=\End_{\Hrpn}(V_\b)$}
    \notation{$\mF$}{The field of fractions of $\mA$}
    \notation{$\fscal$}{The scalar: $z_\b\downarrow_{S(\blam)}=\fscal \id_{S(\blam)}$}
    \notation{$\fscal[t]$}{$=\fscal[t:m]$, a factor of $\fscal$}
    \notation{$\gscal$}{$\fscal=\eps^{\frac{1}{2}d\o_{\blam}n(1-p_\blam)}\gscal^{p_\blam}$}
    \notation{$\HFun_\b$}{A functor $\Mod\H_{d,\b}\to\Mod\Hrn$}
    \notation{$\theta'_t$}{The map $h\mapsto Y_t h$}
    \notation{$\theta'_{t,m}$}{The map $h\mapsto Y_{t,m}h$}
    \notation{$\theta_{t,m}$}{$=\sigma^m\circ\theta'_{t,m}$}
    \notation{$\theta_\b$}{$=\theta_{0,\o_p(\b)}$ restricted to $M^\blam_\b$}
    \notation{$\theta_\blam$}{$=\theta_{0,\o_p(\blam)}$ restricted to
    $S(\blam)$}
    \notation{$\vartheta_\b$}{$\theta_\b$ restricted to $M^\blam_\b$}
    \notation{$\vartheta_\blam$}{$\vartheta_\b^{p_{\b/\blam}}$}
    \notation{$\LTheta_\b,\RTheta_\b$}{Two linear maps $\H_{d,\b}\to\Hrn$}
    \notation{$\H^R_{r,n}$}{ Hecke algebra of type $G(r,1,n)$}
    \notation{$\H^R_{r,p,n}$}{ Hecke algebra of type $G(r,p,n)$}
    \notation{$\H_{d,\b}$}{$=\H_{d,b_1}(\eps\bQ)\otimes\dots\otimes\H_{d,b_p}(\eps^p\bQ)$}
    \notation{$\hatHdb$}{$=\H_{d,\b}\cdot e_\b\cong \H_{d,\b}$}
    \notation{$\Klesh$}{Kleshchev multipartitions in $\Part$}
    \notation{$\Klesh[d,\b]$}{Kleshchev multipartitions in $\Part[d,\b]$}
    \notation{$L(\blam)$}{Simple module for $\Sch$}
    \notation{$L_\b(\blam)$}{Simple module for $\Sch[d,\b]$}
    \notation{$L^\blam_{i,p}$}{Simple module for $\Schrpn$}
    \notation{$M^\blam_\b$}{Permutation module in $V_\b$}
    \notation{$\o_p(\b)$}{$=\min\set{z>0|\b\shift z=\b}$}
    \notation{$\o_\b$}{$=\o_p(\b)$}
    \notation{$\o_\blam$}{$=\o_p(\blam)$}
    \notation{$p_\blam$}{$p/\o_p(\blam)=p/\o_\blam$}
    \notation{$p_{\b/\blam}$}{$p_\b/p_\blam=\o_\blam/\o_\b$}
    \notation{$p_{\bmu/\blam}$}{$p_\bmu/p_\blam=\o_\blam/\o_\bmu$}
    \notation{$\Part$}{The set of $r$--multipartitions of $n$}
    \notation{$\Part[d,\b]$}
    {$\set{\blam\in\Part|b_i=|\blam^{[i]}|, 1\le i\le d}$}
    \notation{$\blam^{[i]}$}{$(\lambda^{(d(i-1)+1)},\dots,\lambda^{(di)})$}
    \notation{$\LL[s]_k$}{$\prod_{i=1}^d(L_k-\eps^s Q_i)$}
    \notation{$\LL[i,j]_{l,m}$}{$\prod_{l\le k\le m}\prod_{s\in I_{ij}}\LL[s]_k$}
    \notation{$\bQ$}{$(Q_1,\dots,Q_d)$}
    \notation{$\bbQ$}{$\eps\bQ\vee\eps^2\bQ\vee\dots\vee\eps^p\bQ$}
    \notation{$\sigma$}{$\Hrn\to\Hrn;T_i\mapsto\eps^{\delta_{i0}}T_i$}
    \notation{$\Schsigma$}{An automorphism of $\Schrpn$}
    \notation{$s_i$}{$(i,i+1)\in\Sym_n$}
    \notation{$S(\blam)$}{Specht module for $\Hrn$}
    \notation{$S_{\b}(\blam)$}{Specht module for $\H_{d,\b}$}
    \notation{$S^\blam_{i,p}$}{Specht module for $\E_{d,\b}$}
    \notation{$S^\blam_t$}{Specht module for $\Hrpn$}
    \notation{$\sscal$}{Schur element of $S(\blam)$}
    \notation{$\sscal[\b]$}{Schur element of $S_{\b}(\blam)$}
    \notation{$\Sch$}{Cyclotomic Schur algebra for $\Hrn$}
    \notation{$\Sch[d,\b]$}{Cyclotomic Schur algebra for $\H_{d,\b}$}
    \notation{$\Schrpn$}{Cyclotomic Schur algebra for $\Hrpn$}
    \notation{$\Std(\blam)$}{Standard $\blam$-tableaux}
    \notation{$T_{a,b}$}{$T_{w_{a,b}}$}
    \notation{$T_{a,b}^{\shift k}$}{$T_w$, where $w=w_{a,b}^{\shift k}$}
    \notation{$\tau$}{$h\mapsto T_0^{-1}hT_0$}
    \notation{$v_\b$}{$v_\b(\bQ)=v_\b^+ u_\b^+=u_\b^- v_\b^-$}
    \notation{$v_\b^{(t)}$}{$v_\b(\eps^t\bQ)$}
    \notation{$V_\b$}{The ideal $v_\b\Hrn$}
    \notation{$V_\b^{(t)}$}{$=v_\b^{(t)}\Hrn$}
    \notation{$\{v^{(tm)}_{\s}\}$}{Seminormal basis of $S^\mF(\blam)^{(tm)}$}
    \notation{$w_{a,b}^{\shift k}$}{$(s_{a+b+k-1}\dots s_{k+1})^b$}
    \notation{$w_\b$}{$w^{\shift{\b_1^{p-2}}}_{b_{p-1},\b_{p}^p}\dots
                w^{\shift{\b_1^1}}_{b_2,\b_3^p} w_{b_1,\b_2^p}$}
    \notation{$Y_t$}{$\LL[t+1,t+p-1]_{1,b_t}T_{b_t,n-b_t}$}
    \notation{$Y_{t,m}$}{$Y_{tm}Y_{tm-1}\dots Y_{t(m-1)+1}$}
    \notation{$z_\b$}{Central element of $\H_{d,\b}$}
\end{multicols}
\end{center}
}

\section{Hecke algebras of type $G(r,1,n)$ and the central elements $z_\b$}
  The main objects studied in this paper are the Hecke algebras of type $G(r,p,n)$, where $r=pd$
  for some $d\in\N$. These algebras are deformations of the group rings of the
  corresponding complex reflection groups of type $G(r,p,n)$.

  The complex reflection groups of type $G(r,1,n)$ are the groups
  $(Z/r\Z)\wr\Sym_n$, where~$\Sym_n$ is the symmetric group on
  $\{1,2,\dots,n\}$. The group of type $G(pd,p,n)$ is a normal subgroup
  of~$(\Z/pd\Z)\wr\Sym_n$ of index~$p$ which is fixed by an automorphism of
  $G(pd,1,n)$. Similarly, if $n\ge3$ then the Hecke algebra of type $G(pd,p,n)$
  can be defined as the fixed point subalgebra of an automorphism of the Hecke
  algebra of type $G(pd,1,n)$.

  In this chapter we define the Hecke algebras of types $G(r,1,n)$ and
  $G(r,p,n)$ and begin to set up the machinery that we need in order to prove
  Theorem~\ref{scalars}. The highlights of this chapter are Lemma~\ref{vb
  shift}, which proves the existence of the central elements~$z_\b$ from
  Theorem~\ref{scalars}, and Theorems~\ref{subalgebra} and ~\ref{comparison}
  which, in Chapter~3, will allow us to compute the scalars~$\fscal$
  and~$\gscal$ in Theorem~\ref{scalars}.

  \subsection{Cyclotomic Hecke algebras}\label{S:Hrpn}
  We begin by defining the Hecke algebras $\Hrn$ and $\Hrpn$ from the
  introduction and recalling the machinery that we need
  from~\cite{HuMathas:Morita}.

Throughout this paper we fix positive integers $n,r,p$ and $d$ such
that $n\ge3$, $p>1$ and $r=pd$. Let $R$ be a commutative ring which
contains a primitive $p$th root of unity $\varepsilon$.

Suppose that $q,Q_1,\dots,Q_r$ are invertible elements of $K$. The
\textbf{Ariki-Koike algebra} $\H^R_{r,n}(q,Q_1,\dots,Q_r)$ is the
unital associative $R$-algebra with generators $T_0,T_1,\dots,T_{n-1}$ and
relations
\[\begin{array} {r@{\ }l@{\ }ll}
    (T_0-Q_1)\dots(T_0-Q_r) &=&0, \\
(T_i-q)(T_i+1) &=&0,&\text{for $1\le i\le n-1$,}\\
T_0T_1T_0T_1&=&T_1T_0T_1T_0,\\
T_{i+1}T_iT_{i+1}&=&T_iT_{i+1}T_i,&\text{for $1\le i\le n-2$,}\\
T_iT_j&=&T_jT_i,&\text{for $0\le i<j-1\le n-2$.}
\end{array}\]

To define the Hecke algebras of type $G(r,p,n)$ we fix $\bQ=(Q_1,\dots,Q_d)\in K^d$
and replace the $(Q_1,\dots,Q_r)$ in the above definition by $\bbQ$, where
\begin{equation}\label{E:epsShifts}
\bbQ=\eps\bQ\vee\eps^2\bQ\vee\dots\vee\eps^p\bQ
      =(\eps Q_1,\dots,\eps Q_d,\dots,\eps^pQ_1,\dots,\eps^p Q_d).
\end{equation}
We set $\Hrn^R(\bbQ):=\H_{r,n}^R(q,\bbQ)$. Then the relation for $T_0$ in
$\Hrn(\bbQ)$ can be written as $(T_0^p-Q_1^p)\dots(T_0^p-Q_d^p)=0$. When $R$ and
$q,Q_1,\dots,Q_d$ are understood we write $\Hrn=\Hrn(\bbQ)$.

\begin{Defn} \label{D:lSplittable}
The \textbf{cyclotomic Hecke algebra of type $G(r,p,n)$} is the
subalgebra $\Hrpn=\Hrpn(\bQ)$ of~$\Hrn(\bbQ)$ generated
by the elements $T_0^p,T_u=T_0^{-1}T_1T_0$ and
$T_1,T_2,\cdots,T_{n-1}$.
\end{Defn}

In this paper we are interested in understanding the decomposition matrices of
the algebras $\Hrpn$. Although this will not be apparent for quite some time, we
have chosen the ordering of the `cyclotomic parameters'
$\eps\bQ\vee\eps^2\bQ\vee\dots\vee\eps^p\bQ$ in order to ensure that the
labelling of the irreducible modules for~$\Hrn$ and~$\Hrpn$ are compatible in
the sense of Theorem~\ref{main simple}.

The algebra $\Hrn$ comes equipped with two automorphisms $\sigma$ and $\tau$
which are useful when
studying~$\Hrpn$. Let~$\sigma$ be the unique automorphism of~$\Hrn$ such that
\begin{equation}\label{E:sigma}
  \sigma(T_0)=\eps T_0\quad\text{ and }\quad\sigma(T_i)=T_i,
        \qquad\text{ for }1\le i<n,
\end{equation}
and define $\tau$ by $\tau(h)=T_0^{-1}hT_0$, for $h\in\Hrn$. It is
straightforward to check that $\Hrpn = \set{h\in\Hrn|\sigma(h)=h}$ is the set of
$\sigma$-fixed points in~$\Hrn$ and that $\tau$ restricts to an automorphism
of~$\Hrpn$. Moreover, $\sigma$ is an automorphism of~$\Hrn$ of order~$p$
and~$\tau$ is an automorphism of~$\Hrpn$ with the property that $\tau^p$ is an inner
automorphism of~$\Hrpn$.

Fix a modular system $(F,\O,K)$ ``with parameters'' for $\Hrpn$. That is, fix
an algebraically closed field~$F$ of characteristic zero, a discrete valuation
ring $\O$ with maximal ideal $\pi$ and with residue field $K\cong\O/\pi$, together
with parameters $\hat{q},\hat{Q}_1,\dots,\hat{Q}_d\in\O^\times$ such that
$q=\hat{q}+\pi$ and $Q_i=\hat{Q}_i+\pi$ for each $i$. Let
$\H^F_{r,p,n}=\H^F_{r,p,n}(\hat{\bQ})$ be the Hecke algebra of type $G(r,p,n)$
over~$F$ with parameters $\hat{q}$ and $\hat{\bQ}=(\hat{Q}_1,\dots,\hat{Q}_d)$
and similarly let $\H_{r,p,n}^\O=\H_{r,p,n}(\hat{\bQ})$ and write
$\H_{r,p,n}^K=\H_{r,p,n}(\bQ)$. We assume that $\H_{r,p,n}^F$ is semisimple.
By Lemma~\ref{L:AKBasis} below, $\H^F_{r,p,n}\cong\H^\O_{r,p,n}\otimes_\O F$ and
$\H^K_{r,p,n}\cong\H^\O_{r,p,n}\otimes_\O K$. Hence, by choosing $\O$--lattices
we can talk of modular reduction from $\H_{r,p,n}^F$--Mod to
$\H_{r,p,n}^K$--Mod.

The automorphisms~$\sigma$ and~$\tau$ commute with modular reduction. Hence, we
have compatible automorphisms~$\sigma$ and~$\tau$ on $\H_{r,n}^F$ and on
$\H_{r,n}^K$.

Let $R\in\{F,K\}$ and let $M$ be an $\H^R_{r,p,n}$--module.  Then we
define a new $\H^R_{r,p,n}$--module $M^\tau$ by ``twisting'' the
action of $\H^R_{r,p,n}$using the automorphism $\tau$.
Explicitly, $M^\tau=M$ as a vector space and the
$\H^R_{r,p,n}$--action on $M^\tau$ is defined by
\[m\cdot h=m\tau(h),\qquad
                \text{for all }m\in M\text{ and }h\in\H^R_{r,p,n}.\]
If $M$ is an $\Hrpn$-module then $M\cong M^{\tau^p}$ because $\tau^{p}$ is an
inner automorphism of~$\Hrpn$. Therefore, there is a natural action of the cyclic
group $\Z/p\Z$ on the set of isomorphism classes of $\H^R_{r,p,n}$--modules.
The {\bf inertia group} of $M$ is the group
\[I_M=\set{k|0\leq k<p, M\cong M^{\tau^k}},\]
which we consider as a subgroup of~$\Z/p\Z$.

Suppose that $S$ is an irreducible $\H^F_{r,p,n}$-module and let $S_K$ be
$\H^K_{r,p,n}$-module obtained from~$S$ by modular reduction. Let $D$ be an
irreducible $\H^K_{r,p,n}$-module. By
\cite[Corollary~5.6]{HuMathas:Morita}, the decomposition number $[S_K:D]$ is {\bf
$p$-splittable} in the sense of Definition~\ref{D:lSplittable} if and only if
$I_{S}=\{0\}=I_{D}$.

If $\alpha\in K$ then the \textbf{$q$-orbit} of~$\alpha$ is the set
$\set{q^b\alpha|b\in\Z}$. Similarly, the \textbf{$(\eps,q)$-orbit} of $\alpha$
is$\set{\eps^aq^b\alpha|a,b\in\Z}$.

One of the main results of~\cite{HuMathas:Morita} is the following.

\begin{Theorem}[(\protect{\cite[Theorem~B]{HuMathas:Morita}})]
  \label{T:SplittableRedction} The decomposition
    numbers of the cyclotomic Hecke algebras of type $G(r,p,n)$ are
    completely determined by the $l$-splittable decomposition
    numbers of certain cyclotomic Hecke algebras
    $\H_{s,l,m}(\bQ')$, where~$l$ divides~$p$, $1\le s\le r$,
    $1\le m\le n$ and where the parameters $\bQ'$ are contained in a
    single $(\eps,q)$-orbit.
\end{Theorem}

Hence, to prove Theorem~\ref{main} it is enough to compute all of the $p$-splittable
decomposition numbers of~$\Hrpn$ and to show that they are determined by the
decomposition numbers of Hecke algebras of type~$G(s,1,m)$, where $s\ge1$
divides~$r$ and $1\le m\le n$. 

To compute the $p$-splittable decomposition numbers of~$\Hrpn$ we make extensive
use of the following result which is a more precise statement of~\eqref{E:Morita}.

\begin{Theorem}[(\cite{DM:Morita,HuMathas:Morita})]\label{T:DMMorita}
  Suppose that $(Q_1,\dots,Q_r)=\bQ_1\vee\dots\vee\bQ_{\gamma}$, where
$Q_i\in\bQ_{\alpha}$ and $Q_j\in\bQ_{\beta}$ are in the same
$q$--orbit only if $\alpha=\beta$. Let
$d_\alpha=|\bQ_{\alpha}|$, for $1\le\alpha\le\gamma$. Then
$\Hrn(q,\bQ)$ is Morita equivalent to the algebra
\[\bigoplus_{b_1+\cdots+b_{\gamma}=n}\H_{d_1,b_1}(q,\bQ_1)\otimes\cdots
\otimes\H_{d_{\gamma},b_{\gamma}}(q,\bQ_{\gamma}).\]
\end{Theorem}

Recall from~\eqref{E:separated} that $\bQ$ is $(\eps,q)$-separated if
\[\prod_{1\le i,j\le d}\prod_{-n<k<n}\prod_{1\leq t<p}
       \big(Q_i-\eps^{t}q^{k}Q_j\big)\ne0.\]
To prove our main results we apply Theorem~\ref{T:DMMorita} to the decomposition
$\bbQ=\eps\bQ\vee\eps^2\bQ\vee\dots\vee\eps^p\bQ$. This
will allow us to understand the Specht modules $S(\blam)$ of~$\Hrn$ in terms of
the Specht modules of the algebras $\H_{d,\b}=\H_{d,b_1}(\eps\bQ)\otimes\cdots
\otimes\H_{d,b_{p}}(\eps^p\bQ)$. The main benefit in doing
this is that because of our ordering on the parameters in~$\bbQ$ it is easier to
understand the action of the automorphisms~$\sigma$ on $\H_{d,\b}$-modules and
this gives us a way to understand the action of~$\sigma$ on~$S(\blam)$. In turn,
this will allow to decompose the restriction of~$S(\blam)$ to~$\Hrpn$.

When using Clifford theory to understand the representation theory of the group
$G(r,p,n)$ in terms of the representation theory of $G(r,1,n)$ it is natural to
assume that the characteristic of the field does not divide $p$, which is the
index of $G(r,p,n)$ in $G(r,1,n)$. As the following results indicates,
$(\eps,q)$-separation can be viewed as a quantum analogue of this condition.

\begin{Lemma}\label{L:separated}
  Suppose that $\bQ\in K^d$ and $R=K$ is a field. Then:
  \begin{enumerate}
    \item Suppose that $\bQ$ is $(\eps,q)$-separated over~$K$. Then
    \[\prod_{-n<k<n}\prod_{1\le t<p}(1-\eps^tq^k)\ne0.\]
    In particular, $\<\eps\>\cap\<q\>=\{1\}$ if $q^k=1$ for some $1\le k<n$.
    \item Suppose that $\<\eps\>\cap\<q\>=\{1\}$ and that $\bQ$ is contained in a single
    $q$-orbit. Then $\bQ$ is $(\eps,q)$-separated over~$K$.
  \end{enumerate}
\end{Lemma}

\begin{proof}
  Part~(a) follows by looking at the terms in the product corresponding to~$i=j$
  in~\eqref{E:separated}. For~(b) observe that if $\<\eps\>\cap\<q\>=\{1\}$ then
  $\eps^tq^k\ne1$ if $t\ne p$. Since~$\bQ$ is contained in a single
  $q$-orbit this implies the result.
\end{proof}

\subsection{Jucys-Murphy elements and a basis for $\Hrn$}
In order to define a basis for $\Hrn$
let $\Sym_n$ be the symmetric group on $n$ letters and let
$s_i=(i,i+1)\in\Sym_n$ be a simple transposition, for $1\le i<n$. Then
$\{s_1,\dots,s_{n-1}\}$ are the standard Coxeter generators of the
symmetric group $\Sym_n$. Let $\ell\map{\Sym_n}\N$ be the
length function on~$\Sym_n$, so that $\ell(w)=k$ if $k$ is minimal
such that $w=s_{i_1}\dots s_{i_k}$, where $1\le i_1,\dots,i_k<n$.  As
the type~$A$ braid relations hold in $\Hrn$ for each $w\in\Sym_n$
there is a well--defined element $T_w\in\Hrn$, where
$T_w=T_{i_1}\dots T_{i_k}$ whenever $w=s_{i_1}\dots s_{i_k}$ and
$k=\ell(w)$.

Set $L_1=T_0$ and $L_{k+1}=q^{-1}T_kL_kT_k$, for $k=1,\dots,n-1$. These
elements $L_i$ are the Jucys--Murphy elements of $\Hrn$ and they
generate a commutative subalgebra of $\Hrn$.

\begin{Lemma}\label{L:AKBasis}
  \begin{enumerate}
    \item{\rm\cite[Theorem~3.10]{AK}} The algebra $\Hrn$ is free as
    an $R$-module with basis
  $\set{L_1^{a_1}\dots L_n^{a_n}T_w|0\le a_j<r\text{ and }w\in\Sym_n}$.
  \item{\rm\cite[Proposition~1.6]{Ariki:Hecke}}
  The algebra $\Hrpn$ is free as an $R$-module with basis
  $\set{L_1^{a_1}\dots L_n^{a_n}T_w|0\le a_j<r, a_1+\dots+a_n\equiv0\pmod p
     \text{ and }w\in\Sym_n}$.
  \end{enumerate}
\end{Lemma}

Inspecting the relations, there is a unique anti-isomorphism
$*$ of $\Hrn$ which fixes each of the generators $T_0,T_1,\dots,T_{n-1}$ of~$\Hrn$. We have
$T_w^*=T_{w^{-1}}$ and $L_k^*=L_k$, for $1\le k\le n$.

We will use the
following well-known properties of the Jucys-Murphy elements
without mention.

\begin{Lemma}[(\protect{cf.~\cite[Lemma~3.3]{AK}})]\label{JM properties}
    Suppose that $1\le i<n$ and $1\le k\le n$. Then
    \begin{enumerate}
    \item $T_i$ and $L_k$ commute if $i\ne k,k-1$.
    \item $T_k$ commutes with $L_kL_{k+1}$ and $L_k+L_{k+1}$.
    \item $T_kL_k=L_{k+1}(T_k-q+1)$ and
        $T_kL_{k+1}=L_kT_k+(q-1)L_{k+1}$.
    \end{enumerate}
\end{Lemma}
For integers $k$ and $s$, with $1\le k\le n$ and $1\le s\le p$, set
\[\LL[s]_k=\prod_{i=1}^d(L_k-\eps^s Q_i).\]
More generally, if $1\le l\le m\le n$ and $1\le i,j\le p$ then set
\[\LL[i,j]_{l,m}=\prod_{\substack{l\le k\le m\\s\in I_{ij}}}\LL[s]_k
   =\prod_{\substack{l\le k\le m\\s\in I_{ij}}}\prod_{t=1}^d(L_k-\eps^s Q_t),
\]
where $I_{ij}=\{i,i+1,\dots,j\}$, if $i\le j$,
and $I_{ij}=\{1,2\dots,j,i,i+1,\dots,p\}$ if $i>j$.

A key property of the Jucys-Murphy elements of $\Hrn$ is that
$T_i$ commutes with any polynomial in $L_1,\dots,L_n$ which is
symmetric with respect to $L_i$ and $L_{i+1}$. In particular, any
symmetric polynomial in $L_1,\dots,L_n$ is central in~$\Hrn$.
Hence, we have the following.

\begin{Lemma}\label{commutes}
    Suppose that $1\le l<m\le n$ and $1\le t\le p$. Then
    \[T_i\LL[t]_{l,m}=\LL[t]_{l,m}T_i\qquad\text{and}\qquad
      L_j\LL[t]_{l,m}=\LL[t]_{l,m}L_j,\]
    for all $i,j$ such that $1\le i<n$, $1\le j\le n$
    and $i\ne l-1,m$.
\end{Lemma}

Throughout this paper we will need some special permutations.
For non-negative integers $a, b$ with $0<a+b\le n$ we set
$w_{a,b}=(s_{a+b-1}\dots s_1)^b$. (In particular, $w_{a,0}=1=w_{0,b}$.) If
we write $w_{a,b}\in\BS_{a+b}$ as a permutation in two-line notation then
\begin{equation}\label{E:wab}
w_{a,b}=\Big(\begin{array}{*6c}1&\cdots&a&a+1&\cdots&a+b\\b+1&\cdots&a+b&1&\cdots&b
        \end{array}\Big).
\end{equation}
For simplicity, we write $T_{a,b}=T_{w_{a,b}}$. Similarly, if $k$ is a
non-negative integer such that $0<a+b+k\le n$ then we set
$w_{a,b}^{\shift k}=(s_{a+b+k-1}\dots s_{k+1})^b$. Then
$w_{a,b}=w_{a,b}^{\shift 0}$ and, abusing notation slightly, we write
$T_{a,b}^{\shift k}=T_{w_{a,b}^{\shift k}}$.

The following result is easily checked.

\begin{Lemma}\label{wab}
  Suppose that $a$, $b$ and $c$ are non-negative integers such that
  $a+b+c\le n$. Then $w_{a,b+c}=w_{a,b}w^{\shift b}_{a,c}$
  and $w_{a+b,c}=w^{\shift a}_{b,c}w_{a,c}$, with the lengths adding.
  Consequently, $T_{a,b+c}=T_{a,b}T^{\shift b}_{a,c}$ and
  $T_{a+b,c}=T^{\shift a}_{b,c}T_{a,c}$. Moreover,
  $T_iT^{\shift c}_{a,b}=T^{\shift c}_{a,b}T_{(i)w^{\shift c}_{a,b}}$
  if $1\le i<n$ and $i\ne a+c$.
\end{Lemma}

\subsection{The elements $v_\b$ and $v_\b^{(t)}$}\label{S:vb}
As remarked in the introduction, all of the results in this paper rely on the
Morita equivalence of Theorem~\ref{T:DMMorita}. This equivalence is induced by
certain $(\H_{d,\b},\Hrn)$-bimodules $V_\b=v_\b\Hrn$. In this section we define
these modules and, in Proposition~\ref{changing}, give one of the key properties
of the elements $v_\b$.

Recall from the introduction that $\Comp$ is the set of compositions
of $n$ into $p$ parts. Thus, $\b\in\Comp$ if and only if
$\b=(b_1,\dots,b_p)$, $b_1+\dots+b_p=n$ and $b_i\ge0$, for all $i$.
Finally, if $\b\in\Comp$ and $i$ and $j$ are integers then we set
$\b_i^j=b_i+\dots+b_j$ if $i\le j$ and $b_i^j=0$ if $i>j$.

The following elements of $\Hrn$ were introduced in
\cite[Definition~2.4]{HuMathas:Morita}. They play an important role throughout
this paper.

\begin{Defn}\label{vb defn}
  Suppose that $\b\in\Comp$. Let
\[v_\b(\bQ)=\LL[1,p-1]_{1,b_p}T_{b_p,\b_1^{p-1}}
  \LL[1,p-2]_{1,b_{p-1}}T_{b_{p-1},\b_1^{p-2}}
  \dots\LL[1,1]_{1,b_2}T_{b_2,\b_1^{1}}
  \LL[2]_{1,\b_1^{1}}\LL[3]_{1,\b_1^{2}}\dots\LL[p]_{1,\b_1^{p-1}}\]
We write $v_\b=v_\b(\bQ)$ and for $t\in\Z$  set
$v_\b^{(t)}=v_\b(\eps^t\bQ)$.

Set $V_\b=v_\b\Hrn$ and, more generally, let $V_\b^{(t)}=v_\b^{(t)}\Hrn$.
\end{Defn}

The element $v_\b$ can be written in many different (and useful) ways. The proof
of the next result requires several long and uninspiring calculations, so we
refer the reader to Proposition~\ref{changingII} in the appendix for the proof.

\begin{prop}\label{changing}
  Suppose that $\b\in\Comp$ and $1\le j\le p$. Then
  \[v_\b = \prod_{j\le k<p}\LL[j,k]_{1,b_{k+1}} T_{b_{k+1},\b_j^k}
  \cdot\prod_{1\le i<j}\LL[i]_{1,\b_{i+1}^p}
          \cdot\prod_{j<k\le p}\LL[k]_{1,\b_j^{k-1}}
          \cdot\prod_{1<i\le j}T_{\b_{i}^p,b_{i-1}}\LL[i,p]_{1,b_{i-1}},
   \]
   where all products are read from left to right with
   \underline{decreasing} values of $i$ and $k$.
\end{prop}

\begin{cor}\label{leftfactor}
    Suppose that $\b\in\Comp$ and $t\in\Z$. Then
    $v_\b^{(t)}\in\LL[t+1]_{1,\b_2^p}\Hrn$.
\end{cor}

\begin{proof}
  It is enough to consider the case when $t=0$ and $v_\b^{(t)}=v_\b$. In this
  case the result follows by taking $j=p$ in Proposition~\ref{changing}.
\end{proof}

For $t=1,\dots,p$, let $Y_t=\LL[t+1,t+p-1]_{1,b_t}T_{b_t,n-b_t}$. If
$\mathbf{a}=(a_1,a_2,\dots,a_m)$ is any sequence then set
$\mathbf{a}\shift k=(a_{k+1},a_{k+2},\dots,a_{k+m})$, where $a_{i+jm}:=a_i$
for $j\in\Z$.

\begin{cor} \label{leftmult} Suppose that $\b\in\Comp$ and $1\le t\le p$. Then
  \[ Y_t v_{\b\shift{t-1}}^{(t-1)} =v_{\b\shift t}^{(t)}Y_t^*.
\]
\end{cor}

\begin{proof} It is enough to consider the case $t=1$. Taking $j=2$ in
  Proposition~\ref{changing},
    \begin{align*}
    Y_tv_\b
      &= \LL[2,p]_{1,b_1}T_{b_1,\b_2^p} \cdot
         \LL[2,p-1]_{1,b_p}T_{b_p,\b_2^{p-1}}\dots\LL[2,2]_{1,b_3}T_{b_3,\b_2^2}
         \LL[3]_{1,\b_2^2}\dots\LL[p]_{1,\b_2^{p-1}}\\
      &\qquad\times\space\LL[1]_{1,\b_2^p} T_{\b_2^p,b_1}\LL[2,p]_{1,b_1}\\
      &=v_{\b\shift1}^{(1)}T_{\b_2^p,b_1}\LL[2,p]_{1,b_1},
\end{align*}
as required.
\end{proof}

The point of Corollary~\ref{leftmult} is that left multiplication by $Y_t$
defines an~$\Hrn$-module homomorphism from
$V_{\b\shift{t-1}}^{(t-1)}=v_{\b\shift{t-1}}^{(t-1)}\Hrn$ to
$V_{\b\shift t}^{(t)}=v_{\b\shift{t}}^{(t)}\Hrn$

\begin{Defn}\label{theta defn}
    Suppose that $1\le t\le p$ and $\b\in\Comp$. Then $\theta'_t$ is the
    $\Hrn$-module homomorphism
    \[\theta'_t\map{V_{\b\shift{t-1}}^{(t-1)}}
                 {V_{\b\shift{t}}^{(t)}}; x\mapsto Y_tx,\]
    for all $x\in V_{\b\shift{t-1}}^{(t-1)}$.
\end{Defn}

Since $v_\b=v_{\b\shift p}^{(p)}$, composing the maps
$\theta'_p\circ\dots\circ\theta'_1$ gives an $\Hrn$-module endomorphism
of~$v_\b\Hrn$. We need another description of this map.

\begin{prop} \label{pleftmult} Suppose that $\b\in\Comp$. Then
$Y_pY_{p-1}\dots Y_2Y_1=v_\b T_\b$.
\end{prop}

This result is proved in the appendix as Proposition~\ref{pleftmultII}.

\subsection{The central element $z_\b$}\label{S:zb}
The aim of this section is to prove the existence of the central element $z_\b$
which appears in Theorem~\ref{scalars}. We start by studying the elements
$Y_tv_{\b\shift{t-1}}^{(t-1)}$. Generalising~\eqref{E:wab}, for $\b\in\Comp$ set
\[w_\b=w^{\shift{\b_1^{p-2}}}_{b_{p-1},\b_{p}^p}
     w^{\shift{\b_1^{p-3}}}_{b_{p-2},\b_{p-1}^p}
    \dots w^{\shift{\b_1^1}}_{b_2,b_3^p}w_{b_1,\b_2^p}.\]
In two-line notation, $w_\b$ is the permutation
\[
\Big(\begin{array}{*{11}c}
1&\dots& \b_1^1&\b_1^1+1&\dots&\b_1^2&\b_1^2+1&\dots&\b_1^{p-1}+1&\dots&\b_1^p\\
\b_2^p+1&\dots&\b_1^p&\b_3^p+1&\dots&\b_2^p&\b_4^p+1&\dots&1&\dots&\b_p^p
\end{array}\Big).\]
Note that $b_1=\b_1^1$, $b_p=\b_p^p$ and $n=\b_1^p$. Also, if $\b=(a,b)$
then $w_\b=w_{a,b}$.

For convenience we set $T_\b=T_{w_\b}$. For example, $T_{a,b}=T_{w_{a,b}}$.

For any $\b=(b_1,\dots,b_p)\in\Comp$ we define $\bp=(b_p,\dots,b_1)$.
Since $w_{a,b}^{-1}=w_{b,a}$ it follows that $w_{\bp}=w_\b^{-1}$.

Finally, set
$\Sym_b=\Sym_{b_1}\times\Sym_{b_2}\times\dots\times\Sym_{b_p}$, which we
consider as a subgroup of $\Sym_n$ in the obvious way. Similarly,
$\H_q(\Sym_\b)$ is a subalgebra of $\H_q(\Sym_n)$ via the natural
embedding.

The following important property of $v_\b$ was established
in~\cite{HuMathas:Morita}.

\begin{Lemma}[\!\!(\protect{\cite[Proposition~2.5]{HuMathas:Morita}})]
    \label{vb commutation}
    Suppose that $\b\in\Comp$ and $1\le i,j\le n$, with $i\ne \b_t^p$
    for $1\le t\le p$. Then
    \begin{enumerate}
    \item $T_i v_\b=v_\b T_{(i)w_\b^{-1}}$, and
    \item $L_j v_\b=v_\b L_{(j)w_\b^{-1}}$.
    \end{enumerate}
\end{Lemma}

Using this fact we can  prove the following two results.

\begin{Lemma}  \label{exchange2}
  Suppose that $1\le t\le p$ and let $i$ and $j$ be integers such that
  $1\le i,j\le n$ and $i\neq\b_\alpha^t$ for $\alpha=t-p+1,t-p+2,\dots,t$.
  Then
    \begin{align*}
T_i\Big(Y_{t}v_{\b\shift{t-1}}^{(t-1)}\Big)&=\begin{cases}\Big(Y_{t}v_{\b\shift{t-1}}^{(t-1)}\Big)T_i,
&\text{if $1\leq i<b_t$;}\\
\Big(Y_{t}v_{\b\shift{t-1}}^{(t-1)}\Big)T_{(i)w_{\b\shift{t-1}'}},
&\text{if $b_t+1\leq i<n$,}
\end{cases}\\
L_j\Big(Y_{t}v_{\b\shift{t-1}}^{(t-1)})&=\begin{cases}\Big(Y_{t}v_{\b\shift{t-1}}^{(t-1)}\Big)L_j,
&\text{if $1\leq j\leq b_t$;}\\
\Big(Y_{t}v_{\b\shift{t-1}}^{(t-1)}\Big)L_{(j)w_{\b\shift{t-1}'}},
&\text{if $b_t+1\leq j\leq n$,}
\end{cases}
\end{align*}
\end{Lemma}

\begin{proof} For the first equality, if $i\ne b_t$ then using
    Lemmas~\ref{wab} and~\ref{commutes}
    \begin{align*}
    T_iY_tv_{\b\shift{t-1}}^{(t-1)}
       &=T_i\LL[t+1,t+p-1]_{1,b_t}T_{b_t,n-b_t}v_{\b\shift{t-1}}^{(t-1)}\\
       &=\LL[t+1,t+p-1]_{1,b_t}T_iT_{b_t,n-b_t}v_{\b\shift{t-1}}^{(t-1)}\\
       &=\LL[t+1,t+p-1]_{1,b_t}T_{b_t,n-b_t}T_{(i)w_{b_t,n-b_t}}
            v_{\b\shift{t-1}}^{(t-1)}.
    \end{align*}
The first claim now follows using Lemma~\ref{vb commutation}. For the
second claim observe that by Corollary~\ref{leftfactor}(b) there
exists an $h\in\Hrn$ such that
\begin{align*}
    L_jY_tv_{\b\shift{t-1}}^{(t-1)}
          &=L_jv^{(t+1,t+p-1)}_{1,b_t}h
      =v^{(t+1,t+p-1)}_{b_t,n-b_t}L_{(j)w_{b_t,n-b_t}}h\\
      &=\LL[t+1,t+p-1]_{1,b_t}T_{b_t,n-b_t}L_{(j)w_{b_t,n-b_t}}v_{\b\shift{t-1}}^{(t-1)}\\
&=Y_tL_{(j)w_{b_t,n-b_t}}v_{\b\shift{t-1}}^{(t-1)}.
\end{align*}
So the result again follows using Lemma~\ref{vb commutation}.
\end{proof}

\begin{Lemma} \label{exchange}
  Suppose that $1\le t\le p$ and let $i$ and $j$ be integers such that
  $1\le i,j\le n$ and $i\neq\b_\alpha^t$ whenever $t-p+1\leq\alpha\le t$. Then
\begin{align*} T_i(Y_{t}\dots Y_2Y_1v_\b)
  &=\begin{cases} (Y_{t}\dots Y_2Y_1v_\b)T_{i+\b_1^{t-1}},
&\text{if $1\leq i<b_t$;}\\
(Y_{t}\dots Y_2Y_1v_\b)T_{i-b_t+b_1+\dots+b_{t-2}},
&\text{if $b_{t}+1\leq i<\b_{t-1}^t$;}\\
\qquad\,\vdots &\\
(Y_{t}\dots Y_2Y_1v_\b)T_{i-\b_2^t},
&\text{if $\b_2^t+1\leq i<\b_1^t$;}\\
(Y_{t}\dots Y_2Y_1v_\b)T_{(i-\b_1^t)w_{\bp}},
&\text{if $\b_1^t+1\leq i<n$;}
\end{cases}\\
L_j(Y_{t}\dots Y_2Y_1v_\b)&=\begin{cases} (Y_{t}\dots
Y_2Y_1v_\b)L_{j+\b_1^{t-1}},
&\text{if $1\leq j\leq b_t$;}\\
(Y_{t}\dots Y_2Y_1v_\b)L_{j-b_t+b_1+\dots+b_{t-2}},
&\text{if $b_{t}+1\leq j\leq \b_{t-1}^t$;}\\
\qquad\,\vdots &\\
(Y_{t}\dots Y_2Y_1v_\b)L_{j-\b_2^t},
&\text{if $\b_2^t+1\leq j\leq\b_1^t$;}\\
(Y_{t}\dots Y_2Y_1v_\b)L_{(j-\b_1^t)w_{\bp}},
&\text{if $\b_1^t+1\leq j\leq n$.}
\end{cases}
\end{align*}
\begin{align*}
T_i(Y_{t}\dots Y_2Y_1v_\b)
   &=(Y_{t}\dots Y_2Y_1v_\b)T_{(i)w_{(b_t,\dots,b_1)}}\\
L_j(Y_{t}\dots Y_2Y_1v_\b)
   &=(Y_{t}\dots Y_2Y_1v_\b)T_{(j)w_{(b_t,\dots,b_1)}}\\
\intertext{In particular, taking $t=p$, we have}
T_i(Y_{p}\dots Y_2Y_1v_\b)
      &=(Y_{p}\dots Y_2Y_1v_\b)T_{(i)w_{\bp}},\\
L_j(Y_{p}\dots Y_2Y_1v_\b)
      &=(Y_{p}\dots Y_2Y_1v_\b)L_{(j)w_{\bp}}.
\end{align*}
\end{Lemma}

\begin{proof} This can be proved in exactly the same way as
    Lemma~\ref{exchange2}. Note that the final claim also follows
    from Proposition~\ref{pleftmult} using Lemma~\ref{wab}.
\end{proof}

The following definition is repeated from~\eqref{E:separated}.

\begin{Defn}\label{separated}
  Suppose that $R$ is a commutative ring with $1$ and set
  \[ A(\eps,q,\bQ) =\prod_{1\le i,j\le d}\prod_{-n<k<n}\prod_{1\leq t<p}
       \big(Q_i-\eps^{t}q^{k}Q_j\big).\]
  Then  $\bQ$ is \textbf{$(\eps,q)$-separated} in
  $R$ if $A(\eps,q,\bQ)$ is invertible in~$R$.
\end{Defn}

Observe that, even though our notation does not reflect this, whether or
not $\bQ$ is $(\eps,q)$-separated also depends on~$n$ and the ring~$R$.

\begin{Remark} When $d=1$, the algebra $\H_{q}(\BS_\b)$ can be naturally embedded into
  $\Hrn$ as a subalgebra; see \cite{Hu:ModGppn}.
In that case, the condition of being $(\eps,q)$-separated means that
$\prod_{|k|<n, 1\leq t<p}\big(1-\eps^t q^k\big)$ is invertible.

Fix $\b\in\Comp$ and set $V_\b=v_\b\Hrn$ and
$\H_{d,\b}=\H_{d,b_1}(\eps\bQ)\otimes\cdots\otimes\H_{d,b_p}(\eps^p\bQ)$.
Then an important result from \cite{HuMathas:Morita} is the following.

\begin{prop}[\!\!(\protect{\cite[Proposition~2.15]{HuMathas:Morita}})]
    \label{faithful}
  Suppose that $\b\in\Comp$ and that $\bQ$ is
  $(\eps,q)$-separated if $d>1$. Then:
  \begin{enumerate}
    \item $\H_{d,\b}$ acts faithfully on
  $V_\b$ from the left and $\End_{\Hrn}(V_\b)\cong\H_{d,\b}$.
  \item $V_\b$ is projective as an $\Hrn$-module and
        $\bigoplus_{\b\in\Comp}V_\b$ is a progenerator for $\Hrn$.
  \end{enumerate}
\end{prop}

\end{Remark}

To describe the action of $\H_{d,\b}$ on $V_\b$ given a
permutation $w=s_{i_1}\dots s_{i_k}\in\Sym_n$ and an integer $c\in\N$
such that $i_j+c<n$, for $1\le j\le k$, define $w^{\shift
c}=s_{i_1+c}\dots s_{i_k+c}$. Then $w^{\shift c}\in\Sym_n$. Note that
this is compatible with our previous definition of $w_{a,b}^{\shift
c}$.

Define $\RTheta_\b$ to be the `natural inclusion map'
$\H_{d,\b}\hookrightarrow\Hrn$. That is, $\RTheta_\b$ is the
$R$-linear map determined by
\begin{align*}
&\RTheta_\b\Big(\big(L_1^{a_{1,1}}\dots L_{b_1}^{a_{1,b_1}}T_{x_1}\big)\otimes
\big(L_1^{a_{2,1}}\dots L_{b_2}^{a_{2,b_2}}T_{x_2}\big)\otimes\cdots\otimes
\big(L_1^{a_{p,1}}\dots L_{b_p}^{a_{p,b_p}}T_{x_p}\big)\Big)\\
&\qquad=\big(L_1^{a_{1,1}}\dots L_{b_1}^{a_{1,b_1}} T_{x'_1}\big)
  \big(L_{\b_1^1+1}^{a_{2,1}}\dots L_{\b_1^1+b_2}^{a_{2,b_2}} T_{x'_2}\big)
  \cdots\big(L_{\b_1^{p-1}+1}^{a_{p,1}}\dots L_{\b_1^{p-1}+b_p}^{a_{p,b_p}}
        T_{x'_p}\big)\\
&\qquad=\big(L_1^{a_{1,1}}\dots L_{b_1}^{a_{1,b_1}} \big)
  \big(L_{\b_1^1+1}^{a_{2,1}}\dots L_{\b_1^1+b_2}^{a_{2,b_2}} \big)
  \cdots\big(L_{\b_1^{p-1}+1}^{a_{p,1}}\dots L_{n}^{a_{p,b_p}}\big)
        T_{x'_1}T_{x'_2}\dots T_{x'_p},
\end{align*}
for all $x_t\in\BS_{b_t}$ and $0\leq a_{j,t}<d$, for $1\leq t\leq p$ and
$1\leq j\leq b_t$, and where $x'_t:=x_t^{\shift{\b_1^{t-1}}}$,  for $1\leq
t\leq p$. The second equality follows because all of these terms commute.
Thus, we have $x'_1=x_1$ and
$\RTheta_\b(T_{x_1}\otimes\dots\otimes T_{x_p})=T_w$, where
$w=x_1x_2^{\shift{\b_1^1}}\dots x_p^{\shift{\b_1^{p-1}}}\in\Sym_\b$, for
$x_t\in\Sym_{b_\t}$.  We emphasize that $\RTheta_\b$ is an $R$-module
homomorphism but \textit{not} a ring homomorphism.

Similarly, define $\LTheta_\b$ to be the $R$-linear map
$\LTheta_\b\map{\H_{d,\b}}\Hrn$ determined by
\begin{align*}
&\LTheta_\b\Big(\big(L_1^{a_{1,1}}\dots L_{b_1}^{a_{1,b_1}}T_{x_1}\big)\otimes
\big(L_1^{a_{2,1}}\dots L_{b_2}^{a_{2,b_2}}T_{x_2}\big)\otimes\cdots\otimes
\big(L_1^{a_{p,1}}\dots L_{b_p}^{a_{p,b_p}}T_{x_p}\big)\Big)\\
&\qquad\qquad=
  \big(L_1^{a_{p,1}}\dots L_{b_p}^{a_{p,b_p}} T_{x''_p}\big)\cdots
  \big(L_{\b_3^p+1}^{a_{2,1}}\dots L_{\b_2^p}^{a_{2,b_2}} T_{x''_2}\big)
  \big(L_{\b_2^p+1}^{a_{1,1}}\dots L_{\b_1^p}^{a_{1,b_1}} T_{x''_1}\big) \\
&\qquad\qquad=
  \big(L_1^{a_{p,1}}\dots L_{b_p}^{a_{p,b_p}}\big) \dots
  \big(L_{\b_3^p+1}^{a_{2,1}}\dots L_{\b_2^p}^{a_{2,b_2}} \big)
  \big(L_{\b_2^p+1}^{a_{1,1}}\dots L_{\b_1^p}^{a_{1,b_1}}\big)
  T_{x''_1}T_{x''_2}\dots T_{x''_p},
\end{align*}
where the $x_t$ and $a_{t,j}$ are as before and
$x''_t:=w_\b^{-1}x_t^{\shift{\b_1^{t-1}}}w_\b=w_\b^{-1} x_t'w_\b$. In
particular, $x''_p=x_p$ and
$x''_1x''_2\dots x''_p
  =w_\b^{-1}\big(x_1x_2^{\shift{\b_1^1}}\dots x_p^{\shift{\b_1^{p-1}}}\big)w_\b
  \in\Sym_\bp$.

  Given these definitions, the proof of Proposition~\ref{faithful}(a), that is, of
\cite[Proposition~2.15]{HuMathas:Morita}, shows that $h\in\H_{d,\b}$ acts
on~$V_\b$ as left multiplication by $\LTheta_\b(h)$.  Moreover,
\begin{equation}\label{Theta}
  \LTheta_\b(h) v_\b = v_\b\RTheta_\b(h),\qquad\text{ for all }h\in\H_{d,\b},
\end{equation}
by Lemma~\ref{vb commutation}. Typically, if $h\in\H_{d,\b}$ then we write
$h\cdot v_\b = \LTheta_\b(h)v_\b$ in what follows. Thus,
\[h\cdot v_\b = v_\b\RTheta_\b(h),\qquad\text{ for all }h\in\H_{d,\b},\]
for $h\in\H_{d,\b}$.

The following lemma introduces the elements $z_\b$. These elements play a
central role in the proofs of all of our Main Theorems from the introduction.

\begin{Lemma}\label{vb shift}
  Suppose that $\bQ$ is $(\eps,q)$-separated and let $\b\in\Comp$. Then there
  exists a unique element $z_\b$ in $\H_{d,\b}$ such that
  \[z_\b\cdot v_\b= Y_pY_{p-1}\dots Y_2Y_1v_\b=v_\b\RTheta_\b(z_\b).\]
  Moreover, $z_\b$ belongs to the centre
  of $\H_{d,\b}$.
\end{Lemma}

\begin{proof} By Proposition \ref{pleftmult}, left
multiplication by $Y_p\dots Y_2Y_1$ defines a homomorphism in
$\End_{\Hrn}\big(V_\b\big)$. Therefore, there exists a
unique element $z_\b$ in $\H_{d,\b}$ such that
\[ Y_pY_{p-1}\dots Y_2Y_1v_\b=\LTheta_\b(z_\b)v_\b=v_\b\RTheta_\b(z_\b)\]
by Proposition~\ref{faithful}(a) and (\ref{Theta}).

It remains to show that $z_\b$ is central in $\H_{d,\b}$.
As $\H_{d,\b}$ acts faithfully on $V_\b$, it is enough to show
that $\LTheta_\b(z_\b h)v_\b=\LTheta_\b(hz_\b)v_\b$, for
all $h\in\H_{d,\b}$. By Lemma~\ref{vb commutation},
\[\LTheta_\b(z_\b h)v_\b=\LTheta_\b(z_\b)\LTheta_\b(h)v_\b
        =\LTheta_\b(z_\b)v_\b\RTheta_\b(h)
        =Y_p\dots Y_2Y_1v_\b\RTheta_\b(h).
\]
Applying (the last statements in) Lemma \ref{exchange}, shows that
\[ Y_p\dots Y_2Y_1v_\b\RTheta_\b(h)=\LTheta_\b(h)Y_p\dots Y_2Y_1v_\b,
                    =\LTheta_\b(h)\LTheta_\b(z_\b)v_\b=\LTheta_{\b}(hz_{\b})v_{\b},\]
as required.
\end{proof}

\subsection{A Morita equivalence for $\H_{r,n}$}\label{Morita}
In this section we give a new description of the Morita equivalence of
Theorem~\ref{T:DMMorita} which will be useful for proving the first half of
Theorem~\ref{scalars}. In particular, in this section we will show that~$z_\b$
is an invertible element of~$\H_{d,\b}$.

By Proposition~\ref{faithful}(b), $V_\b$ is a projective
$\Hrn$-module.  Let $\Hrn(\b)$ be the smallest two-sided ideal of
$\Hrn$ which contains $V_\b=v_\b\Hrn$ as a direct summand. By
\cite[Theorem~1.1]{DM:Morita} the Morita equivalence of Theorem~\ref{T:DMMorita}
is induced by equivalences
\[ \HFun_\b:\Mod\H_{d,\b}\Morita\Mod\Hrn(\b) \]
given by $\HFun_\b(X)=X\otimes_{\H_{d,\b}}V_\b$.
Hence, by Proposition~\ref{faithful}(a) and the general theory of Morita
equivalences (\textit{cf.}~\cite[\Sect2.2]{Benson:I}), we have the following.

\begin{Lemma}[(\protect{\textit{cf.} \cite[Corollary~4.9]{DM:Morita}})]
  \label{functor}
  Suppose that $\bQ$ is $(\eps,q)$-separated in~$R$ and
  let $X$ be a right ideal of $\H_{d,\b}$. Then, as right
  $\Hrn$-modules,
  \[ \HFun_\b(X)\cong\LTheta_\b(X)V_\b.\]
\end{Lemma}

We next show that $\HFun_\b$ can be realised as induction from a subalgebra
of $\Hrn$. To do this we need to produce a subalgebra of
$\Hrn$ which is isomorphic to $\H_{d,\b}$.

Before we state this result, given a sequence $\b=(b_1,\dots,b_p)\in\Comp$
define
\begin{equation}\label{ub+}
u_\b^+(\bQ)=\LL[2]_{1,\b_1^1}\LL[3]_{1,\b_1^2}\dots\LL[p]_{1,\b_1^{p-1}}
\quad\text{and}\quad
  u_\b^-(\bQ)=\LL[p-1]_{1,\b_p^p}\dots\LL[2]_{1,\b_3^p}\LL[1]_{1,\b_2^p}.
\end{equation}
In the notation of \cite[Definition~3.1]{DJM:cyc},
$u_{\b}^+(\bQ)=u_{\bomega}^{+}$ , where
$\bomega=(\bomega^{(1)},\cdots,\bomega^{(r)})$ is the
multipartition
\[ \bomega^{(s)}=\begin{cases}
  (1^{b_{\alpha}}), &\text{if $s=d\alpha$ for some $\alpha$,}\\
(0), &\text{otherwise.}
\end{cases}
\]
Hereafter, we write $u_\b^\pm=u_\b^\pm(\bQ)$.

Taking $j=1$ and $j=p$ in Proposition~\ref{changing}, respectively, we can
write $v_\b=v_\b^+ u_\b^+ = u_\b^- v_\b^-$ where
\begin{align*}
v_\b^+&=\LL[1,p-1]_{1,b_p}T_{b_p,\b_1^{p-1}}
         \LL[1,p-2]_{1,b_{p-1}}T_{b_{p-1},\b_1^{p-2}}
         \dots\LL[1,1]_{1,b_2}T_{b_2,\b_1^{1}}\\
\intertext{and}
v_\b^-&=T_{\b_p^p,b_{p-1}}\LL[p,p]_{1,b_{p-1}}
        \dots T_{\b_3^p,b_2}\LL[3,p]_{1,b_2}
        T_{\b_2^p,b_1}\LL[2,p]_{1,b_1}.
\end{align*}

\begin{Lemma}\label{zb invertible}
 Suppose that $\bQ$ is $(\eps,q)$-separated in~$R$. Let $\b\in\Comp$. Then
 $z_\b$ is invertible in $\H_{d,\b}$.
\end{Lemma}

\begin{proof} The module $V_\b=v_\b\Hrn$ is a projective submodule of
  $\Hrn$-module by Proposition~\ref{faithful}(b), so
  $V_\b=e\Hrn$ for some idempotent $e\in\Hrn$. Therefore,
  $V_\b=e\Hrn=e^2\Hrn\subseteq e\Hrn e\Hrn=V_\b^2\subseteq e\Hrn=V_\b$ so that
  $V_\b=V_\b^2$. Therefore, using the formulae for
  $v_\b$ given before the Lemma,
  \begin{align*}
    V_\b&= (V_\b)^2 = v_\b \Hrn v_\b\Hrn
         =v_\b^+\big(u_\b^+\Hrn u_\b^-\big) v_\b^-\Hrn\\
        &= v_\b^+\big(u_\b^+T_\b u_\b^-\H_q(\Sym_\b) \big)v_\b^-\Hrn,
  \end{align*}
  where the last equality follows by Du and
  Rui~\cite[Proposition~3.1(a)]{DuRui:akmorita}.
  Lemma~\ref{commutes} shows that $\H_q(\Sym_b)v_\b^-=v_\b^-\H_q(\Sym_\bp)$.
  Hence,
  \[ V_\b  = v_\b T_\b u_\b^- v_\b^-\H_q(\Sym_\bp)\Hrn
           \subseteq v_\b T_\b v_\b\Hrn = z_\b\cdot V_\b\subseteq V_\b,\]
  by Proposition~\ref{pleftmult} and Lemma~\ref{vb shift}. Therefore, 
  $V_\b=z_\b\cdot V_\b$ so the endomorphism of $V_\b$ given by left
  multiplication by $z_\b$ has a right inverse in $\End_{\Hrn}(V_\b)$.
  Consequently, $z_\b$ has a right inverse in $\H_{d,\b}$ by
  Proposition~\ref{faithful}(a). Hence, $z_\b$ is invertible in $\H_{d,\b}$
  since it is central.
\end{proof}

\begin{cor} \label{shiftiso} Let $\b\in\Comp$, $\blam\in\Part[d,\b]$ and $\t\in\Z$. Suppose that $\bQ$ is $(\eps,q)$-separated over the field $K$. Then
$V_{\b\<t\>}^{(t)}\cong V_{\b\<t+1\>}^{(t+1)}$.
\end{cor}

\begin{proof} It is enough to consider the case where $t=0$. By 
  Lemma~\ref{vb shift} left multiplication by $Y_1$ induces an $\Hrn$-module
  homomorphism from $V_\b$ to $V_{\b\<1\>}^{(1)}$. This map is an isomorphism
  because left multiplication by $Y_p\dots Y_1$ is invertible by Lemma~\ref{zb
  invertible} (and Lemma~\ref{vb shift}).
\end{proof}

Under the conditions of Lemma~\ref{zb invertible} we can make the following
definition.

\begin{Defn}
  Suppose that $\b\in\Comp$ and that $\bQ$ is $(\eps,q)$-separated in~$R$.
  Let $e_\b=z_\b^{-1}\cdot v_\b T_\b\in V_\b$ and define
  \[\hatHdb = \set{h\cdot e_\b|h\in\H_{d,\b}}
         =\set{e_\b\RTheta_\b(h)|h\in\H_{d,b}}\subseteq V_b.\]
\end{Defn}

Quite surprisingly, $\hatHdb$ is something like a parabolic
subalgebra of $\Hrn$.

\begin{Theorem}\label{subalgebra}
    Suppose that $\b\in\Comp$ and that $\bQ$ is $(\eps,q)$-separated. Then:
    \begin{enumerate}
      \item $e_\b$ is an idempotent in $\Hrn$ and
      $V_\b=e_\b\Hrn$.
      \item $\hatHdb$ is a unital subalgebra of $\Hrn$ with identity
        element $e_\b$.
      \item The map $\H_{d,\b}\longrightarrow\hatHdb; h\mapsto h\cdot e_\b$ is
        an algebra isomorphism.
    \end{enumerate}
\end{Theorem}

\begin{proof} Suppose that $x,y\in\H_{d,\b}$. Then using the
  definitions, (\ref{Theta}) and Lemma~\ref{vb shift} we have that
  \begin{align*}
     (x\cdot e_\b)(y\cdot e_\b)
            &= (xz_\b^{-1}\cdot v_\b T_\b)(yz_\b^{-1}\cdot v_\b T_\b)
            =xz_\b^{-1}\cdot v_\b T_\b v_\b\RTheta_\b(yz_\b^{-1})T_\b\\
           &=xz_\b^{-1}z_\b\cdot v_\b\RTheta_\b(yz_\b^{-1})T_\b
            =x\cdot v_\b\RTheta_\b(yz_\b^{-1})T_\b\\
           &=xyz_\b^{-1}\cdot v_\b T_\b
            =(xy)\cdot e_\b.
   \end{align*}
   Taking $x=y=1_{\H_{d,\b}}$ shows that $e_\b$ is an idempotent in
   $\Hrn$. As $\H_{d,\b}$ acts faithfully on $V_\b$ by
   Proposition~\ref{faithful}(a), all of the claims now follow.
\end{proof}

Theorem~\ref{subalgebra} says that the natural inclusion map
$\RTheta_\b:\H_{d,\b}\hookrightarrow\Hrn$ is an inclusion of algebras when it is
composed with left multiplication by $e_\b$. Note that, in general, the image of
$\RTheta_\b$ is \textit{not} a subalgebra of $\Hrn$.

Combining Theorem~\ref{subalgebra} and Lemma~\ref{functor} gives a second
description of the Morita equivalence $\HFun_\b$.  If $A$ is a subalgebra
of an algebra $B$ then let $\Ind_A^B$ be the corresponding induction
functor.

\begin{cor}
    Suppose that $\bQ$ is $(\eps,q)$-separated and that $X$ is a right
    $\H_{d,\b}$ module, where $\b\in\Comp$. Then
    \[ \HFun_\b(X)\cong\ind_{\hatHdb}^{\Hrn}({(X\cdot e_\b)})
    =X\cdot e_\b\otimes_{\hatHdb}\Hrn.\]
\end{cor}

\subsection{Comparing trace forms on $V_\b$}\label{S:trace}
Theorem~\ref{subalgebra} shows how to realize $\H_{d,\b}$ as a subalgebra of
$\Hrn$. The aim of this section is to use this result to prove a comparison
theorem for the natural trace forms on~$\Hrn$ and~$\H_{d,\b}$. This is one of
the key steps in proving Theorem~\ref{scalars} from the introduction because
the Schur element~$s_\blam$ from~\eqref{E:SchurElements} can be computed from the
trace of the primitive idempotents for the Specht module~$S(\blam)$ in the
semisimple case.

Recall that a trace form on an $R$-algebra $A$ is a linear map $\tr\map AR$ such
that $\tr(ab)=\tr(ba)$, for all $a,b\in A$. The form $\tr$ is non-degenerate if
whenever $0\ne a\in A$ then $\tr(ab)\ne0$ for some $b\in A$.

By \cite{MM:trace} the Hecke algebras $\H_{d,\b}$ and $\Hrn$ are both
equipped with `canonical' non-degenerate trace forms $\Tr_\b$ and
$\Tr$, respectively. The aim of this subsection is to compare these two
trace forms. More precisely, we show that
\[\Tr(h\cdot v_\b T_\b) = \Tr_\b(h)\Tr(v_\b T_\b),\]
for all $h\in\H_{d,\b}$. This result
will be used in the next section to compute the scalar~$\fscal$ from
the introduction.

The trace form $\Tr\map{\Hrn}R$ on $\Hrn$
is the $R$-linear map determined by
\begin{equation}\label{trace}
\Tr(L_1^{a_1}\dots L_n^{a_n}T_xT_y)
=\begin{cases}q^{\ell(x)}, &\text{if }a_1=\dots=a_n=0\text{ and }x=y^{-1},\\
        0,&\text{otherwise}.
    \end{cases}
\end{equation}
(This equation completely determines $\Tr$ by Lemma~\ref{L:AKBasis}.) The trace
form $\Tr_\b$ on~$\H_{d,\b}$ is defined similarly. Comparing these two trace
forms requires more technical calculations with the elements~$v_\b$.

Before Lemma~\ref{zb invertible} we noted that
$v_\b=v_\b^+ u_\b^+$, for some element $v_\b^+$. To compare the trace forms $\Tr$ and $\Tr_\b$
we need a different expression for $v_\b^+$. To state this,
let~$\LH_m$ be the $R$-submodule of $\Hrn$ spanned by the elements
\[\set{T_wL_1^{a_1}\dots L_{m-1}^{a_{m-1}}|0\le a_1,\dots,a_{m-1}<r
                \text{ and } w\in\Sym_m}.\]
Note that $\LH_m$ is not, in general, a subalgebra of $\Hrn$.

The proof of the next result is not particularly pretty so we defer it until
Lemma~\ref{vb-II}. Recall from Section~\ref{S:zb} that $\bp=(b_p,\dots,b_1)$ if
$\b=(b_1,\dots,b_p)$.

\begin{Lemma}\label{vb-}
  Suppose that $\b\in\Comp$. Then
  \[v_\b^+=T_{\bp}\Big(
  \LL[1]_{\b_1^1+1,n}\LL[2]_{\b_1^2+1,n}\dots\LL[p-1]_{\b_1^{p-1}+1,n}
          +\sum_{l=1}^{p-1}\sum_{m=\b_1^{\,l}+1}^{\b_1^{l+1}}\sum_{e=1}^{dl}
             h_{l,m,e}L_m^e\Big)\]
  for some $h_{l,m,e}\in\LH_m$.
\end{Lemma}

Using this result we can prove the promised comparison theorem for $\Tr$ and $\Tr_\b$.

\begin{Theorem}\label{comparison}
    Suppose that $b\in\Comp$. Then
    \[ \Tr(h\cdot v_\b T_\b)= \Tr_\b(h)\Tr(v_\b T_\b),\\ \]
    for all $h\in\H_{d,\b}$.
\end{Theorem}

\begin{proof} By linearity, it is enough to let $h$ run over a basis of
  $\H_{d,\b}$. Let
  \[\AKBasis_\b=\set{L_1^{a_{1,1}}\dots L_{b_1}^{a_{1,b_1}}T_{x_1}\otimes
     \dots\otimes L_1^{a_{p,1}}\dots L_{b_p}^{a_{p,b_p}}T_{x_p}|
      0\le a_{i,t}<d\text{ and }x_t\in\Sym_{b_t}}\]
      be the basis of $\H_{d,\b}$ from Lemma~\ref{L:AKBasis}. Then it is enough to
      show that
   \[\Tr(h\cdot v_\b T_\b)=\Tr_\b(h)\Tr(v_\b T_\b),
                 \qquad\text{for all $h\in\AKBasis_\b$.}\]
  If $h=1_{\H_{d,\b}}$ there is nothing to prove. Therefore, by
  (\ref{trace}) it remains to show that $\Tr(h\cdot v_\b T_\b)=0$ whenever
  $1_{\H_{d,\b}}\ne h\in\AKBasis_\b$. For the rest of the proof fix such
  an~$h$. Write
  $h=L_1^{a_{1,1}}\dots L_{b_1}^{a_{1,b_1}}T_{x_1}\otimes
     \dots\otimes L_1^{a_{p,1}}\dots L_{b_p}^{a_{p,b_p}}T_{x_p}$,
  where $0\le a_{j,t}<d$ and $x_t\in\Sym_{b_t}$, and set
  \[h'=\RTheta_\b(h)=L_1^{a_{1,1}}\dots
  L_{\b_1^1}^{a_{1,b_1}}L_{\b_1^1+1}^{a_{2,1}}\dots
  L_{\b_1^2}^{a_{2,b_2}}\dots L_{\b_1^{p-1}+1}^{a_{p,1}}\dots
  L_{\b_1^p}^{a_{p,b_p}}T_x,\]
  where $x=x_1 x_2^{\shift{\b_1^1}}\dots x_p^{\shift{\b_1^{p-1}}}$.

  Recall from before Lemma~\ref{zb invertible} that $v_\b=v_\b^+ u_\b^+$.
  Therefore, using Lemma~\ref{vb-} and the fact that $\Tr$ is a trace form,
  \begin{align*}
  \Tr(h\cdot v_\b T_\b)&=\Tr(v_\b h' T_\b)=\Tr(v_\b^+u_\b^+h' T_\b)\\
      &=\Tr(T_\bp \hat u^-_\b u^+_\b h' T_\b)
          +\sum_{l=1}^{p-1}\sum_{m=\b_1^{\,l}+1}^{\b_1^{l+1}}\sum_{e=1}^{dl}
          \Tr(T_\bp h_{l,m,e}L_m^eu^+_\b h'T_\b)\\
      &=\Tr(\hat u^-_\b u^+_\b h' T_\b T_\bp)
          +\sum_{l=1}^{p-1}\sum_{m=\b_1^{\,l}+1}^{\b_1^{l+1}}\sum_{e=1}^{dl}
          \Tr(L_m^eu^+_\b h' T_\b T_\bp h_{l,m,e}),
  \end{align*}
  where $h_{l,m,e}\in\LH_m$ and
  $\hat u^-_{\b}:=\LL_{\b_1^1+1,n}^{(1)}\LL_{\b_1^2+1,n}^{(2)}\dots
         \LL_{\b_1^{p-1}+1,n}^{(p-1)}$.
  Fix a triple $(l,m,e)$, from the sum, with
  $1\le l<p$, $\b_1^{\,l}<m\le\b_1^{l+1}$ and $1\le e\le dl$. By assumption,
  $L_m$ appears in $h'$ with exponent $0\le a_{l+1,m'}<d$, where
  $m=\b_1^{\,l}+m'$. Therefore, $L_m^eu^+_\b h' T_\b T_\bp
  h_{l,m,e}$ is a linear combination of terms of the form
  $L_m^eu^+_\b f_1(L)T_w f_2(L)$, where $w\in\Sym_n$, $f_1(L)$ is a
  a polynomial in $L_1,\dots,L_n$ of degree at most $a_{l+1,m'}<d$ as a
  polynomial in $L_m$, and where $f_2(L)$ is a polynomial in
  $L_1,\dots,L_{m-1}$. As $\Tr$ is a trace form,
  \[\Tr(L_m^eu^+_\b f_1(L)T_w f_2(L))
        =\Tr(f_2(L)L_m^eu^+_\b f_1(L)T_w).\]
  Now, considered as a polynomial in $L_m$, $f_2(L)L_m^eu^+_\b f_1(L)$ is a
  polynomial with zero constant term (since $e>0$) and degree
  \[0<f:=e+d(p-l-1)+a_{l+1,m'}< d(p-1)+d=r.\] By the same argument, if
  $m<k\le n$ then $L_k$ appears in $f_2(L)L_m^eu^+_\b f_1(L)$ with exponent
  at most $d(p-l_k'-1)+a_{l_k+1,k'}< d(p-1)<r$, where $k=\b_1^{l_k-1}+k'$ and
  $1\le k'\le b_{l_k}$ . If $k<m$ then $L_k$ could appear in
  $f_2(L)L_m^eu^+_\b f_1(L)$ with exponent greater than $r-1$, however,
  by Lemma~\ref{JM properties} this will not affect the exponents of
  $L_m,\dots,L_n$ when rewrite this term as a linear combination of
  Ariki-Koike basis elements. Hence,  $L_m^f$ is a left divisor of
  $f_2(L)L_m^eu^+_\b f_1(L)$ when it is
  written as a linear combination of Ariki-Koike basis elements.
  Consequently, $\Tr(f_2(L)L_m^eu^+_\b f_1(L))=0$ by (\ref{trace}).
  Therefore, $\Tr(L_m^eu^+_\b h' T_\b T_\bp h_{l,m,e})=0$  so that
  $\Tr(h\cdot v_\b h' T_\b)=\Tr(v_\b h' T_\b)
                   =\Tr(\hat u^-_\b u^+_\b h' T_\b T_\bp)$.

  Now consider $\Tr(\hat u^-_\b u^+_\b h' T_\b T_\bp)$. By definition,
  \begin{align*}
   \hat u^-_\b u^+_\b  h'
      &=\LL[1]_{\b_1^1+1,n}\LL[2]_{\b_1^2+1,n}\dots\LL[p-1]_{\b_1^{p-1}+1,n}
       \cdot\LL[2]_{1,\b_1^1}\LL[3]_{1,\b_1^2}\dots\LL[p]_{1,\b_1^{p-1}}h'\\
      &=\prod_{i=1}^p\LL[i]_{1,\b_1^{i-1}}\LL[i]_{\b_1^i+1,n}\cdot h'.
  \end{align*}
  If $a_{l,m'}\ne0$, for some $l$ and $m'$, then $L_m^{a_{l,m'}}$ divides
  $h'$, where $m=\b_1^{l-1}+m'$ as above. By the argument above
  $\hat u^-_\b u^+_\b h'$, when considered as a polynomial in $L_m$, is a
  polynomial with zero constant term and degree strictly less than $r$.
  Therefore,
  \[\Tr(h\cdot v_\b T_\b)=\Tr(\hat u^-_\b u^+_\b h' T_\b T_\bp)
         =0,\]
  as required. It remains, then, to consider the cases when $a_{l,m'}=0$, for
  $1\le l\le p$ and $1\le m'\le b_l$. That is, when $h'=T_x$ for some
  $1\ne x\in\Sym_\b$. By (\ref{trace}), in this case we have
  \begin{equation}\label{almost trace}
    \Tr(h\cdot v_\b T_\b) = \Tr(\hat u^-_\b u^+_\b T_x T_\b T_\bp)
                           = \Tr(\hat u^-_\b u^+_\b)\Tr(T_x T_\b T_\bp)
  \end{equation}
  Recall that $w_\b$ is a distinguished coset
  representative for $\Sym_\b$, so that
  $\ell(xw_\b)=\ell(x)+\ell(w_\b)$.
  Therefore,  $\Tr(T_xT_\b T_\bp)=\Tr(T_{xw_\b}T_\bp)=0$ by (\ref{trace})
  since $x\ne1$.
  Hence,  $\Tr(h\cdot v_\b T_\b)=0$, completing the proof.
\end{proof}

We can improve on Theorem~\ref{comparison} by explicitly
computing $\Tr(v_\b T_\b)$. In fact, in proving the theorem we have
essentially already done this. To state the result, given
$\b\in\Comp$ set $\alpha(\b)=\sum_{i=1}^p ib_i\in\N$.

\begin{cor}\label{vb trace}
    Suppose that $b\in\Comp$. Then
    \[\Tr(v_\b T_\b)
       =(-1)^{dn(p-1)}{q}^{\ell(w_\b)}\eps^{\frac12rn(p-1)-d\alpha(\b)}
               (Q_1\dots Q_d)^{n(p-1)}.\]
\end{cor}

\begin{proof}By (\ref{almost trace}), and (\ref{trace}), we have that
  \[\Tr(v_\b T_\b)=\Tr(1_{\H_{d,\b}}\cdot v_\b T_\b)
                  =\Tr(\hat u^-_\b u^+_\b)\Tr(T_\bp T_\b)
                  =q^{\ell(w_\b)}\Tr(\hat u^-_\b u^+_\b).\]
  Now, $\Tr(\hat u^-_\b u^+_\b)$ is just the
  constant term of $\hat u^-_\b u^+_\b$ by (\ref{trace}). Therefore,
  \begin{align*}
  \Tr(v_\b T_\b)&=q^{\ell(w_\b)}\prod_{t=1}^p\Big((-1)^d\eps^{td}
                            Q_1\dots Q_d\Big)^{n-b_t}\\
    &=(-1)^{dn(p-1)}q^{\ell(w_\b)}\eps^{\frac12rn(p-1)-d\alpha(\b)}
               ( Q_1\dots Q_d)^{n(p-1)},
  \end{align*}
  since $b_1+\dots+b_p=n$.
\end{proof}

\begin{Remark}Suppose that $\b\in\Comp$. Then it is not difficult to see that
  \[\ell(w_\b)=\sum_{1\le i<j\le p}b_ib_j.\]
\end{Remark}

\section{Specht modules and simple modules for $\Hrpn$}\label{C:Hrpn}
In this chapter we will introduce analogues of the Specht modules for~$\Hrpn$
and hence construct a complete set of irreducible~$\Hrpn$-modules. To do this we
first use the results of the last chapter to prove Theorem~\ref{scalars} from
the introduction. The first step is easy as Schur's Lemma easily implies that
the element $z_\b$ acts on the Specht module $S(\blam)$ of~$\Hrn$ as
multiplication by a scalar~$\fscal$ (Proposition~\ref{scalar}). Using
Theorems~\ref{subalgebra} and~\ref{comparison} we then compute~$\fscal$
explicitly in terms of the Schur elements of~$\H_{d,\b}$ and~$\Hrn$
(Theorem~\ref{flam two}). We show that the scalar~$\fscal$ has a $p_\blam$th
root and so complete the proof of Theorem~\ref{scalars}.

A calculation using seminormal forms and `shifting homomorphisms' shows that,
representation theoretically, taking the $p_\blam$th root of $\fscal$
corresponds to the existence of an $\H_{r,p_\blam,n}$-module endomorphism
$\theta_\blam$ of $S(\blam)$ such that $\theta_\blam^{p_\blam}$ is a scalar
multiple of the identity map on $S(\blam)$. As an $\Hrpn$-module, the Specht
module~$S(\blam)$ then decomposes into a direct sum \[S(\blam) = S^\blam_1\oplus
S^\blam_2\oplus\dots\oplus S^\blam_{p_\blam}\] of eigenspaces
for~$\theta_\blam$. The modules $S^\blam_t$, for $1\le t\le p_\blam$, play the
role of Specht modules for~$\Hrpn$. Using some Clifford theory, we show in
Theorem~\ref{Hrpn simples} that every irreducible $\Hrpn$-module arises as the
simple head of some $S^\blam_t$ in a unique way, up to cyclic shift. These
results complete the proof of Theorem~\ref{main simple} from the introduction.

Section~\ref{S:Clifford} marks the real appearance of the Hecke algebras~$\Hrpn$
of type $G(r,p,n)$ as, up until now, we have worked exclusively with
$\Hrn$-modules. In fact, we do not really start working with~$\Hrpn$ until
section~\ref{S:HrpnSpecht} where we construct Specht modules and simples modules
for~$\Hrpn$.

\subsection{Specht modules for $\H_{d,\b}$ and $\Hrn$}\label{S:Specht}
The algebras
$\H_{d,\b}$ and $\Hrn$ are both cellular algebras~\cite{GL,DJM:cyc} with
the cell modules of both algebras being called Specht modules. In this
section we quickly recall the construction of these modules and
the relationship between the Specht modules of these algebras.

First, recall that a \textbf{partition} of $n$ is a sequence
$\lambda=(\lambda_1,\lambda_2,\dots)$ of weakly decreasing
non-negative integers which sum to~$|\lambda|=n$. The \textbf{conjugate} of
$\lambda$ is the partition $\lambda'=(\lambda_1',\lambda_2',\dots)$, where
$\lambda_i'=\#\set{j\ge1|\lambda_j\ge i}$.

An \textbf{$r$-multipartition} of $n$ is an ordered $r$-tuple
$\blam=(\lambda^{(1)},\dots,\lambda^{(r)})$ of partitions such that
$|\lambda^{(1)}|+\dots+|\lambda^{(r)}|=n$. Let $\Part$ be the set of $r$-multipartitions of
$n$.  The partitions $\lambda^{(s)}$ are the \textbf{components} of $\blam$ and we call
$\blam$ a multipartition when~$r$ is understood. If
$\blam=(\lambda^{(1)},\dots,\lambda^{(r)})$ is a multipartition then its \textbf{conjugate}
is the multipartition $\blam'=({\lambda^{(r)}}',\dots,{\lambda^{(1)}}')$. To each
multipartition $\blam$ we also associate a Young subgroup
$\Sym_\blam=\Sym_{\lambda^{(1)}}\times\dots\times\Sym_{\lambda^{(r)}}$ of $\Sym_n$ in the
obvious way.

The \textbf{diagram} of $\blam$ is the set
$[\blam]=\set{(i,j,s)|1\le j\le \lambda^{(s)}_i\text{ and }1\le s\le r}$.
A \textbf{$\blam$-tableau} is a
map $\t\map{[\blam]}\{1,2,\dots,n\}$, which we think of as a labeling
of the diagram of $\blam$. Thus we write
$\t=(\t^{(1)},\dots,\t^{(r)})$ and we talk of the rows, columns and
components of $\t$. Let $\Std(\blam)$ be the set of standard $\blam$-tableaux.

By \cite[Theorem~3.26]{DJM:cyc}, $\Hrn$ is a cellular algebra with a
cellular basis of the form
\[\set{m_{\s\t}|\s,\t\in\Std(\blam),\text{ for }\blam\in\Part}.\]
Hence, the cell modules of $\Hrn$ are indexed by $\Part$ and if
$\blam\in\Part$ then the corresponding cell module $S(\blam)$ has a basis
of the form $\set{m_\t|\t\in\Std(\blam)}$.

Recall from the introduction that in $\blam\in\Part$ is a multipartition then
$\blam^{[t]}=(\blam^{(dt-d+1)},\blam^{(dt-d+2)},\dots,\blam^{(dt)})$, for
$1\le t\le p$. More generally, set $\blam^{[t+kp]}=\blam^{[t]}$, for $k\in\Z$.

\begin{Defn}\label{Specht modules}
  \begin{enumerate}
    \item Suppose that $\blam\in\Part$. Then the \textbf{Specht module}
      $S(\blam)$ for $\Hrn$ is the cell module indexed by $\blam$
      defined in \cite[Definition~3.28]{DJM:cyc}.
    \item Suppose that $\blam\in\Part[d,\b]$. Then the \textbf{Specht module}
      for $\H_{d,\b}$ is the module
      $S_\b(\blam)\cong S(\blam^{[1]})\otimes\dots\otimes S(\blam^{[p]})$.
  \end{enumerate}
\end{Defn}

We write $S^R(\blam)$ when we want to emphasize that $S(\blam)$ is an
$R$-module. We will give an explicit construction of these modules in
Section~\ref{S:Twist}.

When $\H_{d,\b}$ is semisimple the modules $\set{S_\b(\blam)|\blam\in\Part[d,\b]}$ give a complete set of
pairwise non-isomorphic simple $\H_{d,\b}$-modules. Similarly, the modules
$\set{S(\blam)|\blam\in\Part[r,n]}$ give a complete set of
pairwise non-isomorphic simple $\Hrn$-modules when $\Hrn$ is
semisimple.

More generally, by the general theory of cellular algebras~\cite{GL}, each
Specht module $S(\blam)$ comes with an associative bilinear form and the radical
$\rad S(\blam)$ of this form is an $\Hrn$-module. Define $D(\blam)=S(\blam)/\rad
S(\blam)$. A multipartition $\blam\in\Part$ is \textbf{Kleshchev} if
$D(\blam)\ne0$.  Let $\Klesh(\bbQ)=\set{\blam\in\Part|D(\blam)\ne0}$ be the set
of Kleshchev multipartitions of~$n$. Then
\[\set{D(\blam)|\blam\in\Klesh(\bbQ)}\]
is a complete set of pairwise
non-isomorphic irreducible $\Hrn$-modules. Typically we write
$\Klesh=\Klesh(\bbQ)$ in what follows.

If $A$ is an algebra and $M$ is an $A$-module let $\head(M)$ be the
\textbf{head} of~$M$. That is, $M$ is the largest semisimple quotient of~$M$.
For example, by~\cite{DJM:cyc}, if $\blam\in\Klesh$ then
$D(\blam)=\head(S(\blam))$. If $S$ and $D$ are modules for an algebra, with $D$
irreducible, let $[S:D]$ be the multiplicity of $D$ as a composition factor
of~$S$.

If $\blam$ and $\bmu$ are two multipartitions then $\blam$
\textbf{dominates} $\mu$, and we write $\blam\gedom\bmu$ if
\[\sum_{s=1}^{t-1}|\blam^{(s)}|+\sum_{j=1}^i\blam^{(t)}_j
  \ge \sum_{s=1}^{t-1}|\bmu^{(s)}|+\sum_{j=1}^i\bmu^{(t)}_j\]
for $1\le t\le r$ and $i\ge0$. We write $\blam\gdom\bmu$ if
$\blam\gedom\bmu$ and $\blam\ne\bmu$. The dominance partial order 
on~$\Part$ is useful because of the following fact.

\begin{Lemma}[(\protect{\cite[\Sect3]{DJM:cyc}})]\label{dominance}
  Suppose that $[S(\blam):D(\bmu)]\ne0$, for $\blam\in\Part$ and $\bmu\in\Klesh$.
  Then $\blam\gedom\bmu$. Moreover, if $\bmu\in\Klesh$ then
  $[S(\bmu):D(\bmu)]=1$ and $D(\bmu)=\head S(\bmu)$.
\end{Lemma}

Let $\Klesh[d,\b]=\set{\blam\in\Part[d,\b]|
      \blam^{[t]}\in\Klesh[d,b_t](\eps^t\bQ)\text{ for }1\le t\le p}$.
If $\blam\in\Klesh[d,\b]$ let
\[D_\b(\blam)=S_\b(\blam)/\rad S_\b(\blam)
        \cong D(\blam^{[1]})\otimes\dots\otimes D(\blam^{[p]}).\]
Theorem~\ref{T:DMMorita} and the remarks above imply that
$\set{D_\b(\blam)|\blam\in\Klesh[d,\b]}$ is a complete set of pairwise
non-isomorphic irreducible $\H_{d,\b}$-modules.

Recall the functor $\HFun_\b$ from \Sect\ref{Morita}. By
\cite[Proposition~4.11]{DM:Morita} (see also \cite[Proposition~2.13]{HuMathas:Morita}),
we have the following.

\begin{Lemma}\label{Morita Specht}
  Suppose that $\blam\in\Part[d,\b]$. Then
  \begin{enumerate}
      \item $\HFun_\b\big(S_\b(\blam)\big)\cong S(\blam)$ as $\Hrn$-modules.
      \item $\HFun_\b\big(D_\b(\blam)\big)\cong D(\blam)$ as $\Hrn$-modules.
      \item $\blam=(\blam^{[1]},\dots,\blam^{[p]})\in\Klesh[d,\b](\bbQ)$ is Kleshchev if and
      only if $\blam^{[t]}\in\Klesh[d,b_t](\eps^t\bQ)$, for $1\le
      t\le p$.
  \end{enumerate}
\end{Lemma}

In particular, we can consider
$S(\blam)\cong\HFun_\b\big(S_\b(\blam)\big)=S_{\b}(\blam)\cdot V_\b$ to be a
submodule of $V_\b$.

\subsection{The scalar $\fscal$}\label{ssec32}
It follows from Schur's Lemma the central element $z_\b$ acts on the Specht modules
$S(\blam)$ as multiplication by a scalar, for $\blam\in\Part[d,\b]$. In this
section we explicitly compute this scalar, thus proving half of
Theorem~\ref{scalars} from the introduction.

We define
$\mA=\Z[\dot\eps,\dot{q}^{\pm1},\dot Q_1^{\pm 1},\dots,\dot Q_d^{\pm1},
                          A(\dot{\eps},\dot{q},\dot{\bQ})^{-1}]$,
where $\dot\eps$ is a primitive $p$th root of unity in $\mathbb{C}$ and
$\dot q$ and $\dot\bQ=(\dot Q_1,\dots,\dot Q_d)$ are indeterminates over
$\Z[\dot\eps]$. Let $\mF$ be the field of fractions of $\mA$. If $\bQ$ is
$(\eps,q)$-separated over~$R$ then $R$ can be considered as an $\mA$-module
by letting $\dot\eps$ act on $R$ as multiplication by $\eps$, $\dot q$ act
as multiplication by $q$ and $\dot Q_i$ act as multiplication by
$Q_i$, for $1\le i\le d$.  Therefore,
$\Hrn^R(q,\bQ)\cong\Hrn^{\mA}(\dot q,\dot\bQ)\otimes_{\mA} R$ are
isomorphic $R$-algebras. In particular,
$\HF\cong\Hrn^\mA(\dot q,\dot\bQ)\otimes_{\mA}\mF$.  The algebra $\HF$ is
semisimple by Ariki's semisimplicity criteria~\cite{Ariki:ss}. The algebra
$\HF$ is split semisimple because $\HF$ is a cellular algebra (and every
field is a splitting field for a cellular algebra; see~\cite[Theorem~3.4]{GL}).

Abusing notation, we call the elements of $\mA$ \textit{polynomials} and if
$f(\DoteqQ)\in\mA$ then we define
$f(\eps,q,\bQ)=f(\DoteqQ)\cdot 1_R$ to be the value of
$f(\DoteqQ)$ at $(\eps,q,\bQ)$.

The scalar $\fscal$ in the next Proposition plays a key role in the proofs of
all of our main results, Theorems A--D, from the introduction.

\begin{prop} \label{scalar}
  Suppose that $\bQ$ is $(\eps,q)$-separated in $R$ and that $\b\in\Comp$
  and $\blam\in\Part[d,\b]$. Then there exists a non-zero
  scalar $\fscal\in R$ such that
  \[z_\b\cdot x=\fscal x,\]
  for all $x\in S(\blam)$. Moreover, there exists a non-zero polynomial
  $\dfscal=\fscal(\DoteqQ)\in\mA$ such that
  $\fscal=\dfscal(\eps,q,\bQ)\in R$.
\end{prop}

\begin{proof}The Specht module $S_\b(\blam)$ is free as an $R$-module so,
  by the remarks above, $S_\b(\blam)\cong S_\b^{\mA}(\blam)\otimes_{\mA}
  R$. Therefore, to show that such a scalar exists it is enough to
  consider the case when $R={\mA}$.  Similarly, since $S_\b^{\mA}(\blam)$
  embeds into $S_\b^\mF(\blam)\cong S_\b^{\mA}(\blam)\otimes_{\mA}\mF$ we
  may assume that $R=\mF$. By the remarks above, the algebra
  $\H_{d,\b}^\mF$ is split semisimple and the module $S_\b^\mF(\blam)$ is
  an irreducible $\H_{d,\b}^\mF$-module, so by Schur's Lemma  the
  homomorphism of $S_\b(\blam)$ given by left multiplication by~$z_\b$ is
  equal to multiplication by some scalar $\dfscal$. Notice that $\dfscal$
  is an element of $\mA$ because $z_\b\cdot v_\b T_\b\in\H^\mA_{r,n}$. By
  specialization, the scalar $\fscal\in R$ in the
  statement of the Lemma is given by evaluating the polynomial
  $\dfscal(\DoteqQ)$ at $(\eps,q,\bQ)$. Finally, observe that $\fscal\ne0$
  since $z_\b$ acts invertibly on $V_\b$ by Lemma~\ref{zb invertible}.
\end{proof}

We will determine the scalar $\fscal\in R$ by computing the polynomial
$\dfscal$ in~$\mA$. In fact, we have already done all of the work
needed to determine $\dfscal$. To describe $\dfscal$ we only need
one definition.

Abusing notation slightly, let $\Tr$ be the trace form on $\HF$ given by
(\ref{trace}).  Let~$\chi^\blam$ be the character of $S^\mF(\blam)$, for
$\blam \in \Part$. Then $\set{\chi^\blam|\blam\in\Part}$ is a complete set
of pairwise inequivalent irreducible characters for $\HF$, so it is a basis
for the space of trace functions on $\HF$. In particular, $\Tr$ can be
written in a unique way as a linear combination of the irreducible
characters. Moreover, it is easy to see that every character $\chi^\blam$
must appear in $\Tr$ with \textit{non-zero coefficient} because $\Tr$ is
non-degenerate; see, for example, \cite[Example~7.1.3]{GeckPfeiffer:book}.
Consequently, the following definition makes sense.

\begin{Defn}\label{schur element}
  The \textbf{Schur elements} of $\HF$ are the scalars
  $\sscal=\sscal(\DoteqQ)\in\mF$, for $\blam\in\Part$, such that
  \[\Tr=\sum_{\blam\in\Part}\frac1{\sscal}\chi^\blam.\]
\end{Defn}

For $\blam\in\Part$ fix $F_\blam$ a primitive idempotent in $\HF$ such that
$F_\blam\HF\cong S^\mF(\blam)$. Using, for example seminormal forms
$\HF$~\cite[Theorem~2.11]{M:gendeg}, it is easy to see that
$\chi^\blam(F_\bmu)=\delta_{\blam\bmu}$, for $\blam,\bmu\in\Part$. Hence, a
second characterisation of the Schur elements is that
\begin{equation}\label{E:SchurElts}
\sscal=\frac1{\Tr(F_\blam)}.
\end{equation}

Similarly, for each $\blam\in\Part[d,\b]$ the trace form $\Tr_\b$
determines Schur elements $\sscal[\b]\in\mF$ for $\HF[d,\b]$.  By the remarks above, the Schur elements of
$\HF[d,\b]$ satisfy
\[\sscal[\b]=\prod_{t=1}^p
                    \sscalt(\dot\eps,\dot q,{\dot\eps}^t\dot\bQ)
                =\frac1{\Tr_\b(F_\b(\blam))},\]
where $F_\b(\blam)$ is a primitive idempotent in $\HF[d,\b]$ such that
$S_\b^\mF(\blam)\cong F_\b(\blam)\HF[d,\b]$.

\begin{Theorem}\label{flam two}
    Suppose that $\b\in\Comp$ and that $\blam\in\Part[d,\b]$. Then
    \[
    \dfscal
      =\frac{\sscal}{\sscal[\b]}\Tr(v_\b T_\b).
    \]
    Consequently,
    $\dfscal=(-1)^{n(r-d)}\dot{q}^{\ell(w_\b)}\dot
       \eps^{\frac12rn(p-1)-d\alpha(\b)}(\dot Q_1\dots\dot Q_d)^{n(p-1)}
            \dfrac{\sscal}{\sscal[\b]}$.
\end{Theorem}

\begin{proof}
   To compute $\dfscal$ we may assume that
   $R=\mF$ and work in $\HF$. Let $F_\b(\blam)$ be a primitive
   idempotent in $\HF[d,\b]$ such that
   $S_\b^\mF(\blam)\cong F_\b(\blam)\HF[d,\b]$. Then
   $F_\b(\blam)\cdot e_\b$ is a primitive idempotent in $\HF$ such that
   $F_\b(\blam)\cdot e_\b\HF\cong S^\mF(\blam)$ by
   Theorem~\ref{subalgebra} and Lemma~\ref{Morita Specht}. Therefore,
   using the remarks above,
   \begin{align*}
   \frac1{\sscal}&= \Tr(F_\b(\blam)\cdot e_\b)
    =\Tr(z_\b^{-1}F_\b(\blam)\cdot v_\b T_\b),
                 &&\text{since $z_\b$ is central in $\H_{d,\b}$},\\
   &=\frac1{\dfscal}\Tr(F_\b(\blam)\cdot v_\b T_\b),
                 &&\text{by Proposition~\ref{scalar}},\\
   &=\frac1{\dfscal}\Tr_\b(F_\b(\blam))\Tr(\v_\b T_\b),
                 &&\text{by Theorem~\ref{comparison}},\\
   &=\frac1{\dfscal\sscal[\b]}\Tr(v_\b T_\b).
   \end{align*}
   Rearranging this equation gives the first formula for
   $\dfscal$.  Applying Corollary~\ref{vb trace} proves the
   second.
\end{proof}

\begin{Remark}
  The proof of Theorem~\ref{flam two} is deceptively easy: all of the hard
  work is done in proving Theorem~\ref{subalgebra} and
  Theorem~\ref{comparison}.
\end{Remark}

We want to make the formula for $\dfscal$ more explicit. To do this we recall
the elegant closed formula for the Schur elements obtained by Chlouveraki
and Jacon~\cite{ChlouverakiJacon:AKSchurElts}. Before we can state
their result we need some notation. First, for
$\blam\in\Part[d,\b]$ define $\overrightarrow\blam$ to be the partition obtained
from~$\blam$ by putting all of the parts of~$\blam$ in weakly decreasing order. (For example,
if~$\blam=\big((2,1^2),(3,2,1)\big)$ then $\overrightarrow\blam=(3,2^2,1^3)$.) Next,
if $\lambda$ is a partition define
$$\beta(\lambda)=\sum_{i\ge1}(i-1)\lambda_i=\sum_{i\ge1}\tbinom{\lambda'_i}2,$$
where $\lambda'=(\lambda'_1,\lambda'_2,\dots)$ is the partition conjugate
to~$\lambda$ (the second equality is well-known and straightforward to check).
Given two partitions $\lambda$ and $\mu$ and $(i,j)\in[\lambda]$ define
$$h_{ij}(\lambda,\mu) = \lambda_i-i+\mu'_j-j+1,$$
which Chlouveraki and Jacon call a generalised \textbf{hook length}. Observe that 
if $1\le s\le r$ then~$s$ can be written uniquely in the form $s=d(p_s-1)+d_s$, where 
$1\le p_s\le p$ and $1\le d_s\le d$. Then, as in \eqref{E:blamt},
$\lambda^{(s)}$ is the $d_s$th component of $\blam^{[p_s]}$. Finally, if
$(i,j,s)\in[\blam]$ and $1\le t\le r$ then set
\[h^\blam_{ij}(s,t)=\dot\eps^{p_s-p_t}\dot q^{h_{ij}(\lambda^{(s)},\lambda^{(t)})}
\dot Q_{d_s}\dot Q_{d_t}^{-1}.
\]

After translating their notation to our setting, Chlouveraki and 
Jacon~\cite[Theorem~3.2]{ChlouverakiJacon:AKSchurElts} show that
\[\sscal=(-1)^{n(r-1)}\dot{q}^{-\beta(\overrightarrow\blam)}(\dot q-1)^{-n}
\prod_{(i,j,s)\in[\blam]}\prod_{1\le t\le r}(h^\blam_{ij}(s,t)-1).
\]
There is an analogous formula for~$\sscal[\b]=\prod_{t=1}^p\sscalt$ which can be
obtained by setting $p=1$. Using Theorem~\ref{flam two}, and the equations
above, we obtain the following closed formula for~$\dfscal$.

\begin{cor}\label{flam closed}
  Suppose that $\b\in\Comp$ and that $\blam\in\Part[d,\b]$. Then
  \[\dfscal= \dot\eps^{\frac12rn(p-1)-d\alpha(\b)}\dot q^{\gamma_\b(\blam)}
       (\dot Q_1\dots\dot Q_d)^{n(p-1)}
       \prod_{(i,j,s)\in[\blam]}\prod_{\substack{1\le t\le r\\p_t\ne p_s}}
       (h^\blam_{ij}(s,t)-1),
\]
where 
$\gamma_\b(\blam)=\ell(w_\b)-\beta(\overrightarrow\blam)+\sum_{a=1}^p\beta(\overrightarrow{\blam^{[a]}})$.
In particular, 
$\dfscal\in\Z[\dot\eps,\dot q^{\pm1},\dot Q_1^{\pm1},\dots,\dot Q_d^{\pm1}]$.
\end{cor}

If $\blam\in\Part$ then $\dfscal\in\mA$ by Proposition~\ref{scalar}.
Corollary~\ref{flam closed} establishes the stronger result that~$\dfscal$ is a
Laurent polynomial in~$\Z[\dot\eps,\dot q^{\pm1},\dot Q_1^{\pm1},\dots,\dot Q_d^{\pm1}]$. 
Using the definitions it is easy to see that if $(i,j,s)\in[\blam]$ and $1\le t\le
r$ then $h^\blam_{ij}(s,t)-1$ divides $A(\varepsilon,q,\bQ)$. Therefore, it is
self-evident from Corollary~\ref{flam closed} that $\fscal=\dfscal(\eps,q,\bQ)$
is both well-defined and non-zero whenever $\bQ$ is $(\eps,q)$-separated
over$~R$.


\subsection{Graded Clifford systems}\label{S:Clifford}
We will use Clifford theory extensively in order to understand the
representation theory of~$\Hrpn$ in terms of the representation
theory of~$\Hrn$. This section recalls the theory that we need
most and starts applying it to the algebras $\Hrpn$.

Suppose that $A$ is a finitely generated $R$-algebra. A family of
$R$-submodules $\set{A_s|s\in\Z/p\Z}$ is a \textbf{$\Z/p\Z$-graded
Clifford} system if the following conditions are
satisfied:\begin{enumerate}
\item $A_sA_t=A_{s+t}$ for any $s,t\in\Z/p\Z$;
\item For each $s\in\Z/p\Z$, there is a unit $a_s\in A_s$ such that $A_s=a_sA_0=A_0a_s$;
\item $A=\oplus_{s\in\Z/p\Z}A_s$;
\item $1\in A_0$.
\end{enumerate}

Any automorphism $\alpha$ of an $R$-algebra $A$ induces an
equivalence $\Func^\alpha\map{\Mod A}\Mod A$.  Explicitly, if $M$ is an
$A$-module then $\Func^\alpha(M)=M^\alpha$ is the $A$-module which is equal to
$M$ as an $R$-module but with the action \textit{twisted} by~$\alpha$ so
that if $m\in M$ and $x\in A$ then $m\cdot x=mx^\alpha=m\alpha(x)$, where on the right
hand side we have the usual (untwisted) action of~$A$.

The following general result is proved in
\cite[Proposition~2.2]{Genet:graded}, together
with~\cite[Appendix]{Hu:simpleGrpn} which corrects a gap in the
original argument. Recall that we have assumed that $R$ contains a
primitive $p$th root of unity $\eps$.

\begin{Lemma}\label{genet}
  Suppose that $A$ and $B$ finitely generated $R$-free $R$-algebras such
  that $A=\bigoplus_{t=0}^{p-1}B\theta^t$ where~$\theta$ is a unit in~$A$ such that
  $\theta^p\in B$ and $\theta B=B\theta$. Then there is an isomorphism
  of $(A,A)$-bimodules
  \[A\otimes_B A\cong\bigoplus_{t=0}^{p-1}A^{\theta^t}; %
      b\theta^i\otimes\theta^j\mapsto\sum_{t=0}^{p-1}(\eps^{jt}b\theta^{i+j})_{(t)},\]
  for $b\in B$ and $0\le i,j<p$ and where $(\eps^{jt}b\theta^{i+j})_{(t)}\in A^{\theta^t}$.
  Here we view $\bigoplus_tA^{\theta^t}$ as an $(A,A)$-bimodule by
  making $A$ act from the left as left multiplication
  and from the right on $A^{\theta^t}$ as right multiplication
  twisted by $\theta^t$, for $0\le t<p$.
\end{Lemma}

An explicit isomorphism, as in the lemma, is constructed in~\cite[p.~3391]{Hu:simpleGrpn}.

In the setup of Lemma~\ref{genet} the subspaces
$\set{B\theta^s|s\in\Z/p\Z}$ form a $\Z/p\Z$-graded Clifford
system in $A$. Now we assume that $R=K$ is a field. Let $\alpha$ be the
automorphism of $B$ given by $\alpha(b)=\theta b\theta^{-1}$, for $b\in B$.
Let $\beta$ be the automorphism of $A$ given by
$\beta(b\theta^j)=\eps^{j}b\theta^j$, for $b\in B$ and $j\in\Z/p\Z$.

Let $\Irr(A)$ and $\Irr(B)$ be the sets of isomorphism classes of simple
$A$-modules and simple $B$-modules, respectively. For each
$D(\lam)\in\Irr(A)$ fix a simple $B$-submodule $D^{\lam}$ of
$D(\lam)\Res^A_B$.  It is clear that $D(\lam)^{\alpha}\cong
D(\lam)$ and $(D^{\lam})^{\beta}\cong D^{\lam}$. Let $\o_{\lam}$ be the
smallest positive integer such that $D(\lam)^{\beta^{\o_{\lam}}}\cong
D(\lam)$. Then $\o_{\lam}$ divides~$p$ so we set $p_\lam=p/\o_\lam$. Define an
equivalence relation $\sim_\beta$ on $\Irr(A)$ by declaring that
\[ D(\lam)\sim_\beta D(\mu) \Longleftrightarrow D(\lam)\cong
D(\mu)^{\beta^t}, \quad\text{for some }t\in\Z/p\Z.
\]
Similarly, let $\sim_\alpha$ be the equivalence relation on $\Irr(B)$
given by
\[
D^{\lam}\sim_\alpha D^{\mu} \Longleftrightarrow D^{\lam}\cong
(D^{\mu})^{\alpha^t},\quad\text{for some }t\in\Z/p\Z.  \]

If $D$ is an $A$-module let $\soc_A(M)$ be its \textbf{socle}; that is
the maximal semisimple submodule of~$A$. Similarly, recall that $\head_A(M)$
is the maximal semisimple quotient of~$M$.

The following result is similar to \cite[Lemma~2.2]{GenetJacon}.  The
result in~\cite{GenetJacon} is proved only in the case $R=\C$. As we
now show, the argument applies over any algebraically closed field.

\begin{Lemma}[(cf.\protect{\cite[Lemma~2.2]{GenetJacon}})]\label{GJ2}
    Suppose that $R=K$ is an algebraically closed field and that
    $A=\bigoplus_{t=0}^{p-1}B\theta^t$ as in Lemma \ref{genet}.
    \begin{enumerate}
      \item Suppose that  $D(\lam)\in\Irr(A)$. Then $p_\lam$ is
      the smallest positive integer such that
      $D^{\lam}\cong\big(D^{\lam}\big)^{\alpha^{p_\lam}}$.
      \item Suppose that  $D^\lam\in\Irr(B)$. Then
        $D^\lam\Ind_B^A\cong D(\lam)\oplus D(\lam)^{\beta}
           \oplus\cdots\oplus D(\lam)^{\beta^{\o_\lam-1}}$ and
           $D(\lam)\Res^A_B \cong D^\lam\oplus(D^\lam)^{\alpha}
              \oplus\cdots\oplus \big(D^\lam\big)^{\alpha^{(p_\lam-1)}}$.
      \item $\set{(D^{\lam})^{\alpha^i}|D(\lam)\in\Irr(A)/{\sim_\beta}
              \text{ for } 1\leq i\leq p_\lam}$
      is a complete set of pairwise non-isomorphic absolutely irreducible $B$-modules.
  \item $\set{D(\lam)^{\beta^i}|D^{\lam}\in\Irr(B)/{\sim_\alpha}\text{ for } 1\leq i\leq o_\lam}$
    is a complete set of pairwise non-isomorphic absolutely irreducible $A$-modules.
\end{enumerate}
\end{Lemma}

\begin{proof} Let $D(\lam)\in\Irr(A)$. Let $p'_{\lam}$ be the smallest
    positive integer such that
    $D^{\lam}\cong\big(D^{\lam}\big)^{\alpha^{p'_{\lam}}}$. By
    \cite[Proposition~11.16]{C&R:2}, the module $D(\lam)\Res^A_B $ is
    semisimple. Now,
    \[ \Hom_{A}\big(D^{\lam}, D(\lam)\Res^A_B \big)
          \cong\Hom_{A}\big((D^{\lam})^{\alpha^t}, D(\lam)\Res^A_B \big),
      \quad\text{for any }t\in\Z.
\]
Therefore, there exists an integer $c>0$ such that
\begin{equation}\label{d1}
D(\lam)\Res^A_B \cong\big(D^{\lam}\oplus (D^{\lam})^{\alpha}\oplus\dots
              \oplus (D^{\lam})^{\alpha^{p'_{\lam}-1}}\big)^{\oplus c}.
\end{equation}
By Frobenius Reciprocity,
\[
\Hom_{B}\big(D(\lam)\Res^A_B , D^{\lam}\big)\cong\Hom_{A}\big(D(\lam),D^{\lam}\Ind_B^A\big).
\]
Since $K$ is algebraically closed, both $A$ and $B$ are split over $K$. It follows that
\begin{equation}\label{d2}
\big(D(\lam)\oplus D(\lam)^{\beta}\oplus\dots\oplus D(\lam)^{\beta^{\o_{\lam}-1}}\big)^{\oplus c}
         \subseteq\soc_{A}\big(D^{\lam}\Ind_B^A\big).
\end{equation}
By (\ref{d1}) and (\ref{d2}), we have that \begin{equation}\label{d3}
    \dim D(\lam)=cp'_{\lam}\dim D^{\lam}\quad\text{and}\quad
    p\dim D^{\lam}\geq c\o_{\lam}\dim D(\lam).
\end{equation}
Hence,
\begin{equation}\label{d4}p\geq c^2p'_{\lam}\o_{\lam}. \end{equation}

On the other hand, since $R$ contains a primitive $p$th root of unity,
the integer $p$ and all of its divisors are invertible in $R$. Let
$\pi_{\lam}$ be a linear endomorphism of $D(\lam)$ which induces an
$A$-module isomorphism $D(\lam)\cong D(\lam)^{\beta^{o_\lam}}$. Then
$(\pi_{\lam})^{p_\lam}\in\End_A\big(D(\lam)\big)=K$. 
Renormalising~$\pi_{\lam}$, if necessary, we can assume that
$(\pi_{\lam})^{p_\lam}=\id_\lam$, where $\id_\lam$ is the identity map
on $D(\lam)$.

Let $X$ be an indeterminate over $K$ and suppose then $\o$ divides $p$.
Differentiating the identity
$X^{p/\o}-1=\prod_{j=1}^{p/\o}(X-\eps^{j\o})$ and setting
$X=\pi_\lam$ and $\o=\o_\lam$, shows that
\[ p_\lam\pi_{\lam}^{p_\lam-1}
         =\sum_{j=1}^{p_\lam}\prod_{\substack{1\leq t\leq p_\lam\\ t\neq j}}
        (\pi_{\lam}-\eps^{t\o_{\lam}}).
\]
Thus, \[
\id_\lam=\frac{1}{p_{\lam}}\sum_{j=1}^{p_\lam}\prod_{\substack{1\leq t\leq p_\lam\\ t\neq j}}(\pi_{\lam}-\eps^{t\o_{\lam}})\pi_{\lam}^{1-p_\lam}.
\]
For each integer $1\leq j\leq p_{\blam}$, we define \[
D_j(\lam):=\frac{1}{p_{\lam}}\prod_{\substack{1\leq t\leq p_\lam\\ t\neq j}}
        (\pi_{\lam}-\eps^{t\o_{\lam}})\,\pi_{\lam}^{1-p_\lam}\, D(\lam).
\]
It is easy to check that each $D_j(\lam)$ is a $B$-submodule of $D(\lam)\Res^A_B $ and $D_j(\lam)\theta=D_{j+1}(\lam)$ for each $j\in\Z/p\Z$. In particular,
this implies that $D(\lam)\Res^A_B $ can be decomposed into a direct sum of~$p_{\blam}$ nonzero $B$-submodules. Comparing this with (\ref{d1}), we deduce that $p_{\blam}=p/\o_{\blam}\leq cp'_{\lam}$. Combining this with (\ref{d4}), shows that $c^2p'_{\lam}\o_{\lam}\leq p\leq cp'_{\lam}\o_{\lam}$, which forces that $c=1$, $p=\o_{\lam}p'_{\lam}$, and \[
D^{\lam}\Ind_B^A=\soc_{A}\bigl(D^{\lam}\Ind_B^A\bigr)=D(\lam)\oplus D(\lam)^{\beta}\oplus\cdots\oplus D(\lam)^{\beta^{\o_{\lam}-1}}.
\]
This proves the first two statements of the lemma. The last two statements follow by
Frobenius reciprocity using the first two statements.
\end{proof}

We now apply these results to $\Hrpn$. Recall from Section~\ref{S:Hrpn} that
$\Hrn$ has two automorphisms $\sigma$ and $\tau$ such that $\Hrpn$ is the
$\sigma$-fixed point subalgebra of~$\Hrn$ and $\tau$ restricts to an
automorphism of~$\Hrpn$.

It is straightforward to check that, as a right $\Hrpn$-module,
\[\Hrn = \Hrpn\oplus T_0\Hrpn\oplus\dots\oplus T_o^{p-1}\Hrpn.\]
(For example, use \cite[Lemma~3.1]{HuMathas:Morita}.) Hence, $\Hrn$ is a
$\Z/p\Z$-graded Clifford system over $\Hrpn$.  Applying Lemma \ref{genet}
to $\Hrn=\bigoplus_{i=0}^{p-1}\Hrpn T_0^i$ we obtain the following useful
result.

\begin{prop}\label{bimoduleiso}
    There is a natural isomorphism of $(\Hrn,\Hrn)$-bimodules
\[\Hrn\otimes_{\Hrpn}\Hrn\cong\bigoplus_{m=0}^{p-1}\big(\Hrn\big)^{\sigma^m},\]
where $\Hrn$ acts from the left on $\big(\Hrn\big)^{\sigma^m}$ as left multiplication and from the right
with its action twisted by $\sigma^m$.
\end{prop}

\begin{cor}\label{res ind}
  Suppose that $M$ is an $\Hrn$-module. Then, as $\Hrn$-modules,
  \[M\Res^{\Hrn}_{\Hrpn}\Ind^{\Hrn}_{\Hrpn}\cong\bigoplus_{i=0}^{p-1}M^{\sigma^i}.\]
\end{cor}

\begin{proof}
  By definition, $M\Res^{\Hrn}_{\Hrpn}\Ind^{\Hrn}_{\Hrpn}=M\otimes_{\Hrn}\Hrn\otimes_{\Hrpn}\Hrn$. Now
  apply Proposition~\ref{bimoduleiso}.
\end{proof}

\subsection{Twisting modules by $\sigma$}\label{S:Twist}
In this section we investigate the effect on $\Hrn$-modules of twisting by the
automorphism~$\sigma$ defined in Section~\ref{S:Hrpn}. These results will be
useful when showing that $\fscal$ has a $p_\blam$ root and when
constructing and classifying the irreducible $\Hrpn$-modules.

Recall that $\sigma(T_0)=\eps T_0$ and that $\sigma(T_i)=T_i$, for $1\le i<n$.
It is easy to check that $\sigma(T_w)=T_w$ and that $\sigma(L_m)=\eps L_m$,
for $w\in\Sym_n$ and for $1\le m\le n$. Hence, using the definitions
we obtain the following.

\begin{Lemma}\label{sigma step}
    Suppose that $1\le b\le n$ and $1\le s\le t\le p$. Then
    \[\sigma(\LL[s,t]_{1,b}) = \eps^{bd(t-s+1)}\LL[s-1,t-1]_{1,b}.\]
    Consequently, if $\b\in\Comp$ then $\sigma(v_\b) = \eps^{-nd}v_{\b}^{(-1)}$ and
    $\sigma(Y_t)=\eps^{-db_t}Y_{t-1}$, for $1\le t\le p$.
\end{Lemma}

By the remarks in Section~\ref{S:Clifford},  the automorphism $\sigma$ induces a
functor $\Func^\sigma$ on the category of~$\Hrn$-modules. We want to compare
$\Func^\sigma$ with the functors $\HFun_\b$, for $\b\in\Comp$, which appear in
the Morita equivalences of Theorem~\ref{T:DMMorita}.

\begin{Lemma}\label{Vb twist}
    Let $\b\in\Comp$ and $t\in\Z$. Suppose that $\bQ$ is $(\eps,q)$-separated over $K$.
    Then $V_{\b\shift t}\cong V_{\b\shift{t+1}}^\sigma$.
\end{Lemma}

\begin{proof}Let $\varsigma=\sigma^{-1}$. Then it is enough to show that
  $V_\b^{\varsigma}\cong V_{\b\shift{{1}}}$ which is equivalent to the
  statement in the Lemma when $t=0$. By Corollary \ref{shiftiso}, there is an
  isomorphism $V_\b\bijection V_{\b\shift1}^{(1)}$. On the other hand,
  $V_\b^{\varsigma}\cong\sigma(V_{\b})\cong V_\b^{(-1)}$ by Lemma~\ref{sigma step}. 
  Therefore, the map $v\mapsto (Y_1v)^{\varsigma}$, for $v\in V_\b$,
  gives the required isomorphism $V_\b^{\varsigma}\bijection V_{\b\shift1}$.
\end{proof}

Suppose that $\b\in\Comp$ and recall that, by definition,
\[\H_{d,\b}=\H_{d,\b}(\bbQ)
   =\H_{d,b_1}(\eps\bQ)\otimes\dots\otimes\H_{d,b_p}(\eps^p\bQ).\]
Suppose that $h=h_1\otimes\dots\otimes h_p\in\H_{d,\b}$. Applying the relations,
there is an algebra isomorphism 
\begin{equation}\label{b-shift}
\H_{d,\b}\bijection\H_{d,\b\shift{{-}1}}; h_1\otimes\dots\otimes h_p
    \mapsto h\shift{{-}1}=h_p^\sigma\otimes h_1^\sigma\otimes\dots\otimes
    h_{p-1}^\sigma,
\end{equation}
where we abuse notation slightly and define $\sigma(T_0^{(t)})=\eps^{-1}
T_0^{(t+1)}$ and $\sigma(T_i^{(t)})=T_i^{(t+1)}$, for $1\le i<b_t$ and
where we equate superscripts modulo~$p$. It follows that there is an
equivalence of categories $\Func^\sigma_\b\map{\Mod\H_{d,\b}}\Mod\H_{d,\b\shift{{-}1}}$
given by
\[\Func^\sigma_\b(M_1\otimes\dots\otimes M_p)
        =M_p\otimes M_1\otimes\dots\otimes M_{p-1},\]
for an $\H_{d,\b}$-module $M_1\otimes\dots\otimes M_p$ and where
$\H_{d,\b\shift{{-}1}}$ acts via the isomorphism above.

\begin{prop}\label{sigma commutes} Let $\b\in\Comp$. Suppose that $\bQ$ is $(\eps,q)$-separated over $K$. Then
\[\begin{CD}
  \Mod\H_{d,\b}@>{\Func^\sigma_\b}>>\Mod\H_{d,\b\shift{{-}1}}\\
  @V{\HFun_\b}VV @VV{\HFun_{\b\shift{{-}1}}}V\\
\Mod\Hrn@>>{\Func^\sigma}>\Mod\Hrn
\end{CD}\]
  is a commutative diagram of functors.
\end{prop}

\begin{proof}
    Let $M$ be an $\H_{d,\b}$-module. Then we have to prove that
    \[\big(M\otimes_{\H_{d,\b}}V_\b\big)^\sigma\cong
          \Func^\sigma_\b(M)\otimes_{\H_{d,\b\shift{{-}1}}}V_{\b\shift{{-}1}}\]
    as right $\Hrn$-modules. Mimicking the proof of
    Lemma~\ref{Vb twist}, the required isomorphism is the map
    $m\otimes v\mapsto m\shift{{-}1}\otimes (Y_1v)^\sigma$, for
    $m\otimes v\in M\otimes_{\H_{d,\b}} V_\b$.
\end{proof}

We want to use this result to determine the $\sigma$-twists of various
$\Hrn$-modules. To this end set
$a^\blam_{s,t}=|\lam^{(dt-d+1)}|+\cdots+|\lam^{(dt-d+s-1)}|$, for
$1\le i\le d$ and $1\le t\le p$ and define
\[\begin{aligned}    & u_{\blam^{[t]}}^{+}=u_{\blam^{[t]}}^{+}(\eps^t\bQ)
     = \prod_{s=2}^{d} \prod_{j=1}^{a^\blam_{s,t}}(L_j-\eps^tQ_s)
       \qquad\text{and}\qquad
       x_{\blam^{[t]}}=\sum_{w\in\Sym_{\blam^{[t]}}}T_w,\\
& y_{\blam^{[t]}}=\sum_{w\in\Sym_{\blam^{[t]}}}(-1)^{\ell(w)}T_w,
\end{aligned}
\]
which we think of as elements of $\H_{d,b_t}(\eps^t\bQ)$ in the natural way.
Now set $u^+_{\blam,\b}=u^+_{\blam^{[1]}}\otimes\dots\otimes u^+_{\blam^{[p]}}$
and $x_{\blam,\b}=x_{\blam^{[1]}}\otimes\dots\otimes x_{\blam^{[p]}}$.
We remark that it is easy to check that $u^+_{\blam,\b}$ and
$x_{\blam,\b}$ commute using Lemma~\ref{JM properties}.

By~\cite[Theorem~2.9]{DuRui:branching}, there exists an element
$s_\b(\blam)=s(\blam^{[1]})\otimes\dots\otimes s(\blam^{[p]})\in\H_{d,\b}$
such that $S_\b(\blam)\cong s_\b(\blam)\H_{d,\b}$. Explicitly,
$s(\blam^{[i]})=u_{\bmu^{[i]}}^+(\eps^i\bQ')y_{\bmu^{[i]}}T_{w(\bmu)}
     x_{\blam^{[i]}}u^+_{\blam^{[i]}}$
where $\bmu^{[i]}$ is the multipartition conjugate to $\blam^{[i]}$, for
$1\le i\le p$.  By Lemma \ref{Morita Specht}, we have that
\begin{equation}\label{E:sblam}
S(\blam)\cong\HFun_\b(S_{\b}(\blam))\cong s_{\b}(\blam)\cdot v_{\b}\H_{r,n}.
\end{equation}
{\it Henceforth, we identify $S(\blam)$ with $s_{\b}(\blam)\cdot V_{\b}$
and $S_\b(\blam)$ with $s_\b(\blam)\H_{d,\b}$ via these isomorphisms.}
Observe that $\HFun_\b(S_\b(\blam))=S(\blam)$ with these
identifications.

\begin{Defn}\label{permutation mods}
    Suppose that $\b\in\Comp$ and $\blam\in\Part[d,\b]$. Define
    \[M_\b(\blam)=u^+_{\blam,\b}x_{\blam,\b}\H_{d,\b}
    \qquad\text{and}\qquad
      M^\blam_\b=\HFun_\b\big(M_\b(\blam)\big).\]
\end{Defn}

The definitions above apply equally well to $\Hrn$-modules by taking $p=1$.
In particular, we have elements $u^+_\blam$ and $x_\blam$ in $\Hrn$ and an
$\Hrn$-module $M(\blam)=u^+_\blam x_\blam\Hrn$. Using the definitions it is
easy to check that $x_\blam=\RTheta_\b(x_{\blam,\b})$ and that
$u^+_\blam=u^+_\b\RTheta_\b(u^+_{\blam,\b})$, where $u^+_\b$ is the element
introduced in (\ref{ub+}). It follows that $M^\blam_\b=v^+_\b M(\blam)$.
Hence, in general, $M^\blam_\b$ is a proper submodule of $V_{\b}$.

We can now prove the promised result about $\sigma$-twisted modules.

\begin{prop}\label{twists}
    Let $\b\in\Comp$ and $\blam\in\Part[d,\b]$. Suppose that $\bQ$ is $(\eps,q)$-separated 
    over~$K$. Then
    \[\big(M^\blam_\b\big)^\sigma\cong M_{\b\shift{{-}1}}^{\blam\shift{{-}1}}
      \quad\text{and}\quad S(\blam)^\sigma\cong S(\blam\shift{{-}1}).\]
     Moreover, if $\blam\in\Klesh$ then $D(\blam)^\sigma\cong D(\blam\shift{{-}1})$.
\end{prop}

\begin{proof}
  We have that
  $\sigma\big(u^+_{\blam^{[t]}}(\eps^t\bQ)\big)
            =\eps^{k_t}u^+_{\blam^{[t]}}(\eps^{t-1}\bQ)$,
  for some integer $k_t$,  exactly as in  Lemma~\ref{sigma step}. From the
  definitions, $F_\b^\sigma\big(M_\b(\blam)\big)\cong M_{\b\shift{{-}1}}(\blam\shift{{-}1})$.
  Therefore, using Proposition~\ref{sigma commutes},
  \begin{align*}
  \big(M^\blam_\b\big)^\sigma
     &=\Func^\sigma\big(\HFun_\b\big(M_\b(\blam)\big)\big)
      \cong\HFun_{\b\shift{{-}1}}\big(F_\b^\sigma\big(M_\b(\blam)\big)\big)\\
     &\cong\HFun_{\b\shift{{-}1}}\big(M_{\b\shift{{-}1}}(\blam\shift{{-}1})\big)
      \cong M_{\b\shift{{-}1}}^{\blam\shift{{-}1}},
  \end{align*}
  giving the first isomorphism. A similar argument shows that
  $S(\blam)^\sigma\cong S(\blam\shift{{-}1})$.

Finally, if $\blam$ is Kleshchev then $D(\blam)\ne0$ and there is a
  short exact sequence
  \[0\longrightarrow \rad S(\blam)\longrightarrow S(\blam)
     \longrightarrow D(\blam)\longrightarrow 0.\]
  The functor $\Func^\sigma$ is exact, and $D(\blam\shift{{-}1})$ is the
  head of $S(\blam\shift{{-}1})$, so $D(\blam)^\sigma\cong D(\blam\shift{{-}1})$ because
  $S(\blam)^\sigma\cong S(\blam\shift{{-}1})$ by the last paragraph. (Note that $\blam$ is
  Kleshchev if and only if~$\blam\shift{{-}1}$ is Kleshchev by Lemma~\ref{Morita
  Specht}(c).)
\end{proof}

As $\sigma$ is trivial on $\Hrpn$, Lemma~\ref{Vb twist} and
Proposition~\ref{twists} imply the following.

\begin{cor}
    Suppose that $\bQ$ is $(\eps,q)$-separated over $K$ and that $\b\in\Comp$, $\blam\in\Part[d,\b]$
    and $t\in\Z$. Then:
    \begin{enumerate}
      \item $V_\b\Res^{\Hrn}_{\Hrpn}\cong V_{\b\shift t}\Res^{\Hrn}_{\Hrpn}$,
      \item $M^\blam_\b\Res^{\Hrn}_{\Hrpn}
               \cong M_{\b\shift t}^{\blam\shift t}\Res^{\Hrn}_{\Hrpn}$,
      \item $S(\blam)\Res^{\Hrn}_{\Hrpn}\cong
               S(\blam\shift t)\Res^{\Hrn}_{\Hrpn}$, and,
      \item if $\blam\in\Klesh$ then $D(\blam)\Res^{\Hrn}_{\Hrpn}
                 \cong D(\blam\shift t)\Res^{\Hrn}_{\Hrpn}$.
    \end{enumerate}
\end{cor}

\subsection{Shifting homomorphisms}
In this section we show that our ordering of the cyclotomic
parameters $\bbQ$ in~\eqref{E:epsShifts} implies the existence of
some isomorphisms between Specht modules. The shifting homomorphisms are,
ultimately, what allow us to construct the irreducible $\Hrpn$ modules --- and
hence prove Theorem~\ref{main simple}. These result also underpin our
calculation of the $l$-splittable decomposition numbers of~$\Hrpn$ and,
consequently, our proof of Theorem~\ref{splittable}.

Extending the notation that we used for the modules $V_\b^{(t)}$, for each
multipartition $\blam\in\Part$ let $S(\blam)^{(t)}$ be the Specht module
for $\Hrn$ which is defined with respect to the ordered parameters
$\eps^t\bbQ$ (rather than~$\bbQ$). Then $S(\blam)\cong S(\blam\shift
t)^{(t)}$ as $\Hrn$-modules and $S(\blam{\shift t})^{(t)}$ is a submodule
of $V_{\b\shift t}^{(t)}$. The following result makes this more
explicit.

\begin{Lemma}\label{Y push}
    Suppose that $\bQ$ is $(\eps,q)$-separated over $K$ and that
    $\blam\in\Part[d,\b]$, for $\b\in\Comp$, and $1\le t\le p$. Then
    \[ Y_t\dots Y_1 S(\blam)=S(\blam\shift t)^{(t)}\]
    as subsets of $\H_{r,n}$.
\end{Lemma}

\begin{proof}
  As we have already observed,  left multiplication by $Y_p\dots Y_1$ is
  invertible by Lemma~\ref{zb invertible} and Lemma~\ref{vb shift}.
  Therefore, $Y_t\dots Y_1 S(\blam)\cong S(\blam)$ as a right
  $\Hrn$-modules, so it is enough to show that $Y_t\dots Y_1
  S(\blam)\subseteq S(\blam\<t\>)^{(t)}$.  Recall from before
  Definition~\ref{permutation mods} that we are identifying
  $S_\b(\blam)$ with the ideal $S_\b(\blam)=s_\b(\blam)\H_{d,\b}$ and
  $S(\blam)=s_\b(\blam)\cdot V_\b$. Using Lemma~\ref{exchange}
  we compute
  \begin{align*}
     Y_t\dots Y_1\big(s_\b(\blam)\cdot v_\b\big)
     &=Y_t\dots Y_1v_\b\RTheta_\b\big(s_\b(\blam)\big) \\
     &=\LTheta_{\b\shift t}\big(s_{\b\shift t}(\blam\shift t)\big)
       Y_t\dots Y_1v_\b\\
    &=s_{\b\shift t}(\blam\shift t)\cdot v^{(t)}_{\b\shift t}Y^*_t\dots Y^*_1,
  \end{align*}
  the last equality following from Corollary~\ref{leftmult}.
  Hence, $Y_t\dots Y_1 S(\blam)\subseteq S(\blam\shift t)^{(t)}$ as
  we needed to show.
\end{proof}

Fix $\b\in\Comp$ and $\blam\in\Part[d,\b]$ and suppose that
$\blam=\blam\shift m$, for some integer $1\leq m\leq p$ with $m$
dividing~$p$. Then $\b=\b\shift m$ and $\sigma^m$ is an automorphism
of~$\Hrn$ of order $\frac pm$. Set
\[
\check{\bQ}=\big(Q_1,Q_1\eps,\cdots,Q_1\eps^{m-1},Q_2,\cdots,Q_2\eps^{m-1},
       \cdots, Q_d,\cdots,Q_d\eps^{m-1}\big).
\]
Then $\Hrn=\Hrn(\bbQ)=\Hrn(\ibbQ)$. By definition,
$\Hrpn[\frac pm]=\Hrpn[\frac pm](\check{\bQ})$ is the subalgebra of
$\Hrn$ generated by $T_0^{p/m},T_1,\cdots,T_{n-1}$, so that
\begin{equation}\label{E:Hrpm}
\Hrpn[\frac pm]\cong\set{h\in\Hrn|h=\sigma^m(h)}.
\end{equation}
This observation will be useful below.

For $0\le t<\frac pm$ we now consider the modules $V_\b^{(tm)}$ and
$S(\blam)^{(tm)}$. Then, by definition, $S(\blam)^{(tm)}$ is a submodule of
$V_\b^{(tm)}$, Further, by Lemma~\ref{sigma step} and
Proposition~\ref{twists},
\[\big(V_\b^{(tm+m)}\big)^{\sigma^{-m}}=V_\b^{(tm)}\quad\text{and}\quad
\big(S(\blam)^{(tm+m)}\big)^{\sigma^{-m}}=S(\blam)^{(tm)}.\] 
Motivated by Definition~\ref{theta defn}, define
\[Y_{t,m} = Y_{tm+m}\dots Y_{tm+2}Y_{tm+1},\]
for $0\le t<\frac pm$, and let
$\theta'_{t,m}\map{V_\b^{(tm)}}{V_\b^{(tm+m)}}$ be the map
$\theta'_{t,m}(v)=Y_{t,m}v$, for $v\in V_\b^{(tm)}$.

\begin{Defn}[(Shifting homomorphisms)]\label{theta tm defn}
  Suppose that $\b\in\Comp$ and that $\b=\b\shift m$ for some $1\leq m\leq p$ with $m$
dividing~$p$. For $0\le t<\tfrac pm$
  define $\theta_{t,m}=\sigma^m\circ\theta'_{t,m}$.
\end{Defn}

\begin{Lemma}\label{theta tm}
    Suppose that $\b\in\Comp$, with $\b=\b\shift m$ for some $1\leq m\leq p$ with $m$
dividing~$p$, and suppose that
    $0\le t<\frac pm$. Then
    $\theta_{t,m}\in\End_{\Hrpn[p/m]}\big(V_\b^{(tm)}\big)$.
\end{Lemma}

\begin{proof}
  By Definition~\ref{theta defn} and the remarks above,
  $\theta_{t,m}\in\End_R\big(V_\b^{(tm)}\big)$ since $\b=\b\shift m$. Moreover,
  if $v\in V_\b^{(tm)}$ and $h\in\Hrn$ then
  \[\theta_{t,m}(vh)=\sigma^m\big(\theta'_{t,m}(vh)\big)
                    =\sigma^m\big(\theta'_{t,m}(v)\big)\sigma^m(h),\]
  since $\theta'_{t,m}$ is an $\Hrn$-module homomorphism by
  Definition~\ref{theta defn}.  Therefore,
  $\theta_{t,m}(vh)$ is an $\Hrpn[p/m]$-module homomorphism
  since $\Hrpn[\frac pm]=\Hrn^{\sigma^m}$ by~\eqref{E:Hrpm}.
\end{proof}

\subsection{Seminormal forms and roots of $\fscal$}
In this section we show that if $\blam=\blam\shift m$, for an integer $m$
dividing $p$ such that $1\le m\le p$, then there exists a scalar~$\fscal[1]$
such that $\fscal=\eps^{mnd(l-1)/2}\big(\fscal[1]\big)^l$, where $l=p/m$ as in
Theorem~\ref{scalars}. By relating~$\fscal$ to the shifting homomorphisms we
will show that this factorization of~$\fscal$ corresponds to a factorization of
the endomorphism of~$S(\blam)$ given by left multiplication by~$z_\b$.

Recall that
$\mA=\Z[\dot\eps,\dot{q}^{\pm1},\dot Q_1^{\pm 1},\dots,\dot Q_d^{\pm1},
                          A(\dot{\eps},\dot{q},\dot{\bQ})^{-1}]$
and that $\mF$ is the field of fractions of~$\mA$. As saw in
Section~\ref{ssec32}, the algebra $\HF$ is semisimple. Note that $\dot\bQ$ is
$(\dot\eps,\dot q)$-separated over~$\mF$ so we can apply all of our
previous results.

Fix $\blam\in\Part$ and an integer $m$ such that $\blam=\blam\shift m$ and
$1\le m\le p$ and $m\mid p$. Let $l=p/m$. Since $\HF$ is semisimple the Specht module
$S(\blam)=S^\mF(\blam)$ is irreducible and has, as we now recall, a seminormal
representation over~$\mF$.  First we need some notation.

Recall from Section~\ref{S:Specht} that $\Std(\blam)$ is the set of standard
$\blam$-tableaux. Each standard tableau $\s\in\Std(\blam)$ is an $r$-tuple
$\s=(\s^{(1)},\dots,\s^{(r)})$ of standard tableaux. Extending the notation
for $\blam=(\blam^{[1]},\dots,\blam^{[p]})$ write
$\s=(\s^{[1]},\dots,\s^{[p]})$, where
$\s^{[j]}=(\s^{[jd-d+1]},\dots,\s^{[jd]})$ is a $\blam^{[j]}$-tableau for
$1\le j\le p$. Similarly, if $z\in\Z$ define
$\s\shift z=(\s^{[z+1]},\dots,\s^{[z+p]})$ where, as usual, we set
$\s^{[j+kp]}=\s^{[j]}$ for $1\le j\le p$ and $k\in\Z$.

If $1\le k\le n$ and $\s\in\Std(\blam)$ define the
\textbf{content} of $k$ in $\t$ to be
\[\cont_\s(k) = \dot\eps^j\dot q^{b-a}\dot Q_c\in\mF,\]
if $k$ appears in row $a$ and column $b$ of $\s^{(c+jd)}$. The following
useful fact is easily proved by induction on~$n$.

\begin{Lemma}[(\protect{cf.  \cite[Lemma~3.12]{JM:cyc-Schaper}})]\label{separation}
  Suppose that $\s\in\Std(\blam)$ and $\t\in\Std(\bmu)$, for
  $\blam,\bmu\in\Part$. Then $\s=\t$ if and only if
  $\cont_\s(k)=\cont_\t(k)$, for $1\le k\le n$.
\end{Lemma}

If $\s$ is a standard $\blam$-tableau and $1\le i<n$ let $\s(i,i+1)$ be the
tableau obtained by interchanging the positions of~$i$ and~$i+1$ in~$\s$.
Then $\s(i,i+1)$ is a standard $\blam$-tableau unless~$i$ and~$i+1$ are
either in the same row or in the same column.

\begin{Lemma}[(\protect{Ariki-Koike~\cite[Theorem~3.7]{AK}})]\label{seminormal}
  Let $V(\blam)$ be the $\mF$-vector space with basis
  $\set{v_\s|\s\in\Std(\blam)}$. Then $V(\blam)$ becomes an $\HF$-module
  with $\HF$-action, for $1\le k\le n$ and $1\le i<n$, given by
  \[ v_\s L_k = \cont_\s(k) v_\s\quad\text{and}\quad
  v_\s T_i = \beta_\s(i)v_\s + \big(1+\beta_\s(i)\big)v_{\t},\]
  where $\t=\s(i,i+1)$, $v_\t=0$ if $\t$ is not standard and
  \[\beta_\s(i)=\frac{(\dot q-1)\cont_\t(i)}{(\cont_\t(i)-\cont_\s(i))}.\]
  Moreover, $V(\blam)\cong S^\mF(\blam)$ as $\HF$-modules.
\end{Lemma}

The module $V(\blam)$ is a \textbf{seminormal form} for $S^\mF(\blam)$.

Recall that we have fixed integers $m$ and $l=p/m$ such that $m\mid p$ and
$\blam=\blam\shift m$. Thus, $S^{\mF}(\blam)^{(tm)}\cong S^{\mF}(\blam)$, for
$0\le t<l=p/m$. By~\eqref{E:sblam},
\[S^\mF(\blam)^{(tm)}=s_\b(\blam)\cdot v_\b^{(tm)}\HF.\]
For convenience, we set $v^{(tm)}_{\tlam}=s_\b(\blam)\cdot v_\b^{(tm)}\in\HF$.

Recall from Section~\ref{ssec32} that $\tlam$ is the standard $\blam$-tableau which
has the numbers $1,2,\dots,n$ entered in order from left to right along
the rows of its first component, then its second component and so on.

\begin{Lemma}\label{eigenvector}
  Suppose that $0\le t<l$. Then
  \[v^{(tm)}_{\tlam}L_k=\cont_{\tlam\shift{-tm}}(k)v^{(tm)}_{\tlam},\]
  for $1\le k\le n$.
\end{Lemma}

\begin{proof}It suffices to consider the case $t=0$ when the result is
  effectively a restatement of~\cite[Proposition~3.13]{M:gendeg}. Alternatively,
  this can be proved using Du and Rui's proof~\cite[Theorem~2.9]{DuRui:branching}
  that the Specht module $S(\blam)$ is isomorphic to the
  corresponding cell module from~\cite{DJM:cyc} together with the description of the
  action of~$L_1,\dots,L_n$ on the standard basis of the cell modules
  from~\cite[Proposition~3.7]{JM:cyc-Schaper}.
\end{proof}

\begin{cor}\label{philam}
  Suppose that $0\le t<l$. Then there exists a unique $\HF$-module
  isomorphism
  \[\phi^{(tm)}_\blam{:}{V(\blam)}\bijection S^\mF(\blam)^{(tm)}\]
  such that $\phi^{(tm)}_\blam(v_{\tlam\shift{-tm}})=v^{(tm)}_{\tlam}$.
\end{cor}

\begin{proof}
  By the Lemma, $v^{(tm)}_{\tlam}$ is a simultaneous eigenvector for
  $L_1,\dots,L_n$ with the eigenvalues being given by the content functions
  $\cont_{\tlam\shift{tm}}(k)$, for $1\le k\le n$. By
  Proposition~\ref{seminormal} the corresponding simultaneous eigenspace in
  $V(\blam)$ is $\mF v_{\tlam\shift{-tm}}$, so any $\HF$-module isomorphism from
  $V(\blam)$ to $S^\mF(\blam)^{(tm)}$ must send $v_{\tlam\shift{-tm}}$ to a
  scalar multiple of $v^{(tm)}_{\tlam}$. As $V(\blam)\cong S^\mF(\blam)\cong
  S^\mF(\blam\<tm\>)^{(tm)}= S^\mF(\blam)^{(tm)}$, by renormalising any
  isomorphism $V(\blam)\longrightarrow S^\mF(\blam)^{(tm)}$ we get the result.
\end{proof}

Suppose that $0\le t<l$.  For each standard $\blam$-tableau $\s$ set
$v^{(tm)}_\s=\phi^{(tm)}_\blam(v_{\s\shift{-tm}})$. Then
$\set{v^{(tm)}_\s|\s\in\Std(\blam)}$ is a Young seminormal basis of $S^\mF(\blam)^{(tm)}$ and, by construction,
\[v^{(tm)}_\s L_k =
\phi^{(tm)}_\blam(v_{\s\shift{tm}})L_k=\cont_{\s\shift{tm}}(k)v^{(tm)}_\s,\]
for $1\le k\le n$.  Recall from Lemma~\ref{Y push} that
$Y_{t,m}S(\blam)^{(tm)}=S(\blam)^{(tm+m)}$. We can now describe the map given by
left multiplication by $Y_{t,m}$ more explicitly.

\begin{prop}\label{flamt}
  Suppose that $0\le t\le m$ and $\s\in\Std(\blam)$. Then there exists a
  scalar $\dfscal[t+1:m](\DoteqQ)\in\mF$ such that
  \[Y_{t,m}v^{(tm)}_{\s\shift m}=\dfscal[t+1:m]v^{(tm+m)}_\s,\]
  for all $\s\in\Std(\blam)$.
\end{prop}

\begin{proof}
  By definition, if $\s\in\Std(\blam)$ then
  $v^{(tm+m)}_{\s}L_k=\cont_{\s\shift{tm+m}}(k)v_{\s}^{(tm+m)}$, for $1\le k\le n$. The
  same statement holds true for $Y_{t,m}v^{(tm)}_{\s\shift m}$, so by
  construction $Y_{t,m}v^{(tm)}_{\s\shift m}$ must be a scalar multiple of
  $v^{(tm+m)}_\s$. By direct calculation the map which sends
  $v_{\s\<m\>}^{(tm)}$ to $v_{\s}^{(tm+m)}$, for each $\s\in\Std(\blam)$,
  defines an $\HF$-isomorphism. By Schur's Lemma this scalar is independent
  of~$\s$ so the Lemma follows.
\end{proof}

We write $\dfscal[t]=\dfscal[t:m](\DoteqQ)$ if $m$ is clear from context.  It is
tempting to say that $\dfscal[t]\in\mA$ since left multiplication by $Y_{t,m}$
is defined over~$\mA$, however, the construction of the basis $\{v^{(tm)}_\s\}$
is only valid over~$\mF$. Nonetheless, we will show below that
$\dfscal[t]\in\mA$ using the fact that~$\mA$ is integrally closed in~$\mF$, .

\begin{Lemma}\label{phi sym hom}
    Let $\phi\map{V(\blam)}V(\blam)$ be the $\mF$-linear map such that
    \[\phi(v_\s) = v_{\s\shift m}, \qquad\text{for all }\s\in\Std(\blam).\]
    Then $\phi$ is an $\H_{\dot q}^\mF(\Sym_n)$-module homomorphism. Moreover,
    $\phi(vx)=\phi(v)\sigma^m(x)$, for all $v\in V(\blam)$ and
    $x\in\HF$. Hence, $\phi$ is an $\H^\mF_{r,p/m,n}$-module
    homomorphism.
\end{Lemma}

\begin{proof}
    Suppose that $\s\in\Std(\blam)$ and $1\le i<n$ and let $\t=\s(i,i+1)$.
    Then, by Lemma~\ref{seminormal},
    \begin{align*}
      \phi(v_\s T_i)&=\beta_\s(i)\phi(v_\s)+(1+\beta_\s(i))\phi(v_\t)
                     =\beta_\s(i)v_{\s\shift m}+(1+\beta_\s(i))v_{\t\shift m}\\
            &=v_{\s\shift m}T_i =\phi(v_\s)T_i,
    \end{align*}
    where the second last equality follows because
    $\beta_\s(i)=\beta_{\s\shift m}(i)$.
    Hence, $\phi$ is a $\H_q(\Sym_n)$-homomorphism. To prove the
    second claim it is enough to show that $\phi(v_\s L_k)=\dot\eps^m
    v_{\s\shift m}L_k$, for all $\s\in\Std(\blam)$ and $1\le k\le n$.
    This is immediate because $\cont_{\s\shift m}(k)=\dot\eps^{-m}\cont_\s(k)$ by 
    Lemma~\ref{eigenvector}.
\end{proof}

\begin{cor}
  Suppose that $0\le t<l$ and that $\s\in\Std(\blam)$. Then
  \[\sigma^m\big(v^{(tm)}_{\s}\big)=\dot\eps^{-dmn}v^{(tm-m)}_{\s}.\]
\end{cor}

\begin{proof}
  First note that $\sigma^m(v^{(tm)}_{\tlam})=\dot\eps^{-dmn}v^{(tm-m)}_{\tlam}$ because
  \[\sigma^m(v^{(tm)}_{\tlam})=\sigma^m\big(s_\b(\blam)\cdot v^{(tm)}_\b\big)
       =\dot\eps^{-dmn}s_\b(\blam)\cdot v_\b^{(tm-m)}
       =\dot\eps^{-dmn}v^{(tm-m)}_{\tlam},\]
  by Lemma~\ref{sigma step}. Therefore,
  writing $v^{(tm)}_{\s}=v^{(tm)}_{\tlam}h=\phi^{(tm)}_\blam(v_{\tlam\shift{-tm}}h)$
  we have $v_{\tlam\shift{-tm}}h=v_{\s\shift{-tm}}$, and so
  \begin{align*}
  \sigma^m\big(v^{(tm)}_{\s}\big)
        &=\sigma^m(v^{(tm)}_{\tlam})\sigma^m(h)
        =\dot\eps^{-dmn}v_{\tlam}^{(tm-m)}\sigma^m(h)\\
        &=\dot\eps^{-dmn}\phi^{(tm-m)}_\blam\big(v_{\tlam\shift{m-tm}}\sigma^m(h)\big)
         =\dot\eps^{-dmn}\phi^{(tm-m)}_\blam\big(\phi(v_{\tlam\shift{-tm}}h)\big)\\
         &=\dot\eps^{-dmn}\phi^{(tm-m)}_\blam\big(\phi(v_{\s\shift{{-}tm}})\big)
         =\dot\eps^{-dmn}\phi^{(tm-m)}_\blam\big(v_{\s\shift{m-tm}}\big)\\
         &=\dot\eps^{-dmn}v_{\s}^{(tm-m)},
  \end{align*}
  as required.
\end{proof}

\begin{Theorem}\label{root}
    Suppose that $\blam\in\Part[d,\b]$ be a multipartition such that
    $\blam=\blam\shift m$, for some $\b\in\Comp$ and $1\leq m\leq p$ with $m\mid p$.
    Set $l=p/m$. Then
    \[\dfscal=\dfscal[1]\dots\dfscal[l]
    = \dot\eps^{\frac12dmn(1-l)}\big(\dfscal[1]\big)^{l}.\]
    Consequently, $\dfscal[t]\in\mA$ for $1\le t\le l$.
\end{Theorem}

\begin{proof}
  By Lemma~\ref{scalar} and Proposition~\ref{flamt}, if $\s\in\Std(\blam)$
  then
  \begin{align*}
    \dfscal v^{(0)}_\s &=Y_p\dots Y_1v^{(0)}_\s
         =Y_{l-1,m }\dots Y_{0,m }v^{(0)}_\s\\
        &=\dfscal[1] Y_{l-1,m }\dots Y_{1,m }v^{(m)}_{\s\shift{-m}}
         =\dots=\dfscal[1]\dots\dfscal[l]v^{(p)}_\s.
  \end{align*}
  Therefore, $\dfscal=\dfscal[1]\dots\dfscal[l]$, since
  $v^{(p)}_\s=v^{(0)}_\s$. This proves the first claim.

  For the second claim, observe that by Lemma~\ref{sigma step}
  \[\sigma^{m} (Y_{t,m})=\dot\eps^{(p-1)dm\b_1^{m} }Y_{t-1,m }
            =\dot\eps^{-dmn/l}Y_{t-1,m },\]
  since $\dot\eps^p=1$ and $l\b_1^m=\b_1^p=n$. Therefore,
  \begin{align*}
    \dfscal[t+1]v^{(tm+m)}_\tlam
    &=Y_{t,m}v^{(tm)}_{\tlam\shift{m}}
     =\sigma^{-m}\big(\sigma^m(Y_{t,m}v^{(tm)}_{\tlam\shift{m}}\big)\\
    &=\dot\eps^{-dmn(1+1/l)}\sigma^{-m}\big(Y_{t-1,m}v^{(tm-m)}_{\tlam\shift{m}}\big)\\
    &=\dot\eps^{-dmn(1+1/l)}\dfscal[t]\sigma^{-m}\big(v^{(tm)}_\tlam\big)\\
    &=\dot\eps^{-dmn/l}\dfscal[t]v^{(tm+m)}_\tlam
  \end{align*}
  Therefore,
  $\fscal[t+1]=\dot\eps^{-dmn/l}\fscal[t]=\dots=\dot\eps^{-tdmn/l}\fscal[1]$.
  The second claim follows.

  To complete the proof observe, for example 
  using~\cite[page~138, Exercise~4.18 and~4.21]{Eisenbud:CommAlg},
  that the ring $\mA$ is an integrally closed domain. Therefore, $\dfscal[1]\in\mA$
  because $\dfscal$ and 
  $\big(\dfscal[1]\big)^l=\dot\eps^{\frac12 dmn(l-1)}\dfscal$ both belong to~$\mA$ 
  by Proposition~\ref{scalar}. Hence, $\dfscal[t]\in\mA$, for $1\le t\le m$,
  completing the proof.
\end{proof}

Henceforth,  let $\fscal[t]$ be the value of $\dfscal[t]$
at $(\DoteqQ)=(\eps, q,\bQ)$, for each integer $1\leq t\leq l$.

\begin{cor}\label{factorization}
Suppose that $\bQ$ is $(\eps,q)$-separated over $R$ and let
    $\blam\in\Part[d,\b]$ be a multipartition such that
    $\blam=\blam\shift m$, for some $\b\in\Comp$ and $1\leq m\leq p$ with $m\mid p$.
    Set $l=p/m$. Then
    $\fscal = \fscal[1]\dots\fscal[l]=\eps^{\frac12dmn(1-l)}\big(\fscal[1]\big)^{l}.$
\end{cor}

Combining Corollary~\ref{factorization} with Proposition~\ref{scalar} and
Theorem~\ref{flam two} we have proved Theorem~\ref{scalars} from the
introduction.

Recall from the introduction that 
$\o_\blam=\min\set{k\ge1|\blam^{[k+t]}=\blam^{[t]}, \text{ for all }t\in\Z}$,
and that $p_\blam=p/\o_\blam$. Note that~$\o_\blam$ divides $p$ so that $p_\blam$
is an integer. 

\begin{Defn}\label{D:gscal}
  Suppose that $\blam\in\Part[d,\b]$, for $\b\in\Comp$. Define
  $\dgscal=\dfscal[1:\o_{\blam}]$. If $\bQ$ is $(\eps,q)$-separated let 
  $\gscal=\dgscal(q,\eps,\bQ)$ be the 
  specialization of $\dgscal=\dgscal(\DoteqQ)$ at~$(\DoteqQ)=(\eps,q,\bQ)$.
\end{Defn}

As the scalars $\dgscal$ are central to all of our main results it is important
to have a closed formula for them. Set
$\sqrt\blam=(\blam^{[1]},\dots,\blam^{[\o_\blam]})$. Abusing notation slightly,
$\blam=(\sqrt\blam,\dots,\sqrt\blam)$, where $\sqrt\blam$ is repeated $p_\blam$
times. Recall from before Corollary~\ref{flam closed} that
if $(i,j,s)\in[\blam]$ and $1\le t\le r$ then
$h^\blam_{ij}(s,t)=\dot\eps^{p_s-p_t}
\dot q^{h_{ij}(\lambda^{(s)},\lambda^{(t)})}\dot Q_{d_s}\dot Q_{d_t}^{-1}$.

\begin{prop}\label{P:gscal} Suppose that $\blam\in\Part[d,\b]$, for
  $\b\in\Comp$, and set $n_\blam=n/p_\blam$. Then there 
  exists $k\in\Z$ such that
  \[
    \dgscal =\dot\eps^{\alpha(\blam)+k\o_\blam}\dot q^{\gamma_\b(\sqrt\blam)}
     (\dot Q_1\dots\dot Q_d)^{n_\blam(p-1)} 
    \prod_{(i,j,s)\in[\sqrt\blam]}\prod_{\substack{1\le t\le d\o_\blam}}
    \prod_{\substack{0\le a<p_\blam\\a\ne 0\text{ if }p_t=p_s}}
                     (\dot\eps^{a\o_\blam}h^\blam_{ij}(s,t)-1),
  \]
  where 
  $\gamma_\b(\sqrt\blam)
  =(\ell(w_\b)+\sum_{a=1}^p\beta(\overrightarrow{\blam^{[a]}})-\beta(\bar\blam))/p_\blam$
  and 
  $\alpha(\blam)=\frac12n_\blam(rp-d\o_\blam)-d\alpha(\b)/p_\blam$ 
  are both integers.
\end{prop}

\begin{proof}
  First observe that $n_\blam=n/p_\blam=|\sqrt\blam|\in\N$.
  By Theorem~\ref{root},
  $\dgscal^{p_\blam}=\dot\eps^{\frac12d\o_\blam n(p_\blam-1)}\dfscal$.
  Therefore, by Corollary~\ref{flam closed}, $\dgscal^{p_\blam}$ is equal to 
   \[
     \dot\eps^{\frac12d\o_\blam n(p_\blam-1)}\dfscal
     =\dot q^{p_\blam\gamma_\b(\sqrt\blam)}\dot\eps^{p_\blam\alpha(\blam)}
        (\dot Q_1\dots\dot Q_d)^{n(p-1)} 
        \prod_{(i,j,s)\in[\blam]}\prod_{\substack{1\le t\le r\\p_t\ne p_s}} 
        (h^\blam_{ij}(s,t)-1).
    \]
    Observe that, because $\blam=\blam\shift{\o_\blam}$, if $1\le s,t\le d\o_\blam$ 
    and $0\le a,b<p_\blam$ then $(i,j,s+a\o_\blam)\in[\blam]$ and
    $h^\blam_{ij}(s,t)=\dot\eps^{(a-b)\o_\blam}h^\blam_{ij}(s+a\o_\blam,t+b\o_\blam)$. 
    Therefore,
    \[\prod_{(i,j,s)\in[\blam]}\prod_{\substack{1\le t\le r\\p_t\ne p_s}} 
        (h^\blam_{ij}(s,t)-1)
        =\prod_{(i,j,s)\in[\sqrt\blam]}\prod_{0\le a<p_\blam}
          \prod_{1\le t\le d\o_\blam}
          \prod_{\substack{0\le b<p_\blam\\p_t+b\o_\blam\ne p_s+a\o_\blam}}
          (\dot\eps^{(a-b)\o_\blam}h^\blam_{ij}(s,t)-1).
    \]
    Now, in the right hand products $1\le s,t\le d\o_\blam$, so
    $p_t+b\o_\blam=p_s+a\o_\blam$ if and only if $p_t=p_s$ and $a=b$. Therefore,
    the last equation becomes
    \[\prod_{(i,j,s)\in[\blam]}\prod_{\substack{1\le t\le r\\p_t\ne p_s}} 
        (h^\blam_{ij}(s,t)-1)
        =\prod_{(i,j,s)\in[\sqrt\blam]}\prod_{1\le t\le d\o_\blam}
        \prod_{\substack{0\le a<p_\blam\\ a\ne 0\text{ if }p_t=p_s}}
        (\dot\eps^{a\o_\blam}h^\blam_{ij}(s,t)-1)^{p_\blam}.
    \]
    Taking $p_\blam$th roots, the formula for $\dgscal$ in the statement of the
    Proposition now follows. Note that this determines $\dgscal$ only up to
    multiplication by $\dot\eps^{ko_\blam}$, a $p_\blam$th root of unity, for
    some~$k\in\Z$.  
    
    Finally, since $\dgscal\in\mA$ by Theorem~\ref{root}, it follows that
    $\gamma_\b(\sqrt\blam)$ and $\alpha(\blam)$ are both integers. We remark
    that it is not difficult to show this directly using just the definitions 
    above. We leave this as an exercise for the reader.
\end{proof}

\begin{Remark}\label{R:gscal}
Proposition~\ref{P:gscal} determines $\dgscal$ up to a $p_\blam$th root of
unity~$\dot\eps^{k\o_\blam}$. For the rest of this paper we make a fixed but
arbitrary choice of the root of unity in Proposition~\ref{P:gscal}. For the
sake of definiteness, we take $k=0$ and set
\[
    \dgscal =\dot\eps^{\alpha(\blam)}\dot q^{\gamma_\b(\sqrt\blam)}
     (\dot Q_1\dots\dot Q_d)^{n_\blam(p-1)} 
    \prod_{(i,j,s)\in[\sqrt\blam]}\prod_{\substack{1\le t\le d\o_\blam}}
    \prod_{\substack{0\le a<p_\blam\\a\ne 0\text{ if }p_t=p_s}}
                     (\dot\eps^{a\o_\blam}h^\blam_{ij}(s,t)-1).
\]
The formula for the splittable decomposition numbers in
Theorem~\ref{splittable}, and all of the results which follow (for example,
Definition~\ref{Hrpn Specht}), are relative to the choice of scalar~$\dgscal$.
The reader can check that any other choice works equally well.
\end{Remark}

Theorem~\ref{root} shows that $\dgscal\in\mA$ and, together with
Proposition~\ref{scalar}, this implies that the
specialisation~$\gscal=\dgscal(\eps,q,\bQ)$ is well-defined and non-zero
whenever~$\bQ$ is $(\eps,q)$-separated. Proposition~\ref{P:gscal} immediately
implies the following stronger characterisation of~$\dgscal$.

\begin{cor}
  Suppose that $\blam\in\Part[d,\b]$, for $\b\in\Comp$. Then
  $\dgscal\in\Z[\dot\eps,\dot q^{\pm1},\dot Q_1^{\pm1},\dots,\dot Q_d^{\pm1}]$. 
  Moreover,~$\gscal\ne0$ whenever $\bQ$ is~$(\eps,q)$-separated.
\end{cor}

The last two results follow directly from Corollary~\ref{flam closed}. In
particular, they do not need the machinery developed in this section. The main
results of this section are really Proposition~\ref{flamt} and
Theorem~\ref{root} which connect the polynomials~$\dgscal$ with the
representation theory of~$\Hrpn$ via the shifting homomorphisms. 

\subsection{Specht modules for $\Hrpn$}\label{S:HrpnSpecht}
Theorem~\ref{root} shows that $\dgscal$ is a $p_\blam$th root of~$\dfscal$. This
implies that the endomorphism of~$S(\blam)$ induced by multiplication by $z_\b$
is a~$p_\blam$th power of a `simpler' endomorphism $\theta_\blam$. In this
section we show that as a $\Hrpn$-module the Specht module decomposes into a
direct sum of $\theta_\blam$-eigenspaces each of which is an $\Hrpn$-module.
These eigenspaces are analogues of Specht modules for~$\Hrpn$ and they will allow
us to construct all of irreducible $\Hrpn$-modules.

\begin{Lemma}\label{cycle}
  Suppose that $\b\in\Comp$ and that $\b=\b\shift m$, for some $1\leq m\leq p$
  with $m$ dividing~$p$. Let~$l=p/m$. Then
  $\theta'_{0,m}=\eps^{dmnt/l}\sigma^{tm}\circ\theta'_{t,m}
                 \circ\sigma^{-tm},$
  for $0\le t<l$.
\end{Lemma}

\begin{proof}We first show that
  $\theta'_{t,m}=\eps^{dmn/l} \sigma^m\circ\theta'_{t+1,m}\circ\sigma^{-m}$ whenever
  $0\le t<l$. By construction both maps belong to
  $\Hom_{\Hrn}(V_\b^{(tm)},V_\b^{(tm+m)})$. By Lemma~\ref{sigma step},
  $\sigma^m(Y_{t+1,m})=\eps^{-dmn/l}Y_{t,m}$. Consequently,
  if $v\in V_\b^{(tm)}$ then
  \[ \big(\sigma^m\circ\theta'_{t+1,m}\circ\sigma^{-m}\big)(v)
     =\sigma^m\big( Y_{t+1,m}\sigma^{-m}(v)\big)
     =\eps^{-dmn/l}Y_{t,m}v=\eps^{-dmn/l}\theta'_{t,m}(v)\\
  \]
  Hence, $\theta'_{t,m}=\eps^{dmn/l} \sigma^m\circ\theta'_{t+1,m}\circ\sigma^{-m}$ as claimed.
  Therefore, if $0\le t<l$ then
  $\theta'_{0,m}=\eps^{dmnt/l}\sigma^{tm}\circ\theta'_{t,m}\circ\sigma^{-tm}$ by induction on~$t$.
\end{proof}

By Lemma~\ref{theta tm}, we have that
$\theta_{t,m}=\sigma^m\circ\theta'_{t,m}\in\End_{\Hrpn[p/m]}\big(V_\b^{(mt)}\big)$, for
$0\le t<p/m$. In particular, $\theta_{0,m}\in\End_{\Hrpn[p/m]}\big(V_\b\big)$.

\begin{Lemma}\label{l power}
  Suppose that $\b\in\Comp$ and that $\b=\b\shift m$, for some $1\leq m\leq p$ with $m$
dividing~$p$. Let~$l=p/m$. Then
  \[(\theta_{0,m})^{l}(v)=\eps^{\frac12 dmn(l-1)}z_\b\cdot v,\]
  for all $v\in V_\b$. That is,
  $\big(\theta_{0,m}\big)^{l}=\eps^{\frac12 dmn(l-1)}z_\b$
  as elements of $\End_{\Hrn}\big(V_\b\big)$.
\end{Lemma}

\begin{proof}By Lemma~\ref{cycle},
  $\theta'_{0,m}=\eps^{dmnt/l}\sigma^{tm}\circ\theta'_{t,m}
                     \circ\sigma^{-tm}$
  for $1\le t<l$. Therefore,
  \begin{align*}
      \big(\theta_{0,m}\big)^{l}&=\big(\sigma^m\circ\theta'_{0,m}\big)\circ\big(\sigma^m\circ\theta'_{0,m}\big)\circ\dots
                       \circ\big(\sigma^m\circ\theta'_{0,m}\big)\\
     &=\sigma^m\circ\eps^{dmn(l-1)/l}\sigma^{(l-1)m}\circ\theta'_{l-1,m}
          \circ\sigma^{(1-l)m}\circ\sigma^m\circ\eps^{dmn(l-2)/l}\sigma^{(l-2)m}\\
     &\qquad\circ\theta'_{l-2,m}\circ\qquad\dots\circ\sigma^m\circ\eps^{dmn/l}\sigma^m
             \circ\theta'_{1,m}\circ\sigma^{-m}\circ\sigma^m
             \circ\theta'_{0,m}\\
     &=\eps^{\frac12 dmn(l-1)} \theta'_{l,m}\theta'_{l-1,m}\circ\dots\circ\theta'_{0,m},
  \end{align*}
  since $\sigma^{lm}=\sigma^{p}$ is the identity map on $\Hrn$.
  By Lemma~\ref{vb shift} and the definitions, if $v\in V_\b$ then
  $\big(\theta'_{l-1,m}\circ\theta'_{l-2,m}\circ\dots\circ\theta'_{0,m}\big)(v)
         =Y_p\dots Y_1v=z_\b\cdot v$,
         so the result follows.
\end{proof}

Given $k\in\Z$ and a sequence $\mathbf{a}=(a_1,a_2,\dots,a_m)$ define
$\mathbf{a}\shift k=(a_{k+1},a_{k+2},\dots,a_{k+m})$, where we set
$a_{i+jm}=a_i$ whenever $j\in\Z$ and $1\le i\le m$. Now define
$\o_m(\mathbf{a})=\min\set{k\ge1|\mathbf{a}\shift k=\mathbf{a}}$.  In
particular, i f $\b\in\Comp$ and
$\blam=(\lam^{[1]},\dots,\lam^{[p]})\in\Part[d,\b]$ then this defines integers
$\o_p(\b)$ and~$\o_p(\blam)$. By definition, $o_p(\b)$ and $\o_p(\blam)$ both
divide $p$, so $o^p(\b)$ and $o^p(\blam)$ are both integers. Further, $\o_p(\b)$
divides $\o_p(\blam)$.

For convenience, set $\o_\blam=\o_p(\blam)$, $p_\blam=p/\o_\blam$,
$\o_\b=\o_p(\b)$ and $p_\b=p/\o_\b$. The definition of $\o_\blam$ and $p_\blam$
agree with those given in the introduction.

\begin{Defn}\label{theta lambda}
  Suppose that $\b\in\Comp$ and $\blam\in\Part[d,\b]$. Let $\theta_\blam$ be
  the restriction of $\theta_{0,\o_{\blam}}$ to $S(\blam)$.
\end{Defn}

As in Lemma~\ref{theta tm}, the image of $\theta_\blam$ is contained in
$S(\blam)$ so we can consider~$\theta_\blam$ to be an~$\Hrpn[p_\blam]$-module
endomorphism of $S(\blam)$. Recall the scalar~$\gscal=\dgscal(\eps,q,\bQ)$ 
from Definition~\ref{D:gscal} and Remark~\ref{R:gscal}.

\begin{cor}\label{kernel}
  Suppose that $\b\in\Comp$ and $\blam\in\Part[d,\b]$. Then
  \[ \big(\theta_\blam\big)^{p_\blam}=\gscal^{p_\blam}1_{S(\blam)},\]
   where $1_{S(\blam)}$ is the identity map on $S(\blam)$.
\end{cor}

\begin{proof} Proposition~\ref{scalar} and Lemma~\ref{l power} show that
  $(\theta_\blam)^{p_\blam}
       =\eps^{\frac12dn\o_{\blam}(p_\blam-1)}\fscal 1_{S(\blam)}$.
  Now apply Theorem~\ref{root}.
\end{proof}

\begin{Defn}\label{Hrpn Specht}
  Suppose that $\b\in\Comp$, $\blam\in\Part[d,\b]$ and $1\le t\le p_\blam$.
  Define
  \[S^\blam_t
     =\set{x\in S(\blam)|\theta_\blam(x)=\eps^{to_\blam}\gscal x}
     =\ker\big(\theta_\blam-\eps^{to_\blam}\gscal 1_{S(\blam)}\big).\]
   Set $\pi^\blam_t=\prod_{1\le s\le p_\blam, s\ne t}
         \big(\theta_\blam-\eps^{so_\blam}\gscal\big)$, so
         that $\pi^\blam_t\in\End_{\Hrpn[p_\blam]}\big(S(\blam)\big)$.
\end{Defn}

By definition, $S^\blam_t$ is an $\Hrpn[p_\blam]$-submodule of
$S(\blam)$, for $1\le t\le p_\blam$. By restriction, we 
consider~$S^\blam_t$ to be an $\Hrpn$-module. Recall that $\tau$ is the
automorphism of $\Hrn$ given by $\tau(h)=T_0^{-1}hT_0$, for
$h\in\Hrn$.

\begin{Theorem} \label{thm54}
  Suppose that $\blam\in\Part[d,\b]$, for $\b\in\Comp$, and
  $1\le t\le p_\blam$. Then
  \begin{enumerate}
\item $S_t^\blam T_0=S_{t+1}^\blam$. Equivalently,
  $\big(S^\blam_{t+1}\big)^\tau\cong S^\blam_t$.
\item $S_t^\blam=\pi^\blam_t\big(S(\blam)\big)$;
\item $S(\blam)\Res^{\Hrn}_{\Hrpn}\cong S_1^\blam\oplus\cdots\oplus
  S_{p_\blam}^\blam$;
\item $\dim S_t^\blam=\frac1{p_\blam}\dim S(\blam)$;
\item $\ind_{\Hrpn}^{\Hrn}(S_t^\blam)
   \cong S(\blam)\oplus S(\blam)^\sigma\oplus\dots
              \oplus S(\blam)^{\sigma^{(\o_\blam-1)}}$.
\end{enumerate}
\end{Theorem}

\begin{proof} Suppose that $x\in S_t^\blam$ and let $m=\o_\blam$. By
  definition,
\begin{align*}
    \theta_\blam(xT_0)&=\big(\sigma^m\circ\theta'_{0,m}\big)(xT_0)
              =\sigma^m\big(\theta'_{0,\o_p(\blam)}(x)T_0\big),\\
\intertext{since $\theta'_{0,m}$ is an $\Hrn$-module homomorphism.
Therefore,} \theta_\blam(xT_0)&=\theta_\blam(x)\sigma^m(T_0)
     =\eps^{(t+1)m}\gscal xT_0.
\end{align*}
Hence, $xT_0\in S^\blam_{t+1}$, proving the first half of~(a).
That $S^\blam_{t+1}\cong(S^\blam_t\big)^\tau$ is now immediate because 
if~$x\in S^\blam_{t+1}$ then $x=x'T_0$ for some $x'\in S^\blam_t$. Therefore, if
$h\in\Hrn$ then $xh = x'T_0h=x'\tau(h)T_0$. Hence, we have proved~(a).

By Corollary~\ref{kernel}, the map
$\theta_\blam^{p_\blam}-\gscal^{p_\blam}$ kills
every element of $S(\blam)$. Thus, on $S(\blam)$ we have
\[0=\theta_\blam^{p_\blam}-\gscal^{p_\blam}
=\prod_{1\le s\le p_\blam}\big(\theta_\blam-\eps^{so_\blam}\gscal\big)
=\pi^\blam_t\circ\big(\theta_\blam-\eps^{to_\blam}\gscal\big).\]
Hence, the image of $\pi^\blam_t$ is contained in $S^\blam_t$ and
$\ker\pi^\blam_t=\sum_{s\ne t} S^\blam_s$. Note that the assumption~$\fscal$ is invertible 
in $R$ implies that $\gscal$ is also invertible in $R$. If $x\in S^\blam_t$ then
$\pi^\blam_t(x)=\alpha_t x$, where
$\alpha_t=\gscal\prod_{s\ne t}(\eps^{to_\blam}-\eps^{so_\blam})$ is invertible in $R$.
It follows that if we set $\hat\pi^\blam_s=\frac1{\alpha_s}\pi^\blam_s$
then
\[1_{S(\blam)} =\hat\pi^\blam_1+\hat\pi^\blam_2+
        \dots+\hat\pi^\blam_{p_\blam},\]
and $\hat\pi^\blam_t$ is the projection map from $S(\blam)$
onto~$S^\blam_t$.  Hence,~(b) and~(c) now follow. Moreover, since $\dim
S^\blam_t=\dim S^\blam_{t+1}$ by~(a), we obtain~(d) from~(c).

It remains then to prove~(e). First observe that by part~(a),
\[\ind_{\Hrpn}^{\Hrn}(S_t^\blam)
\cong \ind_{\Hrpn}^{\Hrn}({(S_{t+1}^\blam)^\tau})
  \cong \big(\ind_{\Hrpn}^{\Hrn}(S_{t+1}^\blam)\big)^\tau
  \cong \ind_{\Hrpn}^{\Hrn}(S_{t+1}^\blam).\]
Therefore,
$\ind_{\Hrpn}^{\Hrn}(S_1^\blam)\cong\dots
       \cong\ind_{\Hrpn}^{\Hrn}(S_{p_\blam}^\blam)$.
Hence, using part~(c) ---  which we have already proved --- and applying
Corollary~\ref{res ind}, we see that
\begin{align*}
\Big(S_t^\blam\Ind^{\Hrn}_{\Hrpn}\Big)^{\oplus p_\blam}
     &\cong \big(S^\blam_1\oplus\dots\oplus S^\blam_{p_\blam}\big)\Ind^{\Hrn}_{\Hrpn}
      \cong S(\blam)\Res^{\Hrn}_{\Hrpn}\Ind^{\Hrn}_{\Hrpn}
      \cong\bigoplus_{j=0}^{p-1}S(\blam)^{\sigma^j}\\
     &\cong \bigg(\bigoplus_{j=0}^{\o_\blam-1}S(\blam)^{\sigma^j}
          \bigg)^{\oplus p_\blam},
\end{align*}
where the last isomorphism follows because $S(\blam)^{\sigma^t}\cong
S(\blam\shift{-t})$ by Proposition~\ref{twists}. Applying the
Krull--Schmidt theorem we deduce
\[ S_t^\blam\Ind^{\Hrn}
   \cong S(\blam)\oplus S(\blam)^{\sigma}\oplus\cdots\oplus
              S(\blam)^{\sigma^{(\o_\blam-1)}},
\]
proving (e). This completes the proof of Theorem~\ref{thm54}.
\end{proof}

As in the introduction, let $\ssim$ be the equivalence relation on $\Part$
where $\bmu\ssim\blam$ whenever $\blam=\bmu\shift m$, for some $m\in\Z$. Let
$\Partt$ be the set of $\ssim$-equivalence classes in $\Part$. By
Proposition~\ref{twists}, the set $\Klesh$ of Kleshchev multipartitions is
closed under $\ssim$-equivalence.  Let~$\Kleshh$ be the set of
$\ssim$-equivalence classes of Kleshchev multipartitions. We will abuse
notation and think of the elements of $\Partt$ as multipartitions so that
when we write $\bmu\in\Partt$ we will really mean that $\bmu$ is a
representative of an equivalence class in $\Partt$. Similarly,
$\bmu\in\Kleshh$ means that $\bmu$ is a representative for an equivalence
class in $\Kleshh$.

Let $R=K$ be a field. We call the modules $\set{S^\blam_i|\blam\in\Partt\text{ and }1\le i\le
p_\blam}$ the \textbf{Specht modules} of $\Hrpn$. Using these modules we
can now construct the irreducible $\Hrpn$-modules.

\begin{Defn}
  Suppose that $\blam\in\Klesh$ and $1\le t\le p_\blam$. Define
  $D^\blam_t=\head(S^\blam_t)$.
\end{Defn}

Although this is not clear from the definition, the module $D^\blam_i$ is
irreducible when $\blam\in\Klesh$ and, moreover every irreducible
$\Hrpn$-module arises in this way.

This following result establishes of Theorem~\ref{main simple} from the
introduction and, in fact, proves quite a bit more.

\begin{Theorem}\label{Hrpn simples}
    Suppose that $\bQ$ is $(\eps,q)$-separated over the field $K$. Let
    $\blam\in\Klesh$. Then:
\begin{enumerate}
\item The module $D^\blam_i=\head(S^\blam_i)$ is an irreducible
  $\Hrpn$-module, for $1\le i\le p_\blam$. Moreover,
  $(D^\blam_{i+1})^\tau\cong D^\blam_i$, for $1\le i\le p_\blam$.
\item If $1\leq i,j\leq p_\blam$ then
$[S_{i}^\blam:D_j^\blam]=\delta_{ij}$.
\item The integer $p_\blam$ is the smallest
positive integer such that $D_i^\blam\cong
\big(D_i^\blam\big)^{\tau^{p_\blam}}$.
\item The integer $\o_\blam$ is the smallest positive integer such that
  $D(\blam)\cong D(\blam)^{\sigma^{\o_\blam}}$.
\item $(D_i^\blam)\Ind^{\Hrn}\cong D(\blam)\oplus D(\blam)^{\sigma}
           \oplus\cdots\oplus D(\blam)^{\sigma^{\o_\blam-1}}$ and \newline
$D(\blam)\Res^{\Hrn}_{\Hrpn}\cong D_i^\blam\oplus(D_i^\blam)^{\tau}
          \oplus\cdots\oplus \big(D_i^\blam\big)^{\tau^{p_\blam-1}}$.
\end{enumerate}
Furthermore, the Hecke algebra $\Hrpn$ is split over $K$ and
\[
\set{D_i^{\bmu}|\bmu\in\Kleshh\text{ and } 1\leq i\leq p_\bmu}\]
is a complete set of pairwise non-isomorphic absolutely irreducible
$\Hrpn$-modules.
\end{Theorem}

\begin{proof}By Proposition~\ref{twists}, $D(\blam)^{\sigma}\cong
  D(\blam\shift{{-}1})$, so it is clear that $\o_\blam$ is the smallest positive
  integer such that $D(\blam)\cong D(\blam)^{\sigma^{\o_\blam}}$. Similarly,
  once we know that $D^\blam_i=\head(S^\blam_i)$ is irreducible then
  $(D^\blam_{i+1})^\tau\cong D^\blam_i$ by Theorem~\ref{thm54}(a) since
  twisting by~$\tau$ induces an exact functor on~$\Mod\Hrpn$.

For the other statements, we first consider the case where $K=\bark$ is
algebraically closed so that $\HrpnK$ splits over $\bark$. The algebra $\Hrn$ is
cellular over any ring and so, in particular, it is split over $K$. Therefore,
if $\bmu\in\Kleshh$ then $D^\bark(\bmu)=D(\bmu)\otimes_K\bark$. Fix an
irreducible $\HrpnK$-submodule~$D^\bmu_\bark$ of~$D^\bark(\bmu)$. By Lemma~\ref{GJ2}, 
the integer $p_\blam$ is the smallest positive integer such that
$D^\blam_\bark\cong (D^\blam_\bark)^{\tau^{p_\blam}}$ and, further,
\begin{align*}
D^{\bark}(\blam)\Res_{\HrpnK}&\cong D^\blam_\bark\oplus (D^\blam_\bark)^{\tau}
                   \oplus\cdots\oplus (D^\blam_\bark)^{\tau^{p_\blam-1}}\\
                   \text{and}\hspace*{20mm}
D^\blam_\bark\Ind^{\HrnK}&\cong D^\bark(\blam)\oplus D^\bark(\blam)^{\sigma}
       \oplus\cdots\oplus D^\bark(\blam)^{\sigma^{\o_\blam-1}}.
\end{align*}
Moreover,
$
\set{\big(D^\bmu_\bark\big)^{\tau^i}|\bmu\in\Kleshh\text{ and }
1\leq i\leq p_\bmu}$
is a complete set of pairwise non-isomorphic simple $\HrpnK$-modules.

Suppose that $\bmu\in\Part$ and let
$S^\bmu_{\bark,i}=S^\bmu_i\otimes_K\bark$, for $1\le j\le p_\bmu$.  We
claim that $D^\blam_\bark\cong\head(S^\bmu_{\bark,i})$, for some $i$, if
and only if $\blam\ssim\bmu$ and in this case $i$ is uniquely determined.
Using the restriction formula for $D^\bark(\blam)$ given above, Frobenius
reciprocity and Theorem~\ref{thm54} we find that
\begin{align*}
  \bigoplus_{i=0}^{p_\bmu-1}\Hom_{\HrpnK}\big(S^\bmu_{\bark,i},D^\blam_\bark\big)
  &\cong \Hom_{\HrpnK}\big(S^\bark(\bmu)\Res_{\HrpnK},D_\bark^{\blam}\big)\\
  &\cong \Hom_{\HrnK}\big(S^\bark(\bmu),D_{\bark}^{\blam}\Ind^{\HrnK}\big)\\
  &\cong\bigoplus_{j=0}^{\o_\blam-1}
  \Hom_{\HrnK}\big(S^\bark(\bmu),D^\bark(\blam)^{\sigma^j}\big)\\
  &\cong\begin{cases}
           \bark,&\text{if }\bmu\ssim\blam,\\
           0,&\text{otherwise},
        \end{cases}
\end{align*}
where the last line follows because
$D^\bark(\bmu)=\head(S^\bark(\bmu))$, by Lemma~\ref{dominance}, and
because $D^\bark(\blam)^{\sigma^j}\cong D^\bark(\blam\shift{{-}j})$ by
Proposition~\ref{twists}. This proves our claim.  Without loss of
generality, we can take $\bmu=\blam$. Note that
$\head{S}^\bark(\blam)=D^\bark(\blam)$ is simple. The above
isomorphisms imply that
$\head(S^\blam_{\bark,i})=D^\blam_{\bark,i}=D^\blam_\bark$ is also
simple. By Lemma \ref{dominance},
$[{S}^\bark(\blam):D^\bark(\blam)]=1$ and $D^\bark(\blam)$ is the
simple head of ${S}^\bark(\blam)$. By considering the restriction of the
composition series of ${S}^\bark(\blam)$ to~$\Hrpn$, it follows
that $[S_{\bark,i}^\blam:D_{\bark,j}^\blam]=\delta_{ij}$. This proves
all the statements in the Theorem when $K=\bark$.

We now return to the general case where $K$ is an arbitrary field. By the
last paragraph, $S^\blam_{\bark,i}\cong S^\blam_i\otimes_K\bark$ has a
simple head, so that $D^\blam_i=\head(S_i^\blam)$ is indecomposable. Therefore, $D^\blam_i$ is irreducible (since it is also
semisimple).

To complete the proof of the Theorem we show that
$D^\blam_i\otimes_K\bark\cong D^\blam_{\bark,i}$. Let $l\ge1$ be the minimal positive integer
such that $(D^\blam_i)^{\tau^{\,l}}\cong D^\blam_i$. Then $l\ge p_\blam$ since
$D^\blam_{\bark,i}\cong\head(D^\blam_i\otimes_K\bark)$. Similarly, $\dim_K
D^\blam_i\ge\dim_\bark D^\blam_{\bark,i}$. By \cite[Proposition~11.16]{C&R:2}, there exists an integer $c\ge1$ such that
\[
D(\blam)\Res^{\Hrn}_{\Hrpn}\cong \Big(D_i^\blam\oplus
(D_i^\blam)^{\tau}\oplus\cdots\oplus
\big(D_i^\blam\big)^{\tau^{{l}-1}}\Big)^{\oplus c}.
\]
Taking dimensions, $\dim_K D(\blam)=cl\dim_K D^\blam_i$. Hence, comparing
dimensions on both sides of the restriction formula for $D^\bark(\blam)$
above shows that
\[cl\dim D^\blam_i=\dim_K D(\blam)=\dim_\bark D^\bark(\blam)
        =p_\blam \dim_\bark D_\bark^{\blam}\le p_\blam\dim_KD^\blam_i.\]
Since $l\ge p_\blam$ this forces $c=1$, $l=p_\blam$ and $\dim_K
D^\blam_i=\dim_\bark D^\blam_{\bark,i}$. Therefore,
$D^\blam_{\bark,i}\cong D^\blam_i\otimes_K\bark$, implying that $D^\blam_i$ is
  absolutely irreducible and hence that $K$ is a splitting field for $\Hrpn$.
All of the parts in the theorem now follow from the corresponding
statements for $D^\blam_{\bark,i}$ using the isomorphism
$D^\blam_{\bark,i}\cong D^\blam_i\otimes_K\bark$.
\end{proof}

The algebra $\Hrn(\bbQ)$ is not necessarily semisimple when $d>1$. With a little more work it is possible to show that if $\bQ$ is $(\eps,q)$-separated over $K$ then the following are equivalent:
\begin{enumerate}
  \item $\Hrn$ is (split) semisimple.
  \item $\Hrpn$ is (split) semisimple.
  \item $S^\blam_t=D^\blam_t$, for all $\blam\in\Part$ and $1\le t\le
    p_\blam$.
\end{enumerate}
We omit the details.  If $d=1$ then it is known that $\H_{p,p,n}$ is
semisimple if and only if $\langle\eps\rangle\cap\langle q\rangle=\{1\}$
and $e>n$~\cite[Theorem~5.9]{Hu:ModGppn}.

Extend the dominance order to $\Partt\times\Z$ by defining
$(\blam,j)\gdom(\bmu,i)$ if~$\blam\gdom\bmu$.  Let
$$\mathbf{D}_{\Hrpn}=\big([S_i^\blam:D_j^{\bmu}]\big)_{(\blam,i),(\bmu,j)}$$
be the \textbf{decomposition matrix} of $\Hrpn$, where $\blam\in\Partt$,
$\bmu\in\Kleshh$, $1\le i\le p_\blam$ and $1\le j\le p_\bmu$, and where the rows
and columns of $\mathbf{D}_{\Hrpn}$ are ordered in a way that is compatible with
dominance.

Suppose that $\blam\in\Part$, $\bmu\in\Kleshh$ and $1\le i\le p_\blam$ and
$1\le i\le p_\bmu$.  If $\blam\ne\bmu$ then $[S_i^\blam:D_j^{\bmu}]\neq
0$ only if $(\blam,i)\gdom(\bmu,j)$ because, by Theorem~\ref{Hrpn simples}
and Lemma~\ref{dominance},
\[[S_i^\blam:D_j^{\bmu}]\ne0\quad\implies\quad [S(\blam):D(\bmu)]\neq 0
  \quad\implies\quad \blam\gdom\bmu.\]
On the other hand, $[S^\bmu_i:D^\bmu_j]=\delta_{ij}$ by
Theorem~\ref{Hrpn simples}. Hence, we have proved the following.

\begin{cor}\label{C:unitriangular}
  Suppose that $\bQ$ is $(\eps,q)$-separated over the field $K$. Then
  the decomposition matrix $\mathbf D_{\Hrpn}$ of~$\Hrpn$
  is unitriangular.
\end{cor}

Theorem~\ref{Hrpn simples} and Corollary~\ref{C:unitriangular} complete the
proof of Theorem~\ref{main simple} from the introduction.

\section{Cyclotomic Schur algebras and decomposition numbers} We have now
  constructed Specht modules and  a complete set of simple modules for~$\Hrpn$.
  In particular, we have proved Theorems~\ref{scalars} and~\ref{main simple}
  from the introduction. The key to proving Theorem~\ref{main simple} was the
  construction of the $\Hrpn$-endomorphism~$\theta_\blam$ of the Specht
  module~$S(\blam)$ in Definition~\ref{theta lambda}.

  In this chapter we compute the $p$-splittable decomposition numbers
  of~$\Hrpn$. To do this we first construct a new algebra $\E_d$ which is Morita
  equivalent to~$\Hrpn$. This  allows us to construct an analogue of the
  (cyclotomic) Schur algebra for~$\Hrpn$. The endomorphisms $\theta_\blam$,
  for~$\blam\in\Part[d,\b]$, lift to analogous elements $\vartheta_\blam$
  of~$\Schrpn$. Extending the arguments of \cite{Hu:DecompDEven}, we compute the
  trace of $\vartheta_\blam$ on certain weight spaces (the twining characters).
  These trace functions give a system of linear equations which determine the
  $p$-splittable decomposition numbers of the three algebras $\Schrpn$, $\E_d$
  and $\Hrpn$. This will complete the proofs of Theorems~\ref{main}
  and~\ref{splittable} from the introduction.

  Many of the early results in this section hold over an integral domain,
  however, for convenience we work over a field~$R=K$. We maintain our
  assumption that $\bQ$ is $(\eps,q)$-separated over $K$.

\subsection{A Morita equivalence for $\Hrpn$} In this section we prove a new
Morita equivalence theorem for the cyclotomic Hecke algebras $\Hrpn$ which is an
analogue of Theorem~\ref{T:DMMorita}. This equivalence (Corollary~\ref{cor64})
is both a refinement of \cite[Theorem~A]{HuMathas:Morita} and a generalization
of the Morita equivalence theorem given by the first author for the Hecke
algebras of type $D$~\cite{Hu:TypeDMorita}.

Fix a composition $\b\in\Comp$ and set $\o_\b=\o_p(\b)$ and
$p_\b=p/\o_\b$.  Mirroring Definition~\ref{theta lambda} define
\[\theta_\b=\theta_{0,\o_p(\b)}.\]
Then $\theta_\b\in\End_{\Hrpn[p_\b]}(V_\b)$ by Lemma~\ref{theta tm} and
$\theta_\b(v)=\sigma^{\o_{\b}}(Y_{0,\o_\b}v)$, for all $v\in V_\b$. In particular, $\theta_\b$
is an $\Hrpn$-endomorphism of $V_\b$.

The module $V_\b=v_\b\Hrn$ is an $\Hrpn$-module by restriction. For
simplicity we will usually write $V_\b$, instead of $V_\b\Res^{\Hrn}_{\Hrpn}$,
when we consider $V_\b$ as an $\Hrpn$-module.

\begin{Defn}\label{Eb defn}
  Suppose that $\b\in\Comp$.
  Define $\E_{d,\b}=\End_{\Hrpn}\big(V_\b\big)$.
\end{Defn}

Notice that $\H_{d,\b}$ is a subalgebra of $\E_{d,\b}$, by
Proposition~\ref{faithful}(a), and that~$\theta_\b$ is an element of~$\E_{d,\b}$ by the
remarks above.

\begin{Theorem}\label{Eb isothm}
  Suppose that $\b\in\Comp$. Then, as an algebra, $\E_{d,\b}$ is generated 
  by~$\H_{d,\b}$ and the endomorphism $\theta_\b$.  Moreover, if
  $\set{x_i|i\in I}$ is a $K$-basis of $\H_{d,\b}$ then
  $\set{x_i\theta_\b^{k}|i\in I\text{ and } 0\leq k<p_\b} $
  is a $K$-basis of $\End_{\Hrpn}(V_\b)$. In particular,
  $\dim\E_{d,\b}=p_\b\dim\H_{d,\b}$.
\end{Theorem}

\begin{proof} We first compute the dimension of $\E_{d,\b}$. By Frobenius
  reciprocity,
  \begin{align*}
    \E_{d,\b} &=\Hom_{\Hrpn}(V_\b\Res^{\Hrn}_{\Hrpn},V_\b\Res^{\Hrn}_{\Hrpn})
       \cong\Hom_{\Hrn}(V_\b, V_\b\Res^{\Hrn}_{\Hrpn}\Ind^{\Hrn}_{\Hrpn})\\
       &\cong\bigoplus_{i=0}^{p-1}\Hom_{\Hrn}(V_\b,V_\b^{\sigma^i})
        \cong\bigoplus_{i=0}^{p-1}\Hom_{\Hrn}(V_\b,V_{\b\shift{i}}),
  \end{align*}
where the third isomorphism is Corollary~\ref{res ind} and the
fourth isomorphism follows because $V_\b^{\sigma^i}\cong V_{\b\shift{{-}i}}$ by
Proposition~\ref{Vb twist}. By \cite[Proposition~2.13]{HuMathas:Morita} if
$\b\ne\c$ then $\Hom_{\Hrn}(V_\b,V_\c)=0$ because
$V_\b$ and $V_\c$ belong to different blocks. Therefore, as vector spaces,
\[ \E_{d,\b}
  \cong \bigoplus_{i=0}^{p_\b-1} \Hom_{\Hrn}\big(V_\b,V_{\b\shift{i\o_\b}}\big)\cong \bigoplus_{i=0}^{p_\b-1} \Hom_{\Hrn}\big(V_\b,V_{\b}\big)
  \cong\H_{d,\b}^{\oplus p_\b}
\]
since $\End_{\Hrn}(V_\b)\cong\H_{d,\b}$ by Proposition~\ref{faithful}(a).
Hence, $\dim\E_{d,\b}=p_\b\dim\H_{d,\b}$.

It remains to show that $\H_{d,\b}$ and $\theta_\b$ generate $\E_{d,\b}$ as a
$K$-algebra. First observe that $\theta_\b$ is an invertible element of
$\E_{d,\b}$ because $(\theta_\b)^{p_\b}(v)=\eps^{dn(p-\o_\b)/2}z_\b\cdot v$ by
Lemma~\ref{l power}, for $v\in V_\b$. Therefore, since $\End_{\Hrn}(V_\b)\cong\H_{d,\b}$ by
Proposition~\ref{faithful}(a), it suffices to show that every element of
$\Hom_{\Hrn}(V_\b,V_{\b}^{\sigma^{io_\b}})$ corresponds to $\theta_\b^{-i}x$,
for some $x$ in~$\H_{d,\b}$. Let $\pi_j$ be the projection from $\E_{d,\b}$ to
$\Hom_{\Hrn}(V_\b,V_{\b}^{\sigma^{jo_\b}})$ under the vector space isomorphism
above. Under Frobenius reciprocity the $\Hrpn$-endomorphism
\[
\theta_\b^{-i}\in\End_{\Hrpn}\big(V_{\b}\Res^{\Hrn}_{\Hrpn}\big)
\]
corresponds to the $\Hrn$-homomorphism
$V_{\b}\longrightarrow V_{\b}\otimes_{\Hrpn}\Hrn$ given by
\[v_{\b}h\mapsto \sum_{s=0}^{p-1}\theta_{\b}^{-i}(v_{\b}hT_0^{-s})\otimes T_0^s,\]
for $h\in\Hrn$.
Using Proposition~\ref{GJ2}, and the explicit isomorphism given
in Lemma~\ref{genet},
\begin{align*}
  \pi_j(\theta_\b^{-i})(v_{\b})
  &=\sum_{s=0}^{p-1} \eps^{jo_\b s}\theta_\b^{-i}(v_\b T_0^{-s})T_0^s
     =\sum_{s=0}^{p-1}\eps^{jo_\b s}\theta_\b^{-i}(v_\b)\eps^{-iso_\b}T_0^{-s}T_0^s\\
    &=\sum_{s=0}^{p-1}\eps^{(j-i)so_\b}\theta_\b^{-i}(v_\b)
     =\delta_{ij}p\theta_\b^{-i}(v_\b).
\end{align*}
By assumption $p$ does not divide the characteristic of~$K$, so~$p$ is
invertible in~$K$. So we deduce that $\pi_i(\theta_\b^{-i})$ is
actually an isomorphism from $V_\b$ onto
$V_\b^{\sigma^{i\o_\b}}$. Essentially the same argument shows
that if $x\in\H_{d,\b}$ then
\[\pi_j(x)(v_\b)=\delta_{j0}px\cdot v_\b=\delta_{j0}pv_\b\RTheta_\b(x).\]
Therefore,
$\pi_j(\theta_\b^{-i}x)(v_\b)=\delta_{ij}\delta_{j0}p^2\theta_\b^{-i}(v_\b)\RTheta_\b(x)$.
Note that every homomorphism in
$\Hom_{\Hrn}(V_\b,V_{\b}^{\sigma^{io_\b}})$ can be decomposed into a
composition of the isomorphism $\pi_i(\theta_\b^{-i})$ with an
endomorphism in $\End_{\Hrn}(V_{\b})\cong\H_{d,\b}$. All of the claims
in the theorem now follow.
\end{proof}

The algebra $\E_{d,\b}$ is generated by $\H_{d,\b}$ and $\theta_\b$ by
Theorem~\ref{Eb isothm}. To make this more explicit, for $s=1,2,\dots,p$ let
$T^{(s)}_i$ and $L^{(s)}_j$, for $1\le i<b_s$ and $1\le j\le b_s$, be the
generators of $\H_{d,\b}$. That is,
\[T^{(s)}_i = 1^{\otimes s-1}\otimes T_i\otimes 1^{\otimes p-s}
\quad\text{and}\quad
  L^{(s)}_j = 1^{\otimes s-1}\otimes L_j\otimes 1^{\otimes p-s},\]
interpreted as elements of
  $\H_{d,\b}=\H_{d,b_1}(\eps\bQ)\otimes\dots\otimes\H_{d,b_p}(\eps^p\bQ)$.
The elements $T^{(s)}_i$ and $L^{(s)}_j$, for $1\le s\le p$,
$1\le i<b_s$ and $1\le j\le b_s$, generate $\H_{d,\b}$ subject to the
relations implied by the defining relations for $\H_{r,n}$.

To determine the relations these elements satisfy in $\E_{d,\b}$ we need, at a
minimum, to determine the commutation relations between elements and
$\theta_\b$. Using Lemma \ref{exchange}, it is easy to deduce the following
result.

\begin{Lemma} \label{commutrel}
  Suppose that $\b\in\Comp$, $1\le s\le p$, $1\le i<b_s$ and $1\le j\le
  b_s$. Then
\begin{align*} T_i^{(s)}\theta_\b&=\begin{cases}
  \theta_\b T_{i}^{(s+\o_\b)}, &\text{if } s+\o_\b\leq p,\\
    \theta_\b T_{i}^{(s+\o_\b-p)},\phantom{\,\eps^{\o_\b}} &\text{if } s+\o_\b>p,
\end{cases}\\
L_j^{(s)}\theta_\b &=\begin{cases}
    \eps^{-\o_\b}\theta_\b L_{j}^{(s+\o_\b)}, &\text{if } s+\o_\b\leq p,\\
    \eps^{-\o_\b}\theta_\b L_{j}^{(s+\o_\b-p)}, &\text{if }   s+\o_\b>p.
\end{cases}
\end{align*}
\end{Lemma}

This lemma, when combined with the relation that $\theta_\b^{p_\b}=\eps^{dn(p-\o_{\b})/2}
z_\b$ is central in $\E_{d,\b}$ and the relations coming from $\H_{d,\b}$
gives a complete set of commutator relations for the generators
of~$\E_{d,\b}$. It would be interesting to know whether or not this gives a
presentation for the algebra $\E_{d,\b}$.

\begin{Remark} \label{identify1} Suppose that $\b\in\Comp$ and
  $1\leq s,t\leq p$ and $s\equiv t \pmod{\o_\b}$, so that $b_s=b_t$. Let
  $\pi_{st}$ be the algebra isomorphism
  $\H_{d,b_s}^{(s)}\cong \H_{d,b_t}^{(t)}$ given by
  \[ T_i^{(s)}\mapsto T_i^{(t)}\quad\text{and}\quad T_0^{(s)}=L_1^{(s)}
            \mapsto\eps^{s-t}T_0^{(t)}, \qquad\text{ for } 1\le i\leq n-1.  \]
Thus, $\pi_{st}$ identifies the $s\th$ tensor factor and the $t\th$ tensor
factor in $\H_{d,\b}$ and Lemma~\ref{commutrel} says that conjugation by
$\theta_\b$ coincides with the map $\pi_{st}$, where $t=s+\o_\b$ if  $s+\o_\b\leq p$; or
$t=s+\o_\b-p$ if $s+\o_\b>p$.
\end{Remark}

Extend the equivalence relation $\ssim$ on $\Part$ to $\Comp$ by defining
$\b\ssim\c$ if $\b=\c\shift k$
for some $k\in\Z$, for $\b,\c\in\Comp$. Let $\Compp=\Comp/{\ssim}$ be the set of
$\ssim$-equivalence classes in~$\Comp$. Once again, we write $\b\in\Compp$
to indicate that $\b$ is a representative for an equivalence class in
$\Compp$.

Define $\displaystyle\E_d=\bigoplus_{\b\in\Compp}\E_{d,\b}$. Note that $\E_d$
depends on the parameters $q$ and $\bbQ$ and on~$n$. Further, by
definition,
$\Mod\E_d=\bigoplus_{\b\in\Compp}\Mod\E_{d,\b}$.

\begin{cor} \label{cor64}
  There is a Morita equivalence
  \[\EFunc\map{\Mod\E_d}{\Mod\Hrpn}; M\mapsto M\otimes_{\E_{d,\b}} V_\b,  \]
  for $M\in\Mod\E_{d,\b}$, and $\b\in\Compp$.
\end{cor}

\begin{proof} By Proposition~\ref{faithful}(b),
  $\bigoplus_{\b\in\Comp}V_\b$ is a progenerator for $\Hrn$. Moreover, if
  $\b\in\Comp$ then
  \[ V_\b\Res^{\Hrn}_{\Hrpn}\cong V_\b^{\sigma^t}\Res^{\Hrn}_{\Hrpn}
            \cong V_{\b\shift{-t}}\Res^{\Hrn}_{\Hrpn},\]
  for any $t\in\Z$ by Lemma~\ref{Vb twist}. Therefore,
  $\bigoplus_{\b\in\Compp}V_\b$ is a progenerator for $\Hrpn$ and, by
  well-known arguments, for example~\cite[\Sect2.2]{Benson:I}, it
  induces the Morita equivalence~$\EFunc$ above.
\end{proof}

We now describe the images of the Specht modules and simple modules of
the algebra $\Hrpn$ under this Morita equivalence.

Let $\blam\in\Part[d,\b]$. By definition $\o_\b\mid \o_\blam$ and
$\o_\blam\mid p$. Let $p_{\b/\blam}:=p_\b/p_\blam=\o_\blam/\o_\b\in\N$. Then
$p_\b=p_{\b/\blam}p_\blam$.

\begin{Defn}\label{E spechts}
    Suppose that $\blam\in\Part[d,\b]$, for $\b\in\Compp$. Define
\[ S^\blam=S_\b(\blam)\Ind_{\H_{d,\b}}^{\E_{d,\b}}
\quad\text{and}\quad
D^\blam=D_\b(\blam)\Ind_{\H_{d,\b}}^{\E_{d,\b}}.
\]
Define $\E_{d,\blam}$ to be the subalgebra of $\E_{d,\b}$ generated by
$\H_{d,\b}$ and $(\theta_\b)^{p_{\b/\blam}}$.
\end{Defn}

By definition $\E_{d,\blam}\cong \E_{d,\bmu}$ whenever
$\blam,\bmu\in\Part[d,\b]$ and $p_\blam=p_\bmu$.  Further,
$\dim\E_{d,\blam}=p_\blam\dim\H_{d,\b}$ by Theorem~\ref{Eb isothm}.
Notice that the maps $(\theta_\b)^{p_{\b/\blam}}$ and~$\theta_\blam$
agree when they are restricted to $S(\blam)$.

Now fix generators $s_{\b}(\blam)$ and $d_{\b}(\blam)$ of $S_\b(\blam)$ and
$D_\b(\blam)$, respectively, which we consider as elements of $\E_{d,\b}$.
Motivated by Definition~\ref{Hrpn Specht} and Theorem~\ref{thm54} define
\begin{align*}
    S^\blam_{i,p_\blam}&=s_{\b}(\blam)\prod_{\substack{1\leq t\leq{p_\blam}\\ t\neq i}}
    \big((\theta_\b)^{p_{\b/\blam}}-\gscal\big)\H_{d,\b}\hookrightarrow\E_{d,\blam}\\
    D^\blam_{i,p_\blam}&=d_{\b}(\blam)\prod_{\substack{1\leq t\leq{p_\blam}\\ t\neq i}}
    \big((\theta_\b)^{p_{\b/\blam}}-\gscal\big)\H_{d,\b}\hookrightarrow\E_{d,\blam}.
\end{align*}
By Lemma \ref{commutrel}, $S^\blam_{i,p_\blam}$ and $D^\blam_{i,p_\blam}$
are $\E_{d,\blam}$-submodules of~$S^\blam$ and
$D^\blam$, respectively. Moreover, it is easy to see that
\[
S_{\b}(\blam)\Ind_{\H_{d,\b}}^{\E_{d,\blam}}\cong\bigoplus_{i=1}^{p_{\blam}}S^\blam_{i,p_\blam}
\quad\text{and}\quad
D_{\b}(\blam)\Ind_{\H_{d,\b}}^{\E_{d,\blam}}\cong\bigoplus_{i=1}^{p_{\blam}}D^\blam_{i,p_\blam}.
\]
Now define
\[
S^\blam_{i,p}= S^\blam_{i,p_\blam}\Ind_{\E_{d,\blam}}^{\E_{d,\b}}
\quad\text{and}\quad
D^\blam_{i,p}= D^\blam_{i,p_\blam}\Ind_{\E_{d,\blam}}^{\E_{d,\b}}.
\]
Let $\simb$ be the equivalence relation on $\Part[d,\b]$ where if
$\blam,\bmu\in\Part[d,\b]$ then $\bmu\simb\blam$
if~$\blam=\bmu\shift{ko_\b}$, for some $k\in\Z$.  Let $\Partb[d,\b]$
be the set of $\simb$-equivalence classes in $\Part[d,\b]$ and let
$\Kleshb[d,\b]$ be the equivalence classes in~$\Klesh[d,\b]$. Once
again, we blur the distinction between equivalence classes 
in~$\Partb[d,\b]$ and the multipartitions in these equivalence classes.

\begin{Lemma}\label{E simples}
    Suppose that $\blam\in\Part[d,\b]$, $\bmu\in\Kleshb[d,\b]$,
    $1\le i\le p_\blam$ and $1\le j\le p_\bmu$. Then
    $\EFunc S^\blam_{i,p}\cong S^\blam_i$ and
    $\EFunc D^\bmu_{j,p}\cong D^\bmu_j$.
In particular,
\[
\set{D^\bmu_{j,p}|\bmu\in\Kleshb[d,\b]\text{ and }1\le j\le p_\bmu}
\]
is a complete set of pairwise non-isomorphic absolutely irreducible $\E_{d,\b}$-modules.
\end{Lemma}

\begin{proof} This follows directly from the definitions and standard
  properties of the Schur functor $\EFunc$.
\end{proof}

\subsection{A cyclotomic $q$-Schur algebra for $\Hrpn$}\label{S:Schur}
The next step towards computing the $l$-splittable decomposition numbers
of~$\Hrpn$ is to lift our computations up to an analogue of the (cyclotomic)
Schur algebra for~$\Hrpn$. In this section we define such an algebra~$\Schrpn$
and prove a basis theorem for it.

The cyclotomic Schur algebras~\cite{DJM:cyc} are defined as endomorphism
algebras of certain permutation-like modules. We now define modules
\[M(\blam), M_\b(\blam)=M(\blam^{[1]})\otimes\dots\otimes M(\blam^{[p]})
\text{ and }
M^\blam_\b=\HFun_\b(M_\b(\blam))\cong v^+_\b M(\blam),\] for
$\blam\in\Part[d,\b]$ and $\b\in\Comp$. The \textit{Schur algebras} for the
three algebras $\Hrn$, $\H_{d,\b}$ and $\Hrpn$ are then defined to the
endomorphism algebras of direct sums of these modules.

\begin{Defn}
  \begin{enumerate}
    \item The \textbf{cyclotomic $q$-Schur algebra} of $\Hrn$ is the
      endomorphism algebra
      \[\Sch=\End_{\Hrn}\Big(\bigoplus_{\blam\in\Part}M(\blam)\Big).\]
    \item For $\b\in\Comp$ the \textbf{cyclotomic $q$-Schur algebra} of
      $\H_{d,\b}$ is the endomorphism algebra
      \[\Sch[d,\b]
      =\End_{\H_{d,\b}}\Big(\bigoplus_{\blam\in\Part[d,\b]}M_\b(\blam)\Big).\]
    \item The \textbf{cyclotomic $q$-Schur algebra} of $\Hrpn$ is the
      endomorphism algebra
      $\Schrpn=\bigoplus_{\b\in\Comp}\Schrpn(\b)$, where
      \[\Schrpn(\b)
         =\End_{\Hrpn}\Big(\bigoplus_{\blam\in\Part[d,\b]}M^\blam_\b\Big),\]
      where $M^\blam_\b$ is considered as an $\Hrpn$-module by
      restriction.
  \end{enumerate}
\end{Defn}

The algebra $\Schrpn$ is new, generalizing the Schur algebras of
type~$D$ introduced by the first author in~\cite{Hu:DecompDEven}.
The cyclotomic Schur algebra $\Sch=\Sch(\bbQ)$ was introduced in
\cite{DJM:cyc}. By~Definition~\ref{permutation mods},
$M_\b(\blam)=M(\blam^{[1]})\otimes\dots\otimes M(\blam^{[p]})$ so
that
\[
\Sch[d,\b]\cong\End_{\H_{d,\b}}\big(\bigoplus_{\blam\in\Part[d,\b]}M_\b(\blam)\big)
   \cong\Sch[d,\b_1](\eps\bQ)\otimes\dots\otimes\Sch[d,b_p](\eps^p\bQ).
\]
Moreover, applying the functor $\HFun_\b$ shows that
\begin{equation}\label{Schur iso}
    \Sch[d,\b]\cong\End_{\Hrn}\Big(\bigoplus_{\blam\in\Part[d,\b]}M_\b^\blam\Big).
\end{equation}
Hence, we can --- and do! --- consider $\Sch[d,\b]$ as a subalgebra of $\Schrpn$.

Recall that after Definition~\ref{Eb defn} we defined
$\theta_\b=\theta_{0,\o_\b}\in\End_{\Hrpn[p_\b]}(V_\b)$. By
definition, $M^\blam_\b$ is a submodule of $V_\b$. We next show that
$\theta_\b$ maps $M_\b^\blam$ to $M_\b^{\blam\shift{\o_\b}}$.

\begin{Lemma} \label{ThetaRes}
  Suppose that $\b\in\Comp$ and $\blam\in\Part[d,\b]$. Then $\theta_\b$
  restricts to give an $\Hrpn$-homomorphism from $M_\b^\blam$ to
  $M_\b^{\blam\shift{\o_\b}}$.
\end{Lemma}

\begin{proof}
  Let $Y_\b=Y_{0,\o_\b}=Y_{\o_\b}\dots Y_1$. Then
  $\theta_\b(v)=\sigma^{\o_\b}(Y_\b v)$, for all $v\in V_\b$.  By
  definition, $M_\b^\blam=v_\b^+u^+_\blam x_\blam\Hrn$ and
  \(v_\b^+u^+_\blam x_\blam =v_\b\RTheta_\b(u^+_{\blam,\b}x_{\blam,\b})
      =\LTheta_\b(x_{\blam,\b}u^+_{\blam,\b}) v_\b,\)
  where these elements are defined just before Definition~\ref{permutation mods}.
  Therefore, it is enough to prove that
  $\theta_\b(v_\b^+u^+_\blam x_\blam)
           =\sigma^{\o_\b}(Y_\b v_\b^+u^+_\blam x_\blam)$
  belongs to $M_\b^{\blam\shift{\o_\b}}$. Using (\ref{Theta}) we compute
  \begin{align*}
      Y_\b v_\b^+u_\blam^{+} x_\blam
        &=Y_\b v_\b\RTheta_\b(u_{\blam,\b}^+ x_{\blam,\b})\\
    &=\LTheta_{\b\shift{\o_\b}}(u^+_{\blam\shift{\o_\b},\b\shift{\o_\b}}
                 x_{\blam\shift{\o_\b},\b\shift{\o_\b}})Y_\b v_\b,%
             &&\text{by Lemma~\ref{exchange}},\\
    &=\LTheta_{\b\shift{\o_\b}}(u^+_{\blam\shift{\o_\b},\b\shift{\o_\b}}
    x_{\blam\shift{\o_\b},\b\shift{\o_\b}}) v^{(\o_\b)}_{\b\shift{\o_\b}} Y_\b^*,%
                 &&\text{by Corollary~\ref{leftmult}}.
  \end{align*}
Hence, using Lemma~\ref{sigma step} there exists an integer $c\in\Z$ such that
\begin{align*}
\theta_\b(v_\b^+u^+_\blam x_\blam)
  &=\eps^{c}v_{\b\shift{\o_\b}}\RTheta_\b(u^+_{\blam\shift{\o_\b},\b\shift{\o_\b}}
        x_{\blam\shift{\o_\b},\b\shift{\o_\b}})\sigma^{\o_\b}(Y_\b^*)\\
  &\in v^+_{\b\shift{\o_\b}}u^+_{\blam\shift{\o_\b}}\sigma^{\o_\b}(Y_\b^*).
\end{align*}
Thus, $\theta_\b(v_\b^+u^+_\blam x_\blam)\in
M_\b^{\blam\shift{\o_\b}}$. Moreover, this map is surjective because left
multiplication by~$Y_\b$, and hence by $\sigma^{\o_\b}(Y_\b^*)$, is invertible by
Lemma~\ref{zb invertible} and Lemma~\ref{vb shift}. As $M_\b^\blam$
and $M_\b^{\blam\shift{\o_\b}}$ are both free and of the same rank the
proof is complete.
\end{proof}

Recall from Lemma~\ref{vb shift} that $z_\b$ is a central element of
$\H_{d,\b}$, for $\b\in\Comp$. Consequently, if~$\blam\in\Part[d,\b]$
then
\[z_\b\cdot v_\b^+ u^+_\blam x_\blam
         = (z_\b u^+_{\blam,\b}x_{\blam,\b})\cdot v_\b
         = (u^+_{\blam,\b}x_{\blam,\b}z_\b)\cdot v_\b
         = (u^+_{\blam,\b}x_{\blam,\b})\cdot v_\b\RTheta_\b(z_\b)
     \in M_\b^\blam.\]
Therefore, left multiplication by $z_\b$ induces a homomorphism in $\End_{\Hrn}(M^\blam_\b)$.

\begin{Defn} Suppose that $\b\in\Comp$. Define maps
    $\vartheta_\b$ and $\zeta_\b$ in $\Schrpn(\b)$ by
    \[\vartheta_\b(m)=\theta_\b(m)\quad\text{and}\quad
      \zeta_\b(m)=z_\b\cdot m,\]
      for $m\in M^\blam_\b$, and $\blam\in\Part[d,\b]$.
\end{Defn}

Using this definition and Lemma~\ref{l power} we obtain:

\begin{Lemma}
    Suppose that $\b\in\Comp$. Then $\zeta_\b$ is central in $\Schrpn$
    and
    \[\vartheta_\b^{p_\b}=\eps^{\frac12d\o_{\b}n(p_\b-1)}\zeta_\b.\]
\end{Lemma}

As remarked in (\ref{Schur iso}) above,
$\Sch[d,\b]\cong\End_{\Hrn}\big(\bigoplus_{\blam\in\Part[d,\b]}M_\b^\blam\big)$
and we view $\Sch[d,\b]$ as a subalgebra of $\Schrpn(\b)$ via this isomorphism.

\begin{Theorem} \label{Schur basis}
    As a $K$-algebra, $\Schrpn(\b)$ is generated by
    $\Sch[d,\b]$ and the endomorphism~${\vartheta}_\b$.  Moreover, if
    $\set{x_i|i\in I}$ is a $K$-basis of $\Sch[d,\b]$ then
    \[\set{x_i{\vartheta}_\b^{k}|i\in I \text{ and } 0\leq k<p_\b}\]
    is a $K$-basis of $\Schrpn(\b)$. In particular,
    $\dim\Schrpn(\b)=p_\b\dim \Sch[d,\b]$.
\end{Theorem}

\begin{proof} This can be proved by repeating the argument of Theorem~\ref{Eb isothm}.
\end{proof}

\subsection{Weyl modules, simple modules and Schur functors}\label{S:Weyl}
This section lifts the problem of computing the $p$-splittable decomposition
numbers of $\Hrpn$ up to $\Schrpn$ by constructing Weyl modules, simple modules
for $\Schrpn$. We then construct an analogue of the Schur functor to relate the
categories of $\Schrpn$-modules and $\Hrpn$-modules, via the category of
$\E_d$-modules.

The cyclotomic Schur algebra $\Sch$ is a quasi-hereditary cellular
algebra with basis
$\set{\phi_{\Stab\Ttab}|\Stab\in\SStd(\blam,\bmu), \Ttab\in\SStd(\blam,\bnu)
             \text{ for }\blam, \bmu,\bnu\in\Part}$,
where $\SStd(\blam,\btau)$ is the set of \textit{semistandard
$\blam$-tableaux of type $\btau$} for
$\btau\in\Part$; see \cite[Definition~4.4 and Theorem~6.6]{DJM:cyc}. In this 
paper we do not need the precise combinatorial
definition of semistandard tableaux. For our purposes it is enough to
know that if $x=u_{\btau}^{+}x_{\btau}h\in M(\btau)$, and
$\Stab\in\SStd(\blam,\bmu)$ and $\Ttab\in\SStd(\blam,\bnu)$, then
\[\phi_{\Stab\Ttab}(x) = \delta_{\bnu\btau}m_{\Stab\Ttab}h,\]
where $m_{\Stab\Ttab}$ is a certain element of~$M(\bmu)$.

For each $\blam\in\Part$ there is a \textbf{Weyl module}
$\Delta(\blam)$, which is a cell module for $\Sch$. Let
$L(\blam)=\Delta(\blam)/\rad\Delta(\blam)$, where $\rad\Delta(\blam)$
is the Jacobson radical of $\Delta(\blam)$. Then
$\set{L(\blam)|\blam\in\Part}$ is a complete set of pairwise
non-isomorphic irreducible $\Sch$-modules. Further, if
$\blam,\bmu\in\Part$ then $L(\bmu)$ is the simple head of
$\Delta(\bmu)$ and
\begin{equation}\label{Schur triangular}
[\Delta(\blam):L(\bmu)]=\begin{cases}
    1,&\text{if }\blam=\bmu,\\
    0,&\text{if }\blam\not\gedom\bmu.
\end{cases}
\end{equation}
All of these facts are
proved in \cite[\Sect6]{DJM:cyc}.

Similarly, for $\b\in\Comp$ let $\Delta_\b(\blam)$ and $L_\b(\blam)$
be the Weyl modules and the irreducible modules of $\Sch[d,\b]$, for
$\blam\in\Part[d,\b]$. For $1\le t\le p$, $\lambda,\nu,\mu\in\Part[d,b_t]$ and
$\Stab\in\SStd(\lam,\mu)$, $\Ttab\in\SStd(\lam,\nu)$, let $\phi^{(t)}_{\Stab\Ttab}$ be the
corresponding element of  $\Sch[d,\b]$ given by
\[\phi^{(t)}_{\Stab\Ttab}(x_1\otimes\dots\otimes x_p)=
   x_1\otimes\dots\otimes x_{t-1}\otimes \phi_{\Stab\Ttab}(x_t)
      \otimes x_{t+1}\otimes\dots\otimes x_p.\]

\begin{Lemma} \label{commutrel2}
  Suppose that $\b\in\Comp$, $1\le s\le p$ and that
$\Stab\in\SStd(\lambda,\mu)$, and $\Ttab\in\SStd(\lambda,\nu)$,
where $\lambda,\mu,\nu\in\Part[d,b_s]$. Then
\[ \varphi_{\Stab\Ttab}^{(s)}\vartheta_\b=\begin{cases}
\eps^{-\o_\b k_{\lambda,\nu}}\vartheta_\b\varphi_{\Stab\Ttab}^{(s+\o_\b)},
&\text{if } s+\o_\b\leq p,\\
\eps^{-\o_\b k_{\lambda,\nu}}\vartheta_\b \varphi_{\Stab\Ttab}^{(s+\o_\b-p)},
&\text{if }  s+\o_\b>p.
\end{cases}
\]
where $k_{\lam,\nu}=\sum_{s=1}^{d-1}\sum_{t=1}^{s}(|\lam^{(t)}|-|\nu^{(t)}|).
$\end{Lemma}

\begin{proof} We first note that $\b\shift{\o_\b}=\b$, so that the notations
  $\phi_{\Stab\Ttab}^{(s+\o_\b)}$ and $\phi_{\Stab\Ttab}^{(s+\o_\b-p)}$ make
  sense.  As the map $\phi_{\Stab\Ttab}$ is given by left multiplication by
  an element of $\H_{d,\b}$, the result follows from Lemma~\ref{commutrel}.
  (In the sequel we only need to know that the scalar $\eps^{-\o_\b
  k_{\lam,\nu}}$ above is equal to $\eps^{\o_b k}$, for some $k\in\Z$, which is
  a consequence of Lemma~\ref{commutrel}. That $k=k_{\lam,\nu}$ can be
  determined using the definition of $m_{\Stab\Ttab}$ from~\cite{DJM:cyc}.)
\end{proof}

\begin{Remark} \label{identify2} Suppose that $\b\in\Comp$ and
  $1\leq s,t\leq p$ and $s\equiv t \pmod{\o_\b}$, so that $b_s=b_t$.
  Just as in Remark~\ref{identify1}, if we let
  $\pi'_{st}$ be the algebra isomorphism
  $\Sch[d,b_s]^{(s)}\cong\Sch[d,b_t]^{(t)}$ given by
  $\phi_{\Stab\Ttab}^{(s)}\mapsto\eps^{-\o_\b k_{\lam,\nu}}
             \phi_{\Stab\Ttab}^{(t)},$
  for $\Stab$ and $\Ttab$ as above. Then $\vartheta_\b$ coincides with
  $\pi'_{st}$, where $t=s+\o_\b$ if $s+\o_\b\leq p$; or $t=s+\o_\b-p$ if $s+\o_\b>p$.
\end{Remark}

For each multipartition $\bmu\in\Part[d,\b]$ the identity map
$\phi_\bmu\map{M_\b(\bmu)}{M_\b(\bmu)}$ belongs to $\Sch[d,\b]$. Then
$\phi_\mu$ is an idempotent in $\Sch[d,\b]$ and
$\sum_{\bmu\in\Part[d,\b]}\phi_\bmu$ is the identity element of
$\Sch[d,\b]$. If $M$ is a $\Sch[d,\b]$-module then $M$ has a \textbf{weight space}
decomposition
\[M=\bigoplus_{\bmu\in\Part[d,\b]}M_\bmu,\qquad\text{where }M_\bmu=M\phi_\bmu.\]
Recall from (\ref{ub+}) that $\bomega=(\bomega^{[1]},\cdots,\bomega^{[p]})$ is the
unique multipartition in $\Part[d,\b]$ such that $\bmu\gedom\bomega$ for all $\mu\in\Part[d,\b]$.
By definition, $\phib$ is the identity map on $\H_{d,\b}$ so
that $\phib\Sch[d,\b]\phib\cong\H_{d,\b}$. Hence, we have a
\textbf{Schur functor}
\begin{equation}\label{E:SchurFunctor}
\SFunc\map{\Mod\Sch[d,\b]}{\Mod\H_{d,\b}}; M\mapsto M_{\bomega},
               \quad\text{ for }M\in\Mod\Sch[d,\b].
\end{equation}

By \cite[Corollary~6.14]{DJM:cyc}, the Weyl module $\Delta_\b(\blam)$
has a basis
\[\set{\phi_\Stab|\Stab\in\SStd(\blam,\bmu)\text{ for }\bmu\in\Part[d,\b]}\]
such that $\set{\phi_\Stab|\Stab\in\SStd(\blam,\bmu)}$ is a basis for
the $\bmu$-weight space of $\Delta_\b(\blam)$.  This implies that
$\SFunc(\Delta_\b(\blam))\cong S_\b(\blam)$, for all
$\blam\in\Part[d,\b]$; see \cite[Proposition~2.17]{JM:cyc-Schaper}.
Hence, $\SFunc(L_\b(\blam))\cong D_\b(\blam)$, for all
$\blam\in\Klesh[d,\b]$, since $\SFunc$ is exact.

There is a unique semistandard $\blam$-tableau $\Ttab^\blam$ of type $\blam$
and $\philam$ is a ``highest weight vector'' in $\Delta_{\b}(\blam)$. In
particular, $\philam$ generates $\Delta_\b(\blam)$.

\begin{Lemma} \label{lthpower}
    Suppose that $\blam\in\Part[d,\b]$, for $\b\in\Comp$. Then
    \[\philam\zeta_\b=\fscal\philam\quad\text{and}\quad
      \philam\vartheta_\b^{p_\b}=(\gscal)^{p_\b}\philam.\]
\end{Lemma}

\begin{proof}
    By \cite[(2.18)]{JM:cyc-Schaper}, the Weyl module
    $\Delta_\b(\blam)$ can be identified with a set of maps from
    $\bigoplus_{\bmu\in\Part[d,\b]} M_\b(\bmu)$ to $S_\b(\blam)$ in
    such a way that $\philam$ is identified with the natural projection map
    $M_\b(\blam)\longrightarrow S_\b(\blam)$. Hence,
    $\philam\zeta_\b=\fscal\philam$ by Proposition~\ref{scalar} and
    $ \philam\vartheta_\b^{p_\b}=(\gscal)^{p_\b}\philam$ by
    Corollary~\ref{kernel}
\end{proof}

By Theorem~\ref{Schur basis}, the subspaces
$\bigl\{\Sch[d,\b],\vartheta_\b \Sch[d,\b],\dots,
(\vartheta_\b)^{p_\b-1}\Sch[d,\b]\bigr\}$ define a $\Z/p_\b\Z$-graded
Clifford system for $\Schrpn(\b)$. In particular, conjugation with
$\vartheta_\b$ defines an algebra automorphism of $\Sch[d,\b]$.  For
any $\Sch[d,\b]$-module $M$ let $M^{\vartheta_\b}$ be the
$\Sch[d,\b]$-module obtained by twisting the action of $\Sch[d,\b]$ by
$\vartheta_\b$.

\begin{Lemma} \label{twist1}
    Suppose that $\blam\in\Part[d,\b]$, for $\b\in\Comp$. Then
    \[ \Delta_\b(\blam)^{\vartheta_\b}\cong\Delta_\b(\blam\shift{\o_\b})
    \quad\text{and}\quad L_\b(\blam)^{\vartheta_\b} \cong L_\b(\blam\shift{\o_\b})\]
    as $\Sch[d,\b]$-modules.
\end{Lemma}

\begin{proof} This follows directly from Lemma \ref{commutrel2} and Remark \ref{identify2}.
\end{proof}

The following definitions mirror the constructions of the Specht modules and
irreducible $\E_{d,\b}$-modules given by Definition~\ref{E spechts}.

\begin{Defn} Suppose that $\blam\in\Part[d,\b]$, for $\b\in\Comp$. Define
\[\Delta^\blam=\Delta_\b(\blam)\Ind_{\Sch[d,\b]}^{\Schrpn(\b)}
   \quad\text{and}\quad
  L^\blam=L_\b(\blam)\Ind_{\Sch[d,\b]}^{\Schrpn(\b)}. \]
\end{Defn}

Let $\Schsigma$ be the automorphism of $\Schrpn(\b)$
which, using Theorem~\ref{Schur basis}, is defined on generators by
\[ (x\vartheta_\b^k)^{\Schsigma}x= \eps^{ko_\b}x\vartheta_\b^k,\qquad
\text{for all }x\in \Sch[d,\b] \text{ and }0\le k<p_\b.\]
By definition, $\Schsigma$ restricts to the identity map on $\Sch[d,\b]$. By Lemma~\ref{genet} there is an isomorphism of $\Schrpn(\b)$-$\Schrpn(\b)$-bimodules,
\begin{equation}\label{equa712}
\Schrpn(\b)\otimes_{\Sch[d,\b]}\Schrpn(\b)\cong
\bigoplus_{j=1}^{p_\b}\big(\Schrpn(\b)\big)^{\Schsigma^{j}},
\end{equation}
such that the left $\Schrpn(\b)$-module structure on
$\big(\Schrpn(\b)\big)^{\Schsigma^j}$ is given by left multiplication and
the right action is twisted by $\Schsigma^j$.

Recall that if $\blam\in\Part[d,\b]$ then $p_{\b/\blam}=p_\b/p_\blam$.
Let $\Sch[d,\blam]$ be the subalgebra of $\Schrpn$ generated by
$\Sch[d,\b]$ and $\vartheta_\blam=\vartheta_\b^{p_{\b/\blam}}$.
Let $\bar\philam$ be the image of $\philam$ in $L_\b(\blam)$. For
$1\le i\le p_\blam$ define
\begin{align*}
\Delta_{i,p_\blam}^\blam&=\philam\prod_{\substack{1\leq t\leq p_\blam\\ t\neq i}}
\Big(\vartheta_\blam-\gscal\eps^{\o_{\blam}t}\Big)\Sch[d,\b]\hookrightarrow\Sch[d,\blam],\\
L_{i,p_\blam}^\blam&=\bar\philam\prod_{\substack{1\leq t\le p_\blam\\ t\neq i}}
\Big(\vartheta_\blam-\gscal\eps^{\o_{\blam}t}\Big)\Sch[d,\b]\hookrightarrow\Sch[d,\blam].
\end{align*}
Then, by Lemma \ref{commutrel2} and Lemma \ref{lthpower},
$\Delta_{i,p_\blam}^\blam$ and $L_{i,p_\blam}^\blam$ are
$\Sch[d,\blam]$-submodules of $\Delta^\blam$ and $L^\blam$,
respectively. Next, for $1\le i\le p_\blam$ define
\[
\Delta^\blam_{i,p}=\Delta_{i,p_\blam}^\blam\Ind_{\Sch[d,\blam]}^{\Schrpn(\b)}
\quad\text{and}\quad
L_{i,p}^\blam=L_{i,p_\blam}^\blam\Ind_{\Sch[d,\blam]}^{\Schrpn(\b)}.
\]
As we will see (cf. Lemma~\ref{GJ2}), each $L_{i,p_{\blam}}^{\blam}$ is an
irreducible $\Sch[d,\blam]$-module and each $L_{i,p}^{\blam}$ is a irreducible
$\Schrpn$-module. 

\begin{prop} \label{mainlist}
Suppose that $\blam\in\Part[d,\b]$, for $\b\in\Comp$, and let
$\Schsigma_\blam=(\Schsigma)^{p_{\b/\blam}}$. Then:
\begin{enumerate}
\item if $1\le i\le p_\blam$ then
  \begin{align*}
   \big(\Delta_{i,p_\blam}^\blam\big)^{\Schsigma_\blam}&\cong\Delta_{i+1,p_\blam}^\blam ,
    & \big(\Delta_{i,p}^\blam\big)^{\Schsigma_\blam}&\cong\Delta_{i+1,p}^\blam,\\
    \big(L_{i,p_\blam}^\blam\big)^{\Schsigma_\blam}&\cong L_{i+1,p_\blam}^\blam
    & \big(L_{i,p}^\blam\big)^{\Schsigma_\blam}&\cong L_{i+1,p}^\blam.
    \end{align*}
\item $\Delta_\b(\blam)\Ind_{\Sch[d,\b]}^{\Sch[d,\blam]}
           \cong\oplus_{i=1}^{p_{\blam}}\Delta_{i,p_\blam}^\blam$
and $L_\b(\blam)\Ind_{\Sch[d,\b]}^{\Sch[d,\blam]}
          \cong\oplus_{i=1}^{p_{\blam}}L_{i,p_\blam}^\blam$. 
   Moreover, there is a unique $\Sch[d,\b]$-module isomorphism
    $\Delta_\b(\blam)\longrightarrow
        \Delta_{i,p_\blam}^\blam\Res^{\Sch[d,\blam]}_{\Sch[d,\b]}$
    such that
\[\philam\mapsto\philam\prod_{\substack{1\leq t\leq{p_{\blam}}\\ t\neq i}}
\Big(\vartheta_\blam -\gscal\eps^{\o_{\blam}t}\Big).
\]
This latter map induces an isomorphism
$L_\b(\blam)\longrightarrow L^\blam_{i,p_\blam}\Res^{\Sch[d,\blam]}_{\Sch[d,\b]}$.
\item $\Delta^\blam=\Delta^\blam_{1,p}\oplus\dots\oplus\Delta^\blam_{p_\blam,p}$
and $L^\blam=L^\blam_{1,p}\oplus\dots\oplus L^\blam_{p_\blam,p}$ as
$\Schrpn$-modules.
\item $\Delta^\blam\cong\Delta^{\blam\shift{\o_\b}}$ and
$L^\blam\cong L^{\blam\shift{\o_\b}}$ as $\Schrpn$-modules.
\end{enumerate}
\end{prop}

\begin{proof} We only prove the results for the Weyl modules. The
    other cases can be proved using similar arguments or using the fact
    that twisting by $\Schsigma$ is an exact functor. 

By Lemma \ref{twist1}, we know that
$(\Delta_\b(\blam))^{\vartheta_\blam}\cong\Delta_\b(\blam\shift{\o_{\blam}})=\Delta_\b(\blam)$.
Therefore,
\begin{align*}
 \Delta^{\blam\shift{\o_{\b}}}&=\Delta_\b(\blam\shift{\o_{\b}})\Ind_{\Sch[d,\b]}^{\Schrpn(\b)}
  \cong\Delta_\b(\blam)^{\vartheta_\b}\Ind_{\Sch[d,\b]}^{\Schrpn(\b)}\\
 &\cong
 \big(\Delta_\b(\blam)\Ind_{\Sch[d,\b]}^{\Schrpn(\b)}\big)^{\vartheta_\b}
  =\big(\Delta^\blam\big)^{\vartheta_\b}\cong\Delta^\blam.
\end{align*}
This proves~(d).

Arguing as in Theorem~\ref{thm54}, it is easy to see that
$\philam\in\Delta^\blam_{1,p}+\dots+\Delta^\blam_{{p_\blam,p}}$.
Hence, $\Delta^\blam=\Delta^\blam_{1,p}+\dots+\Delta^\blam_{p_\blam,p}$.
On the other hand, if $1\le i\le p_\blam$ and $f\in \Sch[d,\b]$ then
the isomorphisms in Remark \ref{identify2}, together with the fact that
$\blam\shift{\o_\blam}=\blam$, imply that $\philam f=0$ if and only if
$\philam(\vartheta_\blam^{i}f\vartheta_\blam^{-i})=0$.
It follows that the map \[
\philam\mapsto\philam\bigg(\prod_{\substack{1\leq t\leq{p_{\blam}}\\ t\neq i}}
          \Big(\vartheta_\blam -\gscal\eps^{\o_{\blam}t}\Big)\bigg)
\]
extends uniquely to a $\Sch[d,\b]$-module surjection
$\rho_{i}: \Delta_\b(\blam)\twoheadrightarrow
                \Delta_{i,p_\blam}^\blam\Res^{\Sch[d,\blam]}_{\Sch[d,\b]}$,
for $1\le i\le p_\blam$. In particular,
$\dim\Delta^\blam_{i,p_\blam}\leq\dim\Delta_\b(\blam)$. By construction,
however,
$\dim\Delta^\blam=p_\blam\dim\Delta_\b(\blam)$. Therefore, the maps
$\rho_i$, for $1\le i\le p_\blam$, are all isomorphisms. This
proves~(b), while~(c) follows easily from the definitions and~(b).

It remains to prove part~(a). Suppose that $1\le i\le p_\blam$.
The definition of $\Schsigma$ implies that if~$f\in\Schrpn(\b)$ then
$\philam f=0$ if and only if $\philam f^{\Schsigma_\blam}=0$.
Therefore, the map
\[\philam\prod_{\substack{1\leq t\leq{p_{\blam}}\\t\neq i+1}}
    \Big(\vartheta_\blam-\gscal\eps^{\o_{\blam}t}\Big)
   f\mapsto\philam\prod_{\substack{1\leq t\leq{p_{\blam}}\\t\neq i}}\Big(
   \vartheta_\blam-\gscal\eps^{\o_{\blam}t}\Big)f^{\Schsigma_\blam}
\]
is a well-defined $\Schrpn(\b)$-module homomorphism from $\Delta^\blam_{i+1,p}$ onto
$\big(\Delta^\blam_{i,p}\big)^{\Schsigma_\blam}$.
Similarly, one can prove that
$\big(\Delta^\blam_{i,p_\blam}\big)^{\Schsigma_\blam}\cong \Delta^\blam_{i+1,p_\blam}$.
\end{proof}

The proof of Proposition~\ref{mainlist}(a) yields the following.

\begin{cor} \label{cor716}
    Suppose that $\blam\in\Part[d,\b]$ and that $1\le i\le p_\blam$. Then, as a $K$-vector space
\[
\Delta^\blam_{i,p}\cong\Delta_\b(\blam)\oplus\Delta_\b(\blam)\vartheta_\b\oplus\dots
               \oplus\Delta_\b(\blam)\vartheta_\b^{p_{\b/\blam}-1},
\]
Moreover, the action of $\Schrpn(\b)$ on $\Delta^\blam_{i,p}$ is uniquely
determined by
\begin{enumerate}
  \item $\Delta^\blam_{i,p}\Res^{\Schrpn(\b)}_{\Sch[d,\b]}
      \cong\Delta_\b(\blam)\oplus\Delta_\b(\blam)^{\vartheta_\b^{-1}}\oplus\dots
           \oplus\Delta_\b(\blam)^{\vartheta_\b^{1-p_{\b/\blam}}}$;
\item $(x\vartheta_\b^j)\vartheta_\b^t=x\vartheta_\b^{j+t}$, for
  all $x\in\Delta_\b(\blam)$ and $j,t\in\Z$;
\item $\vartheta_\blam$ acts as the scalar
    $\gscal\eps^{io_\blam}$ on the highest weight vector of $\Delta_\b(\blam)\hookrightarrow\Delta^\blam_{i,p}$.
\end{enumerate}
Analogous statements hold for the simple module $L^\blam_{i,p}$.
\end{cor}

\begin{proof} By definition, \[
\Delta^\blam_{i,p}\cong\Delta_{i,p_{\blam}}^{\blam}\oplus\Delta_{i,p_{\blam}}^{\blam}\vartheta_\b\oplus\dots
               \oplus\Delta_{i,p_{\blam}}^{\blam}\vartheta_\b^{p_{\b/\blam}-1}.
\]
As in the proof of Proposition \ref{mainlist}, we can identify
$\Delta_{i,p_{\blam}}^{\blam}$ with $\Delta_{\b}(\blam)$ using the
isomorphism $\rho_i$, for $1\le i\le p_\blam$. Then the highest weight
vector $\philam$ of $\Delta_{\b}(\blam)$ corresponds to the vector
$\philam\Big(\prod_{\substack{1\leq t\leq{p_{\blam}}\\ t\neq i}}
        \big(\vartheta_\blam -\gscal\eps^{\o_{\blam}t}\big)\Big)$.
This implies that $\vartheta_{\blam}=\vartheta_{\b}^{p_{\b/\blam}}$
acts as the scalar $\gscal\eps^{io_\blam}$ on the highest weight
vector of $\Delta_\b(\blam)\hookrightarrow\Delta^\blam_{i,p}$. All of
the claims in the Corollary now follow.
\end{proof}

\begin{cor} \label{Asimplehd}
    Suppose that $\blam,\bmu\in\Part[d,\b]$.
    \begin{enumerate}
    \item If $1\le i\le p_\blam$ then $L^\blam_{i,p}$ is the simple
        head of $\Delta^\blam_{i,p}$.
    \item If $1\le i\le p_\blam$ and $1\le j\le p_\bmu$ then
        \[[\Delta^\blam_{i,p}:L^\bmu_{j,p}]=\begin{cases}
        \delta_{ij},&\text{if }\blam=\bmu,\\
        0,&\text{if }\blam\not\gedom\bmu.
        \end{cases}\]
    \end{enumerate}
\end{cor}

\begin{proof}
    By (\ref{Schur triangular}) $L_\b(\blam)$ is the simple head of
    $\Delta_\b(\blam)$ and
    \[ [\Delta_\b(\blam):L_\b(\bmu)]=\begin{cases}
         1, &\text{if }\bmu=\blam,\\
         0, &\text{if }\blam\not\gedom\bmu.
    \end{cases} \]
    Hence, the result follows from Proposition~\ref{mainlist} and
    Frobenius reciprocity.
\end{proof}

Recall that $\simb$ is the equivalence relation on $\Part[d,\b]$ such
that $\blam\simb\bmu$ if $\bmu=\blam\shift{ko_\b}$ for some $k\in\Z$.

\begin{cor} The algebra $\Schrpn(\b)$ is split over $K$ and
    \[ \set{L_{i,p}^\blam|\blam\in\Partb[d,\b]\text{ and }1\leq i\leq p_\blam} \]
is a complete set of pairwise non-isomorphic absolutely irreducible $\Schrpn(\b)$-modules.
\end{cor}

\begin{proof} Just as in the proof of Theorem~\ref{Hrpn simples}, this follows from Corollary
    \ref{Asimplehd}, Frobenius reciprocity and some general arguments
    in Clifford theory.
\end{proof}

Recall from~\eqref{E:SchurFunctor} that the Schur functor
$\SFunc\map{\Mod\Sch[d,\b]}\Mod\H_{d,\b}$ is given by
$\SFunc(M)=M\phib$, where $\phib$ is the identity map on $\H_{d,\b}$.
Using the embedding $\Sch[d,\b]\hookrightarrow\Schrpn(\b)$, and the fact that
$v_\b=v_\b^+u_{\bomega}^{+}$, it is easy to check that
$\phib$
corresponds to the natural projection from
$\bigoplus_{\blam\in\Part[d,\b]}M_\b^\blam$ onto
$V_\b=M^{\bomega}_\b$. In particular,
\[\phib\Schrpn(\b)\phib=\E_{d,\b}\quad\text{and}\quad
  \phib\Sch[d,\b]\phib=\H_{d,\b}.\]
Hence, we have a second Schur functor $\SFunc^{(p)}\map{\Mod\Schrpn(\b)}\Mod\E_{d,\b}$
which is given by $\SFunc^{(p)}(M)=M\phib$ and if $\varphi\in\Hom_{\Schrpn(\b)}(M,N)$
then $\SFunc^{(p)}(\varphi)(x\phib)=\varphi(x)$, for all $x\in M$.
It is straightforward to check that we have the following
commutative diagram of functors:
\begin{equation}\label{Res}
\begin{CD}
    \Mod\Schrpn(\b)@>?\Res^{\Schrpn(\b)}_{\Sch[d,\b]}>>\Mod\Sch[d,\b]\\
    @V\SFunc^{(p)}VV @VV\SFunc V\\
    \Mod\E_{d,\b}@>>?\Res^{\E_{d,\b}}_{\H_{d,\b}}>\Mod\H_{d,\b}
\end{CD}
\end{equation}

\begin{Lemma} \label{schur1}
    Suppose that $\blam\in\Part[d,\b]$ and  $1\le i\le p_\blam$. Then
    \[ \SFunc^{(p)}(\Delta_{i,p}^\blam)\cong S^\blam_{i,p}
    \quad\text{and}\quad
    \SFunc^{(p)}(L_{i,p}^\blam)\cong\begin{cases}
             D^\blam_{i,p},&\text{if }\blam\in\Klesh[d,\b],\\
             0, &\text{otherwise.}
    \end{cases}\]
\end{Lemma}

\begin{proof} This follow directly from (\ref{Res}) and
    Lemma~\ref{E simples}.
\end{proof}

\begin{cor} \label{cor321}  Suppose that  $\b\in\Comp$, $\blam\in\Part[d,\b]$,
$\bmu\in\Klesh[d,\b]$, $1\le i\le p_\blam$ and that
$1\le j\le p_\bmu$. Then
$ [\Delta^\blam_{i,p}:L^\bmu_{j,p}]=[S^\blam_{i,p}:D^\bmu_{j,p}]
                         =[S^\blam_i:D^{\bmu}_{j}]$.
\end{cor}

\begin{proof} This follows directly from Lemma \ref{schur1} and
    Lemma~\ref{E simples} together with the easily checked fact that
    the functors $\SFunc^{(p)}$ and $\EFunc$ are exact.
\end{proof}

Therefore, in order to compute the decomposition number
$[S^\blam_i:D^{\bmu}_{j}]$ it is enough to compute the decomposition
number $[\Delta^\blam_{i,p}:L^\bmu_{j,p}]$ for $\Schrpn$.

\subsection{Splittable decomposition numbers} In this section we derive explicit
formulae for the $l$-splittable decomposition numbers
of the algebras~$\Schrpn(\b)$, and hence of~$\Hrpn$ by
Corollary~\ref{cor321}, in characteristic zero. These decomposition numbers
depend explicitly on the decomposition numbers of certain Ariki--Koike
algebras and on the scalars~$\gscal$ introduced in Lemma~\ref{theta lambda}. By
Theorem~\ref{T:SplittableRedction} this will determine all of the decomposition
numbers of~$\Hrpn$.

Suppose that $\blam$ and $\bmu$ are multipartitions in $\Part[d,\b]$. We want
to compute the decomposition numbers
$[\Delta^\blam_{i,p}:L^\bmu_{j,p}]$ for $1\le i\le p_\blam$ and
$1\le j\le p_\bmu$.  By Corollary~\ref{cor716} and the exactness 
of~$\vartheta_\b$, if $p_\blam=p_\bmu$ then
\begin{equation}\label{cyclic}
[\Delta^\blam_{i,p}:L^\bmu_{j,p}]=[\Delta^\blam_{i+1,p}:L^\bmu_{j+1,p}],
\end{equation}
where we read $i+1$ and $j+1$ modulo $p_\blam$.  Therefore, these
decomposition numbers are determined by the decomposition numbers
\[d_{\blam\bmu}^{(j)}=[\Delta^\blam_{0,p}:L^\bmu_{j,p}],\]
for $1\le j\le p_\bmu$. In fact, as noted above, it is enough to
compute the splittable decomposition numbers. That is, the
$d^{(j)}_{\blam\bmu}$ such that $p_\blam=p_\bmu$, for
$\blam,\bmu\in\Part[d,\b]$.

Before we start to compute the decomposition numbers $d^{(j)}_{\blam\bmu}$
we introduce some new notation.  If $A$ is any finite dimensional algebra
let $\Groth(A)$ be the Grothendieck group of finitely generated
$A$-modules. If $M$ is an $A$-module let $[M]$ be the image of~$M$
in~$\Groth(A)$. In particular, note that the Grothendieck group of
$\Groth(\Sch)$ is equipped with two distinguished bases:
\[\set{[\Delta(\blam)]|\blam\in\Part}\quad\text{and}\quad
  \set{[L(\blam)|\blam\in\Part}.\]
Similar remarks apply to the Grothendieck groups of the cyclotomic Schur
algebras~$\Sch[d,\b]$ and~$\Schrpn(\b)$, for~$\b\in\Comp$.

Fix integers $l$ and $m$ such that $p=lm$ and suppose that
$\bmu\in\Part[d,\b]$, for some $\b\in\Comp$.  Then a multipartition
$\bmu$ is \textbf{$l$-symmetric} if
\[\bmu =\bnu^{\,l}:= (\underbrace{\bnu,\dots,\bnu}_{l\text{ times}}),\]
for some multipartition $\bnu\in\Part[r/l,n/l]$. Note that if
$d^{(j)}_{\blam\bmu}$ is an $l$-splittable decomposition number then $\blam$
and $\bmu$ are both $l$-symmetric multipartitions.

Let $\Part[d,\b]^{\,l}$ be the set of $l$-symmetric multipartitions in
$\Part[d,\b]$. It is easy to see that
\[\Part[d,\b]^{\,l} = \set{\bmu|\bmu\in\Part[d,\b]\text{ and }\o_\bmu|m}.\]
If $\Part[d,\b]^{\,l}$ is non-empty then $\o_\b|m$ and we define
$\b_m=(b_1,\dots,b_m)$. If $\bmu\in\Part[d,\b]^{\,l}$ define
$\bmu_m=(\bmu^{[1]},\dots,\bmu^{[m]})$.  Then
$\bmu_m\in\Part[r/l,\b_m]\subseteq\Part[r/l,n/l]$. It is easy to check that
the map $\bnu\mapsto\bnu^{\,l}$ defines a bijection from $\Part[r/l,\b_m]$ to
$\Part[d,\b]^{\,l}$, with the inverse map being given by $\bmu\mapsto \bmu_m$.

We now return to our main task of computing splittable decomposition
numbers. We will do this by deriving a system of equations which
uniquely determine the decomposition numbers $d^{(j)}_{\blam\bmu}$,
for $1\le j\le l=p_\blam$.

For the rest of this subsection fix $\blam\in\Part[d,\b]$
and set $m=\o_\blam$ and $l=p_\blam$. Then
$\b_m=(b_1,\dots,b_m)\in\Comp[r/l,n/l]$ and $\blam_m\in\Part[r/l,\b_m]$.
By (\ref{Schur iso}) the cyclotomic  Schur algebras $\Sch[r/l,\b_m]$ and
$\Sch[d,\b]$ are related by
\[\Sch[r/l,\b_m]\cong\Sch[d,b_1]\otimes\dots\otimes\Sch[d,b_m]
\quad\text{and}\quad
  \Sch[d,\b]\cong\big(\Sch[r/l,\b_m]\big)^{\otimes l}.\]
For $\bmu\in\Part[d,\b]$ let
$d_{\blam_m\bmu_m}=[\Delta_{\b_m}(\blam_m):L_{\b_m}(\bmu_m)]$ be the
corresponding decomposition number for the cyclotomic Schur algebra $\Sch[r/l,\b_m]$. Since
\[\Delta_{\b_m}(\blam_m)\cong\Delta(\blam^{[1]})\otimes\dots
               \otimes\Delta(\blam^{[m]})
               \quad\text{and}\quad
L_{\b_m}(\bmu_m)\cong L(\bmu^{[1]})\otimes\dots\otimes L(\bmu^{[m]})\]
we have that
\begin{equation}\label{dec prod}
d_{\blam_m\bmu_m}=\prod_{i=1}^m[\Delta(\blam^{[i]}):L(\bmu^{[i]})]
         =d_{\blam^{[1]}\bmu^{[1]}}\dots d_{\blam^{[m]}\bmu^{[m]}},
\end{equation}
where
$d_{\blam^{[i]}\bmu^{[i]}}=[\Delta(\blam^{[i]}):L(\bmu^{[i]})]$, for
$1\le i\le m=\o_\blam$.

Recall that if $\bmu\in\Part[d,\b]$ then
$p_{\b/\bmu}=p_\b/p_\bmu=\o_{\bmu}/\o_\b$.  If $\bmu\in\Part[d,\b]^{\,l}$
is $l$-symmetric then $\o_\bmu$ divides $m$, so we define
$p_{\bmu/\blam}=p_\bmu/p_\blam$. Then $p_{\bmu/\blam}\in\N$ and
$p_{\bmu/\blam}=\o_\blam/\o_\bmu=p_{\b/\blam}/p_{\b/\bmu}$.

\begin{Lemma}\label{relations}
    Suppose that $\blam\in\Part[d,\b]$, $l=p_\blam$ and  $m=\o_\blam$. Then:
    \begin{enumerate}
       \item $[\Delta_{\b_m}(\blam_m)]
         =\Sum_{\bnu\in\Part[d,\b]^{\,l}} d_{\blam_m\bnu_m} [L_{\b_m}(\bnu_m)].$
     \item $[\Delta^\blam_{0,p}]=\Sum_{\bnu\in\Part[d,\b]}\sum_{1\le j\le p_\bnu}
             d_{\blam\bnu}^{(j)}[L^\bnu_{j,p}].$
      \item If $\bmu\in\Part[d,\b]^{\,l}$ then
         $d^{(1)}_{\blam\bmu}+d^{(2)}_{\blam\bmu}+\dots
     +d^{(l)}_{\blam\bmu}=p_{\bmu/\blam}d_{\blam_m\bmu_m}^{\,l}$.
    \end{enumerate}
\end{Lemma}

\begin{proof}
  Part~(a) is just a rephrasing of the definition of decomposition numbers
  combined with the bijection
  $\Part[d,\b]^{\,l}\bijection\Part[r/l,\b_m];\bmu\mapsto\bmu_m$. Part~(b) follows similarly.

  Suppose that $\bmu\in\Part[d,\b]$. We prove (c) by computing the
  decomposition multiplicity of $L_\b(\bmu)$ on both sides of part~(b) upon
  restriction to~$\Sch[d,\b]$. By Corollary~\ref{cor716},
  \[
    \Delta^{\blam}_{0,p}\Res^{\Schrpn(\b)}_{\Sch[d,\b]}
     \cong \Delta_\b(\blam)\oplus\Delta_\b(\blam)^{\vartheta_\b^{-1}} \oplus\dots
       \oplus\Delta_\b(\blam)^{\vartheta_\b^{1-p_{\b/\blam}}}.
  \]
  Now, every composition factor of $\Delta_\b(\blam)$ is isomorphic to
  $L_\b(\bnu)$, for some $\bnu\in\Part[d,\b]$, and
  $L_\b(\bnu)^{\vartheta_\b}\cong L_{\b\shift{\o_\b}}(\bnu)$ by Lemma~\ref{twist1}.
  Therefore, the decomposition multiplicity of $L_\b(\bmu)$ in
  $\Delta^{\blam}_{0,p}\Res^{\Schrpn(\b)}_{\Sch[d,\b]}$ is
  \[
  \frac{p_{\b/\blam}}{p_{\b/\bmu}}[\Delta_\b(\blam):L_\b(\bmu)]
          =p_{\bmu/\blam}d_{\blam_m,\bmu_m}^{\,l},
  \]
  where the second equality follows from (\ref{dec prod}).

  Now consider the multiplicity of $L_\b(\bmu)$ on the right hand side of~(b).
  If $\bnu\in\Part[d,\b]$ and $1\le j\le p_\bnu$ then, using
  Corollary~\ref{cor716} again,
  \[
    L^{\blam}_{j,p}\Res^{\Schrpn(\b)}_{\Sch[d,\b]}
     \cong L_\b(\blam)\oplus L_\b(\blam)^{\vartheta_\b^{-1}} \oplus\dots
       \oplus L_\b(\blam)^{\vartheta_\b^{1-p_{\b/\bnu}}}.
  \]
  Therefore, $[L^{\blam}_{j,p}\Res^{\Schrpn(\b)}_{\Sch[d,\b]}:L_\b(\bmu)]=1$
  by Lemma~\ref{twist1}. Equating the multiplicity of $L_\b(\bmu)$ on
  both sides of~(b) now gives~(c).
\end{proof}

Lemma~\ref{relations} gives our first relation satisfied by the decomposition
numbers $d_{\blam\bmu}^{(j)}$. We now use formal characters to find more
relations.  Let $K[\Part]$ be the $K$-vector space with basis
$\set{e^\bmu|\bmu\in\Part}$. The ($K$-valued) \textbf{formal character} of
the $\Sch[d,\b]$-module $M$ is
\[\cha M = \sum_{\bmu\in\Part[d,\b]}(\dim M_\bmu)e^\bmu,\]
an element of $K[\Part]$.  The coefficients appearing in the formal
characters are the traces of the identity maps on the weight spaces. We
need a more general version of the formal character which records the
traces of powers of $\vartheta_\blam^t$ on certain
weight spaces, for $1\le t<l=p_\blam$.

Fix an integer $t$ with $1\le t<p_\blam$. Let $l_t=\gcd(t,l)$ be the
greatest common divisor of~$t$ and~$l$ and set $\ell_t=l/l_t$. By
convention, we set $l_0=l$. Then $r/\ell_t=dml_t$ so that
$K[\Part[dml_t,n/\ell_t]]=K[\Part[r/\ell_t,n/\ell_t]]$.

Now suppose that $M$ is an $\Schrpn(\b)$-module and that $\bgam=\gamma^{\ell_t}$
is an $\ell_t$-symmetric multipartition. We will show in Lemma~\ref{twining}
below that $\vartheta_\blam^t$ stabilizes each $\ell_t$-symmetric weight 
space~$M_{\gamma^{\ell_t}}$. With this in mind, we define the 
\textbf{twining character} of~$M$ to be
\[
\chat M=\sum_{\gamma\in\Part[r/\ell_t,n/\ell_t]}
       \Tr\big(\vartheta_\blam^t, M_{\gamma^{\ell_t}}\big)e^\gamma
       \in K[\Part[r/\ell_t,n/\ell_t]].
\]
It is easy to see that, just like the usual character, the twining
character lifts to a well-defined map
$\chat\map{\Groth(\Schrpn(\b))}K[\Part[r/\ell_t,n/\ell_t]]$ on
the Grothendieck group of $\Schrpn(\b)$.

The following Lemma will allow us to compute the twining character
$\chat$ on both sides of Lemma~\ref{relations}(b).

\begin{Lemma}\label{twining}
  Suppose that $\blam\in\Part[d,\b]$ and $1\le t<l=p_\blam$. Then
  \[{\chat}\Delta^\blam_{i,p}
            =\eps^{itm}p_{\b/\blam}\gscal^t\cha\Delta_{\b_{l_tm}}(\blam_{l_tm}),\]
  for $1\le i\le p_\blam$. Moreover, if $\bmu\in\Part[d,\b]^{\,l}$ and $1\le j\le p_\bmu$ then
  \[{\chat}L^\bmu_{j,p} =\eps^{jtm}p_{\b/\bmu}
            \gscal[\bmu]^{tp_{\bmu/\blam}}\cha L_{\b_{l_tm}}(\bmu_{l_tm}).\]
\end{Lemma}

\begin{proof}
  We only prove the formula for $\chat L^\bmu_{j,p}$ and leave the
  almost identical calculation of $\chat\Delta^\blam_{0,p}$ to the
  reader. To ease the notation let $m'=\o_\bmu$ so that
  $\b_{m'}=(\b^{[1]},\dots,\b^{[m']})$ and
  $\bmu_{m'}=(\bmu^{[1]},\dots,\bmu^{[m']})\in\Part[r/p_\bmu,\b_{m'}]$.

  To determine $\chat L^\bmu_{j,p}$ for each
  $\gamma\in\Part[r/\ell_t,n/\ell_t]$ we need to compute
  \begin{align*}
    \Tr\big(\vartheta_\blam^t, (L^\bmu_{j,p})_{\gamma^{\ell_t}}\big)&=\Tr\big(\vartheta_\b^{tp_{\b/\blam}}, (L^\bmu_{j,p})_{\gamma^{\ell_t}}\big)
  =\Tr\big((\vartheta_\b^{p_{\b/\bmu}})^{tp_{\bmu/\blam}}, (L^\bmu_{j,p})_{\gamma^{\ell_t}}\big)\\
  &=\Tr\big(\vartheta_\bmu^{tp_{\bmu/\blam}}, (L^\bmu_{j,p})_{\gamma^{\ell_t}}\big).
  \end{align*}
  By Corollary~\ref{cor716} we can identify $L^\bmu_{j,p}$ with the $K$-vector
  space
  \[L_{\b}(\bmu)\oplus L_{\b}(\bmu)\vartheta_\b
      \oplus\dots\oplus L_{\b}(\bmu)\vartheta_\b^{p_{\b/\bmu}-1},
  \]
  where the action of $\Schrpn(\b)$ on $L^\bmu_{j,p}$ is determined by
  \begin{enumerate}
      \item $L^\bmu_{j,p}\Res^{\Schrpn(\b)}_{\Sch[d,\b]}\cong L_{\b}(\bmu)\oplus
    L_{\b}(\bmu)^{\vartheta_\b^{-1}}
       \oplus\dots \oplus L_{\b}(\bmu)^{\vartheta_\b^{1-p_{\b/\bmu}}}$,
  \item $(x\vartheta_\b^a)\vartheta_\b^c=x\vartheta_\b^{a+c}$,
    for all $x\in L_{\b}(\bmu)$ and $a,c\in\Z$,
  \item $\vartheta_{\bmu}$ acts as the scalar
       $\eps^{j\o_\bmu} \gscal[\bmu]$ on the highest weight vector of $L_{\b}(\bmu)$.
  \end{enumerate}
  Note that $p_{\bmu/\blam}=m/m'\in\N$, since $\bmu\in\Part[d,\b]^{\,l}$, and
  $\vartheta_\bmu=\vartheta_\b^{p_{\b/\bmu}}=\vartheta_\blam^{p_{\blam/\bmu}}$.
  Therefore,
  \[ \Tr\big(\vartheta_\blam^t, (L^\bmu_{j,p})_{\gamma^{\ell_t}}\big)
      =\Tr\big(\vartheta_\bmu^{tp_{\bmu/\blam}}, (L^\bmu_{j,p_{\bmu}})_{\gamma^{\ell_t}}\big)
      =p_{\b/\bmu}\Tr\big(\vartheta_\bmu^{tp_{\bmu/\blam}}, L_{\b}(\bmu)_{\gamma^{\ell_t}}\big).
  \]
To compute this trace first observe that if $\bar\phi_{\t^\bmu}$ is the highest weight
vector of $L_\b(\bmu)$ then, by~(c) above (which comes from Corollary~\ref{cor716}),
\begin{equation}\label{high weight}
\bar\phi_{\t^\bmu}\vartheta_\blam^{t}=\eps^{jtm}\gscal[\bmu]^{tp_{\bmu/\blam}}
               \bar\phi_{\t^\bmu}.
\end{equation}
Now, $p=\ell_t l_t m=\ell_t l_t p_{\bmu/\blam}m'$ so we can identify the
two modules $L_{\b}(\bmu)$ and $L_{\b_{l_t m}}(\bmu_{l_t m})^{\otimes\ell_\t}$.
Using Lemma~\ref{commutrel2}, if $1\le j\le p/\ell_t$ then
\begin{equation}\label{tm push}
\varphi_{\Stab\Ttab}^{(j)}\vartheta_\blam^t
         =\eps^{-mtk}\vartheta_\blam^t\varphi_{\Stab\Ttab}^{(tm+j)}
\end{equation}
for some $k\in\Z$, where we identity $\varphi_{\Stab\Ttab}^{(j)}$ and
$\varphi_{\Stab\Ttab}^{(j')}$ if $j\equiv j'\pmod{p}$. Therefore,
since~$\bar\phi_{\t^\bmu}$ generates $L_\b(\bmu)$, it follows from
$(\ref{high weight})$ and (\ref{tm push}) that each simple $p$-tensor
\[
\bbeta=(x_1^{(1)}\otimes\cdots\otimes x_{l_tm}^{(1)})\otimes\dots
       \otimes(x_1^{(\ell_t)}\otimes\cdots \otimes x_{l_tm}^{(\ell_t)})
\]
in $L_{\b}(\bmu)_{\gamma^{\ell_t}}$ is mapped by
$\vartheta_\blam^t=\vartheta_\bmu^{tp_{\bmu/\blam}}$ to a scalar multiple
of
\[
(x_1^{(tm+1)}\otimes\cdots\otimes x_{l_tm}^{(tm+1)})
     \otimes\cdots\otimes(x_1^{(tm+\ell_t)}\otimes\cdots
     \otimes x_{l_tm}^{(tm+\ell_t)}),
\]
where we identity $x_i^{(j)}=x_i^{(j')}$ whenever $j\equiv
j'\pmod{\ell_t}$ for $1\le i\le l_tm$. Thus, to calculate
$\Tr\big(\vartheta_\blam^t, L_{\b}(\bmu)\big)$ we only need to consider the case when $x_i^{(s)}=x_i^{(tm+s)}$, for all
$1\le i\le l_tm$ and all $1\le s\le \ell_t$. By construction,
$(tm)/(l_tm)\not\equiv0\pmod{\ell_t}$, so this can only happen if
\[ x_i^{(s)}=x_i^{(s')},\quad\text{whenever }
        1\leq i\leq l_tm \text{ and }1\leq s, s'\leq\ell_t.
\]
Consequently, $\bbeta$ contributes to the twining character only if
$\bbeta=\beta\otimes\dots\otimes\beta$ ($\ell_t$~times),  for some
$\beta\in L_{\b_{l_t m}}(\bmu_{l_t m})$. Notice that if
$\beta\in L_{\b_{l_t m}}(\bmu_{l_t m})_\gamma$, for some
$\gamma\in\Part[r/\ell_t,n/\ell_t]$ then
$\bbeta\in L_\b(\bmu)_{\gamma^{\ell_t}}$. In particular, this shows that
$\vartheta_\blam^t$ stabilizes $L_{\b_{l_t m}}(\bmu_{l_t m})_\gamma$ as we
claimed when introducing the twining character.

In (\ref{high weight}) we have already shown that $\vartheta_\blam^t$ acts
as multiplication by $\eps^{jtm}\gscal[\bmu]^{tp_{\bmu/\blam}}$ on the
highest weight vector of $L_{\b_{l_tm}}(\bmu_{l_tm})^{\otimes\ell_t}$. On
the other hand, by (\ref{tm push}) and abusing the notation of
Lemma~\ref{commutrel2} slightly, if $1\le j\le\ell_t$ then
\[\big(\varphi_{\Stab\Ttab}^{(j)}\big)^{\otimes\ell_t}\vartheta_\blam^t
  =\eps^{-mt\ell_tk}\vartheta_\blam^t\big(\varphi_{\Stab\Ttab}^{(j)}\big)^{\otimes\ell_t}
  =\vartheta_\blam^t\big(\varphi_{\Stab\Ttab}^{(j)}\big)^{\otimes\ell_t},\]
where the last equality follows because $mt\ell_t=p(t/l_t)$ is divisible by~$p$. Therefore,
writing $\beta^{\otimes\ell_t}=\bar\phi_{\t^\bmu}\phi^{\otimes\ell_t}$,
for some $\phi\in\Sch[l_tm,\b_{l_tm}]$, we have that
\[\beta^{\otimes\ell_t}\vartheta_\blam^t
       =\bar\phi_{\t^\bmu}\phi^{\otimes\ell_t}\vartheta_\blam^t
       =\bar\phi_{\t^\bmu}\vartheta_\blam^t\phi^{\otimes\ell_t}
       =\eps^{jtm}\gscal[\bmu]^{tp_{\bmu/\blam}}\bar\phi_{\t^\bmu}\phi^{\otimes\ell_t}
       =\eps^{jtm}\gscal[\bmu]^{tp_{\bmu/\blam}}\beta^{\otimes\ell_t},\]
where the third equality uses (\ref{high weight}). Consequently,
\[
  \Tr\big(\vartheta_\blam^t, \big(L^\bmu_{j,p}\big)_{\gamma^{\ell_t}}\big)
     =p_{\b/\bmu}\eps^{jtm}\gscal[\bmu]^{tp_{\bmu/\blam}}
     \dim L_{\b_{l_tm}}(\bmu_{l_tm})_\gamma.
\]
Summing over $\Part[d,\b]^{\,l}$ gives the desired formula for
$\chat\big(L^\bmu_{j,p}\big)$ and completes the proof.
\end{proof}

\begin{cor}\label{decomposition relations0}
  Suppose that $\blam,\bmu\in\Part[d,\b]^{\,l}$, and $0\le t<l=p_\blam$, $l'=p_{\bmu}$. Then
  in~$K$
  \[p_{\bmu/\blam}\Big(\frac{\gscal}{\gscal[\bmu]^{p_{\bmu/\blam}}}\Big)^t
       d_{\blam_m,\bmu_m}^{\,l_t}
     =\eps^{tm}d^{(1)}_{\blam\bmu}+\eps^{2tm}d^{(2)}_{\blam\bmu}+\dots
              +\eps^{l'tm}d^{(l')}_{\blam\bmu}.\]
\end{cor}

\begin{proof}
    If $t=0$ then the result is just Lemma~\ref{relations}(c). If
    $t\ne1$ then combining Lemma~\ref{twining} and Lemma~\ref{relations}(b) shows that
    \[\cha\Delta_{\b_m}(\blam_m)^{\otimes l_t}
         =\sum_{\bmu\in\Part[d,\b]^{\,l}}\sum_{1\le j\le p_\bmu}\eps^{jmt}d^{(j)}_{\blam\bmu}
     \frac{p_{\b/\bmu}\gscal[\bmu]^{tp_{\bmu/\blam}}}{p_{\b/\blam}\gscal^t}
        \cha L_{\b_m}(\bmu_m)^{\otimes l_t}.
    \]
    On the other hand, by Lemma~\ref{relations}(a),
    \[\cha\Delta_{\b_m}(\blam_m)^{\otimes l_t}
         =\sum_{\bmu\in\Part[d,\b]^{\,l}}d_{\blam_m\bmu_m}^{\,l_t}\cha L_{\b_m}(\bmu_m)^{\otimes l_t}.\]
   As the characters $\{\cha L_{\b_m}(\bnu_m)\}$ are linearly independent,
   comparing the coefficient of~$\cha L_{\b_m}(\bmu_m)$ on both sides gives
   the result.
\end{proof}

\begin{cor}\label{decomposition relations}
  Suppose that $l$ divides $p$, $\blam,\bmu\in\Part[d,\b]^{\,l}$ and that
  $p_\blam=l=p_{\bmu}$. If $0\le t<l$ then, in~$K$,
  \[\Big(\frac{\gscal}{\gscal[\bmu]}\Big)^td_{\blam_m\bmu_m}^{\,l_t}
     =\eps^{tm}d^{(1)}_{\blam\bmu}+\eps^{2tm}d^{(2)}_{\blam\bmu}+\dots
              +\eps^{ltm}d^{(l)}_{\blam\bmu}.\]
\end{cor}

We can now complete the proof of the main results of this paper. Recall from
just before Theorem~\ref{splittable} in the introduction that we defined
matrices $V(l)$ and $V_i(l)$, whenever~$l$ divides~$p$ and $1\le i \le l$.

\begin{Theorem} \label{8mainthm3}
    Suppose that $\blam,\bmu\in\Part[d,\b]$ and $p_{\lam}=l=p_\blam$, for
    some $\b\in\Comp$. Then, for $1\le j\le p_\blam$,
    \[[\Delta^\blam_{0,p}:L^\bmu_{j,p}]
               \equiv \frac{\det V_j(l)}{\det V(l)}\pmod{\Char K}.\]
    In particular, $[\Delta^\blam_{0,p}:L^\bmu_{j,p}]=\frac{\det V_j(l)}{\det V(l)}$
    if $K$ is a field of characteristic zero.
\end{Theorem}

\begin{proof}
  By Corollary~\ref{decomposition relations} the decomposition numbers
  $d^{(1)}_{\blam\bmu},\dots,d_{\blam\bmu}^{(l)}$ satisfy the matrix
  equation
  \[ V(l)\begin{pmatrix}d^{(1)}_{\blam\bmu}\\\vdots\\d_{\blam_m\bmu_m}^{(l)}\end{pmatrix}
    =\begin{pmatrix}\big(\frac{\gscal}{\gscal[\bmu]}\big)^0d_{\blam_m\bmu_m}^{\,l_0}\\\vdots\\
         \big(\frac{\gscal}{\gscal[\bmu]}\big)^{l-1} d_{\blam_m,\bmu_m}^{\,l_{(l-1)}}\end{pmatrix}
  \]
  Hence, the theorem follows by Cramer's rule.
\end{proof}

Observe that the condition $p_\blam=l=p_\bmu$ says that the decomposition
numbers $[\Delta^\blam_{i,p}:L^\bmu_{j,p}]$ are $l$-splittable, for
$1\le i,j<l$. Moreover,
$[\Delta^\blam_{i,p}:L^\bmu_{j,p}]=[\Delta^\blam_{0,p}:L^\bmu_{j-i,p}]$ by
\eqref{cyclic}. Hence, by Corollary~\ref{cor321} and Theorem~\ref{8mainthm3} we
have computed all of  the $l$-splittable decomposition numbers of $\Schrpn(\b)$
and~$\Hrpn$.

\begin{cor}\label{l-splittable}
  Suppose that $\blam\in\Part[d,\b]$, $\bmu\in\Klesh[d,\b]$, for some $\b\in\Comp$,
  and that $p_\blam=p_\bmu=l$. Then, for $1\le i,j\le p_\blam$,
    \[[S^\blam_i:D^\bmu_j]=[\Delta^\blam_{i,p}:L^\bmu_{j,p}]
    \equiv \frac{\det V_{j-i}(l)}{\det V(l)}\pmod{\Char K}.\]
\end{cor}

In particular, this establishes Theorem~\ref{splittable} from the
introduction.  Finally, we are able to prove Theorem~\ref{main}, our Main
Theorem from the introduction.

\begin{proof}[of Theorem~\ref{main}]
  By Theorem~\ref{T:SplittableRedction} the decomposition numbers of
    $\Hrpn$ are completely determined by the $l$-splittable
    decomposition numbers of the Hecke algebras $\H_{s,l,m}$, where~$l$
    divides~$p$, $1\le s\le r$ and $1\le m\le n$. Hence,
    Theorem~\ref{main} follows from Corollary~\ref{l-splittable}.
\end{proof}

We remind the reader that the polynomials $\dgscal$, for $\blam\in\Part[,\b]$,
are determined by Proposition~\ref{P:gscal} and Remark~\ref{R:gscal}.
Hence, this result explicitly determines the $l$-splittable decomposition
numbers of $\Schrpn$ (and of $\Hrpn$).

When $K$ is a field of positive characteristic the results above only
determine the $l$-splittable decomposition numbers of $\Schrpn$ and
$\Hrpn$ modulo the characteristic of~$K$.

\appendix
\def\theequation{\Alph{section}\arabic{equation}}
\def\thesubsection{\Alph{section}\arabic{subsection}}

\section{Technical calculations for $v_\b$}
  In Chapter~2 we omitted the proofs of Propositions~\ref{changing}
  and \ref{pleftmult} and Lemma~\ref{vb-} because their proofs are long and
  uninspiring calculations. This appendix proves these three results.

\subsection{Proof of Proposition~\ref{changing}}
  We start by proving Proposition~\ref{changing}, which gives several different
  expressions for the element~$v_\b$ from Definition~\ref{vb defn}.

We need the following fact which is generalisation of a fundamental
result of Dipper and James~\cite[Lemma~3.10]{DJ:BMorita}.

\begin{Lemma}[\!\!(\protect{\cite[Proposition~3.4]{DM:Morita}})]\label{DM}
  Suppose that $a$, $b$, $s$ and $t$ are positive integers with $1\le
  a+b<n$ and $1\le s\le t\le p$. Let
  $v^{(s,t)}_{a,b}=\LL[s,t]_{1,a}T_{a,b}\LL[t+1,s-1]_{1,b}$.
  Then
  \[T_i v_{a,b}^{(s,t)} = v_{a,b}^{(s,t)} T_{(i)w_{a,b}}\qquad\text{and}\qquad
  L_j v_{a,b}^{(s,t)}= v_{a,b}^{(s,t)} L_{(j)w_{a,b}},\]
for all $i,j$ such that $1\le i,j\le a+b$ and $i\ne a,a+b$.
\end{Lemma}

Recall from Section~\ref{S:Hrpn} that $\b\shift k=(b_{k+1},b_{k+2},\dots,b_{k+p})$
if $\b\in\Comp$ and $k\in\Z$, where we set $b_{i+p}=b_i$ for $1\le i\le p$.

\begin{Lemma}\label{technical}
    Suppose that $\b\in\Comp$ and that $1\le j\le s\le p$. Then
    \begin{align*}
    \prod_{j\le k<s}&\LL[j,k]_{1,b_{k+1}} T_{b_{k+1},\b_j^k}
        \cdot\prod_{j<k\le p}\LL[k]_{1,\b_j^{k-1}}
    \cdot\prod_{1\le i<j}\LL[i]_{1,\b_{i+1}^p}\\
    &=\prod_{j+1\le k<s}\LL[j+1,k]_{1,b_{k+1}} T_{b_{k+1},\b_{j+1}^k}
    \cdot\LL[j]_{1,\b_{j+1}^s}T_{b_{j+1}^s,b_j}\prod_{1\le i<j}\LL[i]_{1,\b_{i+1}^p}
          \cdot\prod_{j<k\le p}\LL[k]_{1,\b_j^{k-1}},
    \end{align*}
   where all products are read from left to right with
   \underline{decreasing} values of $i$ and $k$.
\end{Lemma}

\begin{proof}Let $L(s)$ and $R(s)$, respectively, be the left and right
    hand side of the formula in the statement of the Lemma. We show
    that $L(s)=R(s)$ by induction on $s$. To start the induction observe that, by
    our conventions,
    \[L(j)=\prod_{j<k\le p}\LL[k]_{1,b_j^{k-1}}\cdot
          \prod_{1\le i<j}\LL[i]_{1,\b_{i+1}^p}=R(j).\]
    Hence, the Lemma is true when $s=j$. If $j\le s<p$ then, by induction,
    \begin{align*}
     L(s+1)&=\LL[j,s]_{1,b_{s+1}} T_{b_{s+1},\b_j^s}L(s)
             =\LL[j,s]_{1,b_{s+1}} T_{b_{s+1},\b_j^s}R(s)\\
      &=\LL[j,s]_{1,b_{s+1}}T_{b_{s+1},\b_j^s}
    \proD{j+1\le k<s}\LL[j+1,k]_{1,b_{k+1}} T_{b_{k+1},\b_{j+1}^k}
    {\cdot}\LL[j]_{1,\b_{j+1}^s}T_{\b_{j+1}^s,b_j}
      \proD{1\le i<j}\LL[i]_{1,\b_{i+1}^p}\cdot\proD{j<k\le p}\LL[k]_{1,\b_j^{k-1}}\\
  &=\LL[j,s]_{1,b_{s+1}}T_{b_{s+1},\b_j^s}
    \prod_{j+1\le k<s}\LL[j+1,k]_{1,b_{k+1}} T_{b_{k+1},\b_{j+1}^k}
    \cdot\prod_{1\le i<j}\LL[i]_{\b_{j+1}^s+1,\b_{i+1}^p}\\
  &\qquad\times v^{(1,j)}_{\b_{j+1}^s,b_j}\proD{j+1<k\le p}\LL[k]_{b_j+1,\b_j^{k-1}},
  \intertext{since $T_{a,b}$ commutes with $\LL[i]_{1,k}$ by Lemma~\ref{JM
  properties} whenever $a+b\le k$ and $1\le i\le p$ (we use this fact
  several times below).
Therefore, using Lemma~\ref{wab} and
  Lemma~\ref{DM},}
  L(s+1)&=\LL[j,s]_{1,b_{s+1}}T_{b_{s+1},\b_{j+1}^s}
         T_{b_{s+1},b_j}^{\shift{\b_{j+1}^s}}
    \prod_{j+1\le k<s}\LL[j+1,k]_{1,b_{k+1}} T_{b_{k+1},\b_{j+1}^k}
    \cdot\prod_{1\le i<j}\LL[i]_{\b_{j+1}^s+1,\b_{i+1}^p}\\
    &\qquad\times\LL[s+1,p]_{1,\b_{j+1}^s}
    v^{(1,j)}_{\b_{j+1}^s,b_j}\proD{j+1<k\le s}\LL[k]_{b_j+1,\b_j^{k-1}}
          \cdot\proD{s+1<k\le p}\LL[k]_{\b_j^s+1,\b_j^{k-1}}\\
     &=\LL[j]_{1,b_{s+1}}v^{(j+1,s)}_{b_{s+1},\b_{j+1}^s}
    \prod_{j+1\le k<s}\LL[j+1,k]_{1,b_{k+1}} T_{b_{k+1},\b_{j+1}^k}
    \cdot\prod_{1\le i<j}\LL[i]_{\b_{j+1}^s+1,\b_{i+1}^p}\\
    &\qquad\times T_{b_{s+1},b_j}^{\shift{\b_{j+1}^s}}
 T_{\b_{j+1}^{s},b_j}\LL[j+1,p]_{1,b_j}
         \proD{j+1<k\le s}\LL[k]_{b_j+1,\b_j^{k-1}}
          \cdot\proD{s+1<k\le p}\LL[k]_{\b_j^s+1,\b_j^{k-1}}\\
    &=v^{(j+1,s)}_{b_{s+1},\b_{j+1}^s}\LL[j]_{\b_{j+1}^s+1,\b_{j+1}^{s+1}}
    \prod_{j+1\le k<s}\LL[j+1,k]_{1,b_{k+1}} T_{b_{k+1},\b_{j+1}^k}
    \cdot\prod_{1\le i<j}\LL[i]_{\b_{j+1}^s+1,\b_{i+1}^p}\\
    &\qquad\times T_{\b_{j+1}^{s+1},b_j}\LL[j+1,p]_{1,b_j}
         \proD{j+1<k\le s}\LL[k]_{b_j+1,\b_j^{k-1}}
          \cdot\proD{s+1<k\le p}\LL[k]_{\b_j^s+1,\b_j^{k-1}}\\
    &=\prod_{j+1\le k<s+1}\LL[j+1,k]_{1,b_{k+1}} T_{b_{k+1},\b_{j+1}^k}
    \cdot\prod_{1\le i<j}\LL[i]_{\b_{j+1}^{s+1}+1,\b_{i+1}^p}\\
    &\qquad\times \LL[s+1,p]_{1,\b_{j+1}^s}v^{(1,j)}_{\b_{j+1}^{s+1},b_j}
         \proD{j+1<k\le s}\LL[k]_{b_j+1,\b_j^{k-1}}
          \cdot\proD{s+1<k\le p}\LL[k]_{\b_j^s+1,\b_j^{k-1}}\\
     &=R(s+1),
\end{align*}
where the two lines we have, in essence, reversed some of the previous
steps. This completes the proof.
\end{proof}

We are now ready to prove Proposition~\ref{changing}. This result includes the definition of
$v_\b$ as the special case $j=1$.
For the reader's convenience we restate the result.

\begin{prop}\label{changingII}
  Suppose that $\b\in\Comp$ and $1\le j\le p$. Then
  \[v_\b = \prod_{j\le k<p}\LL[j,k]_{1,b_{k+1}} T_{b_{k+1},\b_j^k}
  \cdot\prod_{1\le i<j}\LL[i]_{1,\b_{i+1}^p}
          \cdot\prod_{j<k\le p}\LL[k]_{1,\b_j^{k-1}}
          \cdot\prod_{1<i\le j}T_{\b_{i}^p,b_{i-1}}\LL[i,p]_{1,b_{i-1}},
   \]
   where all products are read from left to right with
   \underline{decreasing} values of $i$ and $k$.
\end{prop}

\begin{proof} We argue by induction on $j$. When $j=1$ the Lemma is a
  restatement of Definition~\ref{vb defn}, so there is nothing to prove.
  Suppose now that $1\le j<p$ and that the formula in the Proposition holds.
  Then by induction and Lemma~\ref{technical} (with $s=p$), we see that
  \begin{align*}
  v_\b&= \prod_{j+1\le k<p}\LL[j,k]_{1,b_{k+1}} T_{b_{k+1},\b_j^k}
          \cdot\prod_{1\le i<j}\LL[i]_{1,\b_{i+1}^p}
          \cdot\prod_{j<k\le p}\LL[k]_{1,\b_j^{k-1}}
          \cdot\prod_{1<i\le j}T_{\b_{i}^p,b_{i-1}}\LL[i,p]_{1,b_{i-1}}\\
      &=\prod_{j+1\le k<s}\LL[j+1,k]_{1,b_{k+1}} T_{b_{k+1},\b_{j+1}^k}
    \cdot\LL[j]_{1,\b_{j+1}^p}T_{b_{j+1}^p,b_j}\prod_{1\le i<j}\LL[i]_{1,\b_{i+1}^p}
          \cdot\prod_{j<k\le p}\LL[k]_{1,\b_j^{k-1}}\\
      &\qquad\times\prod_{1<i\le j}T_{\b_{i}^p,b_{i-1}}\LL[i,p]_{1,b_{i-1}}\\
    &=\prod_{j+1\le k<p}\LL[j+1,k]_{1,b_{k+1}} T_{b_{k+1},\b_{j+1}^k}
    \cdot\prod_{1\le i<j}\LL[i]_{\b_{j+1}^p+1,\b_{i+1}^p}
    \cdot v^{(1,j)}_{\b_{j+1}^p,b_j}
    \cdot\prod_{j+1<k\le p} \LL[k]_{b_j+1,\b_j^{k-1}}\\
    &\qquad\times\prod_{1<i\le j}T_{\b_{i}^p,b_{i-1}}\LL[i,p]_{1,b_{i-1}}\\
    &=\prod_{j+1\le k<p}\LL[j+1,k]_{1,b_{k+1}} T_{b_{k+1},\b_{j+1}^k}
    \cdot\prod_{1\le i<j}\LL[i]_{\b_{j+1}^p+1,\b_{i+1}^p}
    \cdot\prod_{j+1<k\le p} \LL[k]_{1,\b_{j+1}^{k-1}}\\
    &\qquad\times v^{(1,j)}_{\b_{j+1}^p,b_j}
    \prod_{1<i\le j}T_{\b_{i}^p,b_{i-1}}\LL[i,p]_{1,b_{i-1}}\\
    &= \proD{j+1\le k<p}\LL[j+1,k]_{1,b_{k+1}} T_{b_{k+1},\b_{j+1}^k}
          \cdot\proD{1\le i<j+1}\LL[i]_{1,\b_{i+1}^p}
          \cdot\proD{j+1<k\le p}\LL[k]_{1,\b_{j+1}^{k-1}}
          \cdot\proD{1<i\le j+1}T_{\b_{i}^p,b_{i-1}}\LL[i,p]_{1,b_{i-1}},
  \end{align*}
  which is precisely the statement of the Proposition for $j+1$.
\end{proof}

\subsection{Proof of Proposition~\ref{pleftmult}}
Proposition~\ref{pleftmult} is quite an important result because it implies  the
existence of the central element $z_\b\in\H[d,\b]$. See the proof of
Lemma~\ref{vb shift}.

\begin{prop} \label{pleftmultII} Suppose that $\b\in\Comp$. Then
$Y_pY_{p-1}\dots Y_2Y_1=v_\b T_\b$.
\end{prop}

\begin{proof} To prove the Lemma it is enough to show by induction on $t$ that
    \[Y_t\dots Y_1=\LL[1,t-1]_{1,b_t}T_{b_t,\b_1^{t-1}}
    \dots\LL[1,1]_{1,\b_2}T_{b_2,\b_1^1}\LL[2]_{1,b_1}\dots\LL[t]_{1,\b_1^{t-1}}
    \prod_{t<s\le p}\LL[s]_{1,\b_1^t}\cdot T^{\shift{\b_1^{t-1}}}_{b_t,\b_{t+1}^p}\dots
       T_{b_1,\b_2^p}.\]
When $t=1$ the right hand side of this equation is just $Y_1$ so there
is nothing to prove. Now suppose that $1<t<p-1$. Then, by induction and
Lemma~\ref{commutes},
\begin{align*}
    Y_{t+1}\dots Y_1 &=\LL[t+2,t+p]_{1,b_{t+1}}T_{b_{t+1},n-b_{t+1}}\cdot
           \LL[1,t-1]_{1,b_t}T_{b_t,\b_1^{t-1}}\dots\LL[1,1]_{1,\b_2}
           T_{b_2,\b_1^1}\\
     &\qquad\times\space\LL[2]_{1,b_1}\dots\LL[t]_{1,\b_1^{t-1}}
           \prod_{t<s\le p}\LL[s]_{1,\b_1^t}\cdot
             T^{\shift{\b_1^{t-1}}}_{b_t,\b_{t+1}^p}
               \dots T^{\shift{\b_1^1}}_{b_2,\b_3^p}T_{b_1,\b_2^p}\\
    &=\LL[t+2,t+p]_{1,b_{t+1}}
        \cdot T_{b_{t+1},\b_1^t}T^{\shift{\b_1^t}}_{b_{t+1},\b_{t+2}^p}\cdot
           \LL[1,t-1]_{1,b_t}T_{b_t,\b_1^{t-1}}\dots\LL[1,1]_{1,\b_2}
           T_{b_2,\b_1^1}\\
     &\qquad\times\space\LL[2]_{1,b_1}\dots\LL[t]_{1,\b_1^{t-1}}
         \prod_{t<s\le p}\LL[s]_{1,\b_1^t}\cdot
             T^{\shift{\b_1^{t-1}}}_{b_t,\b_{t+1}^p}
               \dots T^{\shift{\b_1^1}}_{b_2,\b_3^p}T_{b_1,\b_2^p}\\
     &=\LL[t+2,p]_{1,b_{t+1}}\cdot v^{(1,t)}_{b_{t+1},\b_1^t}\cdot
           \LL[1,t-1]_{1,b_t}T_{b_t,\b_1^{t-1}}\dots\LL[1,1]_{1,\b_2}
           T_{b_2,\b_1^1}\\
     &\qquad\times\space\LL[2]_{1,b_1}\dots\LL[t]_{1,\b_1^{t-1}}
               \cdot T^{\shift{\b_1^t}}_{b_{t+1},\b_{t+2}^p}
       \cdot T^{\shift{\b_1^{t-1}}}_{b_t,\b_{t+1}^p}
           \dots T^{\shift{\b_1^1}}_{b_2,\b_3^p}T_{b_1,\b_2^p}\\
\intertext{Therefore, by Lemma~\ref{DM} we have}
Y_{t+1}\dots Y_1 &=v^{(1,t)}_{b_{t+1},\b_1^t}\cdot\LL[t+2,p]_{\b_1^t+1,\b_1^{t+1}} \cdot
           \LL[1,t-1]_{1,b_t}T_{b_t,\b_1^{t-1}}\dots\LL[1,1]_{1,\b_2}
           T_{b_2,\b_1^1}\\
     &\qquad\times\space\LL[2]_{1,b_1}\dots\LL[t]_{1,\b_1^{t-1}}
               \cdot T^{\shift{\b_1^t}}_{b_{t+1},\b_{t+2}^p}
       \cdot T^{\shift{\b_1^{t-1}}}_{b_t,\b_{t+1}^p}
           \dots T^{\shift{\b_1^1}}_{b_2,\b_3^p}T_{b_1,\b_2^p}\\
           &=\LL[1,t]_{1,b_{t+1}}T_{b_{t+1},\b_1^t}
    \dots\LL[1,1]_{1,\b_2}T_{b_2,\b_1^1}\LL[2]_{1,b_1}\dots\LL[t+1]_{1,\b_1^t}
    \prod_{t+1<s\le p}\LL[s]_{1,\b_1^{t+1}`}\\
    &\qquad\times T^{\shift{\b_1^t}}_{b_{t+1},\b_{t+2}^p}\dots T_{b_1,\b_2^p},
\end{align*}
completing the proof of our claim. Taking $t=p$ in the claim completes the
proof.
\end{proof}

\subsection{Proof of Lemma~\ref{vb-}}
In this section we prove Lemma~\ref{vb-} and hence complete the proofs of all of
our main results. Recall from section~\ref{S:trace} that $\LH_m$ is the
$R$-submodule of $\Hrn$ spanned by the elements
\[\set{T_wL_1^{a_1}\dots L_{m-1}^{a_{m-1}}|0\le a_1,\dots,a_{m-1}<r
                \text{ and } w\in\Sym_m}.\]
To prove Lemma~\ref{vb-} we need the following result.

\begin{Lemma}\label{bump}
  Suppose that $a,b,k$ and $l$ are positive integers such that
  $k\le l\le a$ and $1\le s\le t\le p$. Then
  \[\LL[s,t]_{k,l}T_{a,b}=T_{a,b}\Big(\LL[s,t]_{b+k,b+l}
    +\sum_{m=b+k}^{b+l}\sum_{e=1}^{d(t-s+1)}h_{m,e}L_m^e\Big),\]
  for some $h_{m,e}\in\LH_m$.
\end{Lemma}

\begin{proof}
  For the duration of this proof let $L_{k,l}(Q)=\prod_{m=k}^{\,l}(L_m-Q)$, for
  $Q\in R$. Then
  $\LL[s,t]_{k,l}=\prod_{i=1}^d\prod_{u=s}^tL_{k,l}(\eps^uQ_i)$. By
  the right handed version of \cite[Lemma~5.6]{M:gendeg},
  \[ L_{k,l}(Q)T_{a,b}
         =T_{a,b}\Big(L_{b+k,b+l}(Q)+\sum_{m=b+k}^{b+l}h_mL_m\Big),
  \]
  for some $h_m\in\LH_m$. Therefore, there exist elements $h_{m,i,t}\in\LH_m$
  such that
  \[\LL[s,t]_{k,l}T_{a,b}=T_{a,b}\prod_{i=1}^d\prod_{u=s}^t
      \Big(L_{b+k,b+l}(\eps^uQ_i)+\sum_{m=b+k}^{b+l}
      h_{m,i,u}L_m\Big).\]
  Collecting the terms in the product, we obtain
  $\LL[s,t]_{b+k,b+l}$, as the leading term, plus a linear combination of
  terms which are products of $d(t-s+1)$ elements, each of which is equal
  to either $L_{b+k,b+l}(\eps^uQ_i)$ or $h_{m,i,u}L_m$, for some $m,i,u$ as
  above. Expand the factors $L_{b+k,b+l}(\eps^uQ_i)$ into a sum of monomials in
  $L_{b+k},\dots,L_{b+l}$ and consider the resulting linear combination of
  products of these summands with the terms $h_{m,i,u}L_m$ above. Fix one
  of these products of $d(t-s+1)$ terms, say $X$, and let~$m$ be maximal
  such that $L_m$ appears in~$X$. By assumption the rightmost $L_m$ which
  appears in $X$ cannot have both $T_m$ and $T_{m-1}$ to its right, so using
  Lemma~\ref{JM properties} we can rewrite~$X$ as a linear combination of
  terms of the form $h_{X,e}L_m^e$, where $1\le e\le d(t-s+1)$ and
  $h_{X,e}\in\LH_m$. Note that when we rewrite $X$ in
  this form some of the $L_{m'}$, with $m'<m$, are changed into $L_m$ when
  we move them to the right.  However, $T_m$ never appears to the right
  of these newly created $L_m$. The final exponent of~$L_m$ is at
  most $d(t-s+1)$ because no factor can increase the exponent of~$L_m$ by
  more than one. The result follows.
\end{proof}

\begin{Lemma}\label{vb-II}
  Suppose that $\b\in\Comp$. Then
  \[v_\b^+=T_{\bp}\Big(
  \LL[1]_{\b_1^1+1,n}\LL[2]_{\b_1^2+1,n}\dots\LL[p-1]_{\b_1^{p-1}+1,n}
          +\sum_{l=1}^{p-1}\sum_{m=\b_1^{\,l}+1}^{\b_1^{l+1}}\sum_{e=1}^{dl}
             h_{l,m,e}L_m^e\Big)\]
  for some $h_{l,m,e}\in\LH_m$.
\end{Lemma}

\begin{proof} Recall that
$ v_\b^+=\LL[1,p-1]_{1,b_p}T_{b_p,\b_1^{p-1}}
      \LL[1,p-2]_{1,b_{p-1}}T_{b_{p-1},\b_1^{p-2}}\dots
      \LL[1,1]_{1,b_2}T_{b_2,\b_1^1}$.
To prove the lemma let $v_{\b,p}^+=1$ and set
$v_{\b,k}^+=v_{\b,k+1}^+ \LL[1,k]_{1,b_{k+1}}T_{b_{k+1},\b_1^k}$, for
$1\le k<p$. We claim that if $1\le k\le p$ then
\[v_{\b,k}^+=T_{(b_p,\dots,b_{k+1},\b_1^k)}
      \Big(\LL[1,p-1]_{\b_1^{p-1}+1,\b_1^p}\dots\LL[1,k]_{\b_1^k+1,\b_1^{k+1}}
      +\sum_{l=k}^{p-1}\sum_{m=\b_1^{\,l}+1}^{\b_1^{l+1}}\sum_{e=1}^{dl}
        h'_{l,m,e}L_m^e\Big),\]
for some $h'_{l,m,e}\in\LH_m$. When $k=p$ there is nothing to prove, so we
may assume that $1\le k<p$ and, by induction, that the claim is true for
$v_{\b,k+1}^+$. Therefore, by Lemma~\ref{bump},
\[\begin{array}{l@{}l}
  v_{\b,k}^+&=T_{(b_p,\dots,b_{k+2},\b_1^{k+1})}
  \Big(\LL[1,p-1]_{\b_1^{p-1}+1,\b_1^p}\dots\LL[1,k+1]_{\b_1^{k+1}+1,\b_1^{k+2}}
      +\Sum_{l=k+1}^{p-1}\sum_{m=\b_1^{\,l}+1}^{\b_1^{l+1}}\sum_{e=1}^{dl}
        h'_{l,m,e}L_m^e\Big)\\
  &\qquad\times\quad
  T_{b_{k+1},\b_1^k}\Big(\LL[1,k]_{\b_1^k+1,\b_1^{k+1}}
  +\Sum_{m=\b_1^k+1}^{\b_l^{k+1}}\sum_{e=1}^{dk}h_{m,e}''L_m^e\Big),
\end{array}\]
for some $h_{l,m,e}',h_{m,e}''\in\LH_m$. Now, by Lemma~\ref{JM properties},
$T_{b_{k+1},\b_1^{k}}$ commutes with $L_m$ whenever $m>\b_1^{k+1}$. Moreover,
if $m>\b_1^{k+1}$ then
\[h'_{l,m,e}L_m^e T_{b_{k+1},\b_1^{k}}=h'_{l,m,e} T_{b_{k+1},\b_1^{k}} L_m^e
  =T_{b_{k+1},\b_1^{k}}h''_{l,m,e}L_m^e ,\]
where
$h''_{l,m,e}=T_{b_{k+1},\b_1^{k}}^{-1}h'_{l,m,e} T_{b_{k+1},\b_1^{k}}$.
It is easy to check that $h''_{l,m,e}\in\LH_m$.
Next note that $T_{(b_p,\dots,b_{k+2},\b_1^{k+1})}T_{b_{k+1},\b_1^k}
            =T_{(b_p,\dots,b_{k+1},\b_1^k)}$.
Therefore, $v_{\b,k}^+$ is equal to
\[\begin{array}{l@{}l}
  v_{\b,k}^+&=T_{(b_p,\dots,b_{k+1},\b_1^k)}
  \Big(\LL[1,p-1]_{\b_1^{p-1}+1,\b_1^p}\dots\LL[1,k+1]_{\b_1^{k+1}+1,\b_1^{k+2}}
      +\Sum_{l=k+1}^{p-1}\sum_{m=\b_1^{\,l}+1}^{\b_1^{l+1}}\sum_{e=1}^{dl}
        h''_{l,m,e}L_m^e\Big)\\
  &\qquad\times\quad\Big(\LL[1,k]_{\b_1^k+1,\b_1^{k+1}}
  +\Sum_{m=\b_1^k+1}^{\b_1^{k+1}}\sum_{e=1}^{dk}h_{m,e}''L_m^e\Big).
\end{array}\]
To complete the proof of the claim observe that
\[\LL[1,p-1]_{\b_1^{p-1}+1,\b_1^p}\dots\LL[1,k+1]_{\b_1^{k+1}+1,\b_1^{k+2}}
    =\LL[1]_{\b_1^{k+1}+1,n}\dots\LL[k+1]_{\b_1^{k+1}+1,n}
    \LL[k+2]_{\b_1^{k+2}+1,n}\dots\LL[p-1]_{\b_1^{p-1}+1,n}.\]
Therefore, when we write this element as a polynomial in
$L_{\b_1^{k+1}+1},\dots,L_n$, the exponent of $L_m$ is at most $dl$ if
$\b_1^{\,l}<m\le\b_1^{l+1}$ for some $k+1\leq l\leq p-1$.  Using this
observation it is now a straightforward exercise to expand the formula
for $v_{\b,k}^+$ above and show that $v_{\b,k}^+$ can be written in
the required form, thus completing the proof of the claim.

Returning to the proof of the lemma, observe that $v_\b^+=v_{\b,1}^+$ and
that the statement of the lemma is the special case of the claim
above when $k=1$ (and setting $k=0$ in the last displayed equation).
\end{proof}




\end{document}